\documentclass[12pt]{amsart}
\usepackage{amsmath}
\usepackage{amssymb}
\usepackage{amsthm}
\usepackage{amscd}
\usepackage{amsxtra}
\usepackage{verbatim}
\usepackage{xcolor}
\usepackage{color}
\usepackage{enumerate}
\usepackage{mathrsfs}

\usepackage{array,arydshln}

\allowdisplaybreaks

\numberwithin{equation}{section}

\definecolor{dblue}{rgb}{0,0,0.45}
\definecolor{red}{rgb}{0.7,0,0}

 \RequirePackage{geometry}
\newtheorem{theorem}{Theorem}[section]
\newtheorem{lemma}[theorem]{Lemma}
\newtheorem*{lemma*}{Lemma}
\newtheorem{corollary}[theorem]{Corollary}
\newtheorem{proposition}[theorem]{Proposition}

\theoremstyle{definition}

\newtheorem{remark}[theorem]{Remark}

\newtheorem{example}[theorem]{Example}

\theoremstyle{remark}

\newcommand{\R}{{\mathbb R}}

\newcommand{\cA}{{\mathcal A}}
\newcommand{\cC}{{\mathcal C}}
\newcommand{\cD}{{\mathcal D}}

\newcommand{\cG}{{\mathcal G}}
\newcommand{\cH}{{\mathcal H}}
\newcommand{\cI}{{\mathcal I}}

\newcommand{\cK}{{\mathcal K}}
\newcommand{\cL}{{\mathcal L}}

\newcommand{\cS}{{\mathcal S}}

\newcommand{\cY}{{\mathcal Y}}
\newcommand{\cZ}{{\mathcal Z}}

\newcommand{\p}{\partial}

\newcommand{\la}{\langle}
\newcommand{\ra}{\rangle}
\newcommand{\nn}{\nonumber}
\newcommand{\ve}{\varepsilon}
\newcommand{\kyosu}{\sqrt{-1}\,}
\begin{document}

\title{Heat trace asymptotics 
on 
equiregular 
sub-Riemannian manifolds
}
\author{  Yuzuru Inahama and Setsuo Taniguchi
}
\maketitle

\begin{abstract}
We study a ``div-grad type" sub-Laplacian  with respect to a smooth measure
and its associated heat semigroup 
on a compact equiregular sub-Riemannian manifold.
We prove a short time asymptotic expansion of the heat trace
up to any order.
Our main result holds true for 
any smooth measure on the manifold, 
but it has a spectral geometric meaning 
when Popp's measure is considered.
Our proof is probabilistic.
In particular, we use S. Watanabe's distributional Malliavin calculus.
\\
\\
Mathematics Subject Classification (2010):~ 53C17, 60H07, 58J65, 35K08, 41A60.
\\
Keywords:~ sub-Riemannian geometry,
heat kernel, stochastic differential equation,
 Malliavin calculus,  asymptotic expansion.
\end{abstract}


\section{Introduction and main result}
In Introduction of his textbook on sub-Riemannian geometry \cite{mo}, 
R. Montgomery emphasized the importance of 
spectral geometric problems in sub-Riemannian geometry
by asking ``Can you 'hear' the sub-Riemannian metric 
from the spectrum of its sublaplacian?"
(Of course, this is a slight modification of M. Kac's renowned question.)
In the same paragraph, he also mentioned Malliavin calculus, 
which is a powerful
 infinite-dimensional functional analytic method
for studying stochastic differential equations (SDEs)
under the H\"ormander condition on the coefficient vector fields.

However, there is no canonical choice of measure
on a general sub-Riemannian manifold
and hence no canonical choice of sub-Laplacian.
Therefore, in order to pose spectral geometric questions, 
one should consider 
a subclass of sub-Riemannian manifolds.
In this regard, the class of 
equiregular sub-Riemannian manifolds seems suitable for the following reason.
As Montgomery himself proved in \cite[Section 10.6]{mo}, 
there exists a canonical smooth volume called Popp's measure
 on an equiregular sub-Riemannian manifold.
Popp's measure is determined by the sub-Riemannian metric only.

In the present paper, we contribute to this topic
by proving a short time asymptotic expansion of 
the heat trace up to an arbitrary order on a
compact  equiregular sub-Riemannian manifold.
Our main tool is Watanabe's distributional Malliavin calculus.

To state our main result, 
we start by recalling 
the definition of an equiregular sub-Riemannian manifold.
Note that in many literatures 
an {\it equiregular} sub-Riemannian manifold is simply called {\it regular}.

We say that 
 $(M,\cD,g)$ is a sub-Riemannian manifold if   
(i)~$M$ is a connected, smooth manifold of
dimension $d$, 
(ii)~$\cD\subset TM$, $TM$ being the tangent bundle of $M$,
is a smooth distribution of constant rank $n~(1 \le n \le d)$
which satisfies the H\"ormander condition at every $x \in M$
and
(iii)~$g=(g_x)_{x\in M}$,
where each $g_x$ is an inner product on the fiber $\cD_x$,
and $x\mapsto g_x$ is smooth as a function of $x$.
(When there is no risk of confusion, we simply say that 
$M$ is a sub-Riemannian manifold.)

The precise statement of 
 the H\"ormander condition on $\cD$ at $x \in M$ is as follows:
Define $\cD_0(x)=\{0\}$,
$\cD_1 (x) =\cD (x)$ and
\[
\cD_{k} (x)
=
\mbox{linear span of } 
\Bigl\{  
 \underbrace{
[ [[A_1, A_2], \ldots  ], A_l]
 }_{(l-1) {\rm brackets}}
 (x)
 \, \Big\vert\,
 1 \le l \le k,  \,  A_{1}, \ldots, A_l \in C^\infty(M;\cD)
 \Bigr\}
\]
for $k \ge 2$.
Here, $C^\infty(M;\cD)$ stands for the $C^\infty$-module
 of smooth sections of $\cD$ over $M$.
We say that 
$\cD$ satisfies the H\"ormander condition at $x$ 
if there exists $N =N(x)$ such that 
$\cD_N (x)= T_x M$.

A sub-Riemannian manifold
 $(M,\cD,g)$ is said to be equiregular 
 if $\dim \cD_k(x)$ is constant in $x \in M$ for all $k \ge 1$.
The smallest constant $N_0$ such that 
$\cD_{N_0} (x)= T_x M$ is called the {\it step} of the H\"ormander condition.
In this case, 
$\nu :=\sum_{k=1}^{N_0} k(\dim \cD_k(x)-\dim\cD_{k-1}(x))$,
is also constant in $x$ and equals the Hausdorff dimension 
of $M$ equipped with the usual sub-Riemannian distance.

Now we define a ``div-grad type" sub-Laplacian 
on a  sub-Riemannian manifold $M$.
Let $\mu$ be a smooth volume on $M$, that is, 
$\mu$ is a measure on $M$ whose restriction to every local coordinate chart 
is written as a strictly positive smooth density function 
times the Lebesgue measure on the chart.
In the equiregular case, the most important example of smooth volume
is Popp's measure introduced in Section 10.6, \cite{mo} (see also \cite{br})
since Popp's measure is determined solely
by the equiregular sub-Riemannian  structure.

We study the second-order differential operator of the form
$\Delta = \text{\rm div}_\mu \nabla^{\cD}$,
where $\nabla^{\cD}$ is the horizontal gradient in the direction of $\cD$
and $\text{\rm div}_\mu$ is the divergence with respect to $\mu$.
(In our convention, $\Delta$ is a non-positive operator.)
By the way it is defined, $\Delta$ with its domain being $C_0^{\infty} (M)$
is clearly symmetric on $L^2 (\mu)$.
If $M$ is compact, then $\Delta$ is known to be 
essentially self-adjoint on $C^{\infty} (M)$ and
$e^{t \Delta /2}$ is of trace class for every $t >0$, 
where $(e^{t \Delta /2})_{t \ge 0}$ is the heat semigroup
 associated with $ \Delta /2$.

\medskip

Now we are in a position to state our main result in this paper. 
The proof of this theorem is immediate from Theorem \ref{tm.subrie}.
As we have already mentioned, 
it has a spectral geometric meaning when $\mu$ is  Popp's measure.
\begin{theorem}\label{tm.main}
Let $M$ be a compact equiregular sub-Riemannian manifold
of Hausdorff dimension $\nu$
and let $\mu$ be a smooth volume on $M$.
Then, we have the following  
asymptotic expansion of the heat trace:
\begin{equation}\label{asy.trc}
{\rm Trace} (e^{t \Delta /2})  \sim 
\frac{1}{t^{\nu/2}} 
(c_0+ c_1t + c_2 t^2+\cdots)
\qquad
\mbox{as $t \searrow 0$}
\end{equation}
for certain constants $c_0 >0$ and  $c_1, c_2, \ldots \in {\mathbb R}$.
\end{theorem}

\medskip

Since the asymptotic expansion in Theorem \ref{tm.main} is 
up to an arbitrary order, 
we can prove
 meromorphic prolongation
of the spectral zeta function associated with $\Delta$ by a standard argument.
Denote by 
$0 =
\lambda_0 < \lambda_1 \le \lambda_2 \le \cdots$
be all the eigenvalues of $-\Delta$ in increasing order
with the multiplicities being counted and set 
\[
\zeta_{\Delta}  (s) = 
\sum_{i=1}^{\infty} 
 \lambda_i^{-s}
\qquad
(s \in {\mathbb C}, \,  \Re s > \frac{\nu}{2}).
\]
By the Tauberian theorem,
the series on the right-hand side absolutely converges 
and defines a holomorphic function 
on $\{ s \in {\mathbb C} \mid \Re s > \nu /2 \}$.

\begin{corollary}\label{cor.z}
Let assumptions be the same as in Theorem \ref{tm.main}.
Then, $\zeta_{\Delta}$ admits a meromorphic prolongation
to the whole complex plane ${\mathbb C}$.
\end{corollary}

\medskip

To the best of our knowledge, Theorem \ref{tm.main} and
Corollary \ref{cor.z} seem new for a general compact
equiregular sub-Riemannian manifold.
It should be noted, however, that the leading term 
of the asymptotics \eqref{asy.trc} is already known.
(See \cite{me, hk} for example. 
No explicit value of $c_0$ is known in general.)
For some concrete examples or 
relatively small classes of compact equiregular 
sub-Riemannian manifolds, 
the full asymptotic expansion \eqref{asy.trc} 
or the meromorphic extension 
of the spectral zeta function was proved.
(See \cite{st, bgs, chya, po, bfi1, bfi2} and references therein.)
Most of such classes are subclasses of  
step-two or  corank-one sub-Riemannian manifolds.

Our proof of Theorem \ref{tm.main} is based on Takanobu's result \cite{tak} 
on the short time asymptotic expansion of hypoelliptic heat kernels 
on the diagonal.
A preceding work by Ben Arous \cite{bena} should also be mentioned.
Both of   \cite{bena, tak} are probabilistic
and formulated on ${\mathbb R}^d$. 
Compared to \cite{bena},
 \cite{tak} has the following features:
(i) The SDE has a drift term. 
Unlike most of the problems for SDEs, 
a drift term makes this kind of asymptotic quite complicated.
(ii) The H\"ormander condition is assumed only at the starting point.
(iii) The asymptotics expansion takes place 
at the level of Watanabe distributions, 
which is stronger than an asymptotic expansion of the heat kernel.
On the other hand,  \cite{bena} proves 
a uniform asymptotic expansion of the heat kernel
with respect to the starting point 
as it varies in a compact ``equiregular"  subset of ${\mathbb R}^d$.
(It seems that Takanobu's motivation was to investigate 
what happens when the H\"ormander condition is not so nice.
He discovered that a pathological phenomenon happens
when the condition is weak enough.
Later, this phenomenon was further studied by 
Ben Arous and L\'eandre \cite{bl1, bl2}.)

We first prove a uniform asymptotic expansion
at the level of Watanabe distributions 
under the equiregular H\"ormander condition 
for a driftless SDE on ${\mathbb R}^d$ (Theorem \ref{tm.R^d}).  
Although it is similar to the 
main results in these papers, this theorem,
precisely speaking,  is not included in \cite{bena, tak}.
We basically follow the argument in \cite{tak} to prove this theorem,
but we believe that our proof is simpler and more readable
for reasons that will be specified later (Remark \ref{re.hikaku}).

Thanks to recent developments of the stochastic parallel 
transport on sub-Riemannian manifolds, 
we can construct the $\Delta/2$-diffusion process on $M$
as a strong solution to an SDE of Eells-Elworthy-Malliavin type.
Since the solution is non-degenerate in the sense of Malliavin calculus,
a standard localization procedure for heat kernels works.
Thus, our asymptotic problem on $M$ reduces 
to one on ${\mathbb R^d}$.
(The reason why it suffices to consider the driftless case 
in Theorem \ref{tm.R^d} is as follows.
The SDE corresponding to $\Delta/2$ on $M$
and its localized version
have a drift term, 
but it can be dealt with by Girsanov's theorem
fortunately.
Hence, our asymptotic problem reduces to the driftless case.)

In our view, (possible) advantages of 
the probabilistic approach to analytic problems 
on sub-Riemannian manifolds are as follows.
(For more information on this approach,
see Thalmaier's recent survey \cite{thal}.)
Unlike in the elliptic (i.e., Riemannian) case, 
analytic methods (in particular, the theory of 
pseudo differential operators)
does not work perfectly under a general 
bracket-generating condition
(except for the corank-one or the step-two case).
On the other hand, 
Malliavin calculus works under a general 
bracket-generating condition and 
the step of the condition does not really matter. 
Therefore, there seems to be a good chance 
that probability theory turns out to be more powerful 
than analysis
at least for certain problems in sub-Riemannian geometry.

Merits of using Watanabe's version of Malliavin calculus in 
sub-Riemannian geometry could be as follows.
First, it is probably the most powerful among a few versions
of Malliavin calculus.
In particular, it has a nice asymptotic theory.
Second, it is highly self-contained.
(For example, existence of the heat kernel can be shown
within this theory and the heat kernel is expressed  
by a generalized Feynman-Kac formula. See Section 6.)
This aspect of Watanabe's theory has not been paid 
much attention in the Riemannian case, 
probably because properties of many important objects on Riemannian manifolds were already obtained
by analytic methods and one could just borrow them.
On sub-Riemannian manifolds, however, 
analysis has not been fully developed.
Hence, there is a possibility that the self-containedness 
will turn out to be of great advantage in the future development of this research topic.

The organization of this paper is as follows.
In Section 2, a very brief review of 
Watanabe's distributional Malliavin calculus is given.
In Section 3, the free nilpotent groups/algebras 
and canonical diffusion processes on them are introduced.
These processes  approximate the diffusion process we 
actually investigate. 
In Section 4, we summarize many non-trivial
 properties of 
vector fields on ${\mathbb R}^d$ that satisfy the (equiregular)
H\"ormander condition.
The main purpose 
of Section 5 is to present and prove our key theorem on ${\mathbb R}^d$
(Theorem \ref{tm.R^d})
by using Malliavin calculus.
This theorem is a ``uniform version" of the main result in \cite{tak}
and can also be considered as a 
``Watanabe distribution version" of the main result in \cite{bena}.
In Section 6, we prove our main theorem (Theorem \ref{tm.main})
by showing a uniform asymptotic expansion 
of the heat kernel on a compact 
equiregular sub-Riemannian manifold $M$.
By localization and Girsanov's theorem, the proof of this fact
is reduced to that of the Euclidean case (Theorem \ref{tm.R^d}).
In Section 7, we give explicit expressions of 
the leading constants of the asymptotic expansions 
for some special examples of sub-Riemannian manifold.

In a paper of this kind, the term {\it distribution} may have 
three different meanings:
{\rm (i)}~A subbundle of the tangent bundle of a manifold 
(e.g. Martinet distribution, contact distribution).
{\rm (ii)}~A generalized function 
(e.g. Schwartz distribution, Watanabe distribution).
{\rm (iii)}~
A probability measure, in particular, the law of a random variable
(e.g. normal distribution,  chi-squared distribution).
We use this term only for {\rm (i)} and {\rm (ii)} in this paper.
Since {\rm (i)} and {\rm (ii)} are very different, we believe 
there is no risk of confusion.


\section{Preliminaries from Malliavin calculus}
Let $W =C_0 ([0,1], {\mathbb R}^n)$ be the set
of continuous functions from $[0,1]$ to ${\mathbb R}^n$
which start at $0$. This is equipped with the usual sup-norm.
The $n$-dimensional Wiener measure on $W$ is denoted by ${\mathbb P}$.
We denote by 
\[
H
=
\Bigl\{  h \in W
~\Big\vert~
 \mbox{absolutely continuous and } 
\|h\|^2_{H} :=\int_0^1 |h^{\prime}_s|^2 ds <\infty
\Bigr\}
\]
the Cameron-Martin subspace of $W$.
The triple $(W, H, {\mathbb P})$ is called the classical Wiener space.
The canonical realization on $W$
of $n$-dimensional
Brownian motion is denoted by 
$(w_t)_{0 \le t \le 1}=(w_t^1, \ldots, w_t^n)_{0 \le t \le 1}$.

We recall Watanabe's theory of 
generalized Wiener functionals (i.e., Watanabe distributions) in Malliavin calculus.
Most of the contents and the notations
in this section are contained in Sections V.8--V.10, Ikeda and Watanabe \cite{iwbk}
with trivial modifications.
We also refer to
Shigekawa \cite{sh}, Nualart \cite{nu},  Hu \cite{hu}
and Matsumoto and Taniguchi \cite{mt}.

The following are of particular importance in this paper:

\medskip

{\bf (a)}~ Basics of Sobolev spaces:
We denote by ${\mathbb D}_{p,r} ({\mathcal X})$ 
 the Sobolev space of ${\mathcal X}$-valued 
(generalized) Wiener functionals, 
where $p \in (1, \infty)$, $r \in {\mathbb R}$, and ${\mathcal X}$ is a real separable Hilbert space.
As usual, we will use the spaces 
${\mathbb D}_{\infty} ({\mathcal X})= \cap_{k=1 }^{\infty} \cap_{1<p<\infty} {\mathbb D}_{p,k} ({\mathcal X})$, 
$\tilde{{\mathbb D}}_{\infty} ({\mathcal X}) 
= \cap_{k=1 }^{\infty} \cup_{1<p<\infty}  {\mathbb D}_{p,k} ({\mathcal X})$ of test functionals 
and  the spaces ${\mathbb D}_{-\infty} ({\mathcal X}) = \cup_{k=1 }^{\infty} \cup_{1<p<\infty} {\mathbb D}_{p,-k} ({\mathcal X})$, 
$\tilde{{\mathbb D}}_{-\infty} ({\mathcal X}) = \cup_{k=1 }^{\infty} \cap_{1<p<\infty} {\mathbb D}_{p,-k} ({\mathcal X})$ of 
 Watanabe distributions as in \cite{iwbk}.
When ${\mathcal X} ={\mathbb R}$, we simply write ${\mathbb D}_{p, r}$, etc.
The ${\mathbb D}_{p, r} ({\mathcal X})$-norm is denoted by $\| \,\cdot\, \|_{p,r}$.
The precise definition of an asymptotic expansion 
up to any order in these spaces can be 
found in \cite[Section V-9]{iwbk}.
We denote by $D$  the gradient operator ($H$-derivative) 
and by $L = -D^* D$ the Ornstein-Uhlenbeck operator.

\medskip

{\bf (b)}~ Meyer's equivalence of Sobolev norms:
See \cite[Theorem 8.4]{iwbk}. 
A stronger version can be found in \cite[Theorem 4.6]{sh},  
\cite[Theorem 1.5.1]{nu} or Bogachev \cite[Theorem 5.7.1]{bog}.
It states that the Sobolev norms 
$\| F \|_{p,k} = \| (I -L)^{k/2} F \|_{L^p}$ and 
$\|  F \|_{L^p} + \| D^k F \|_{L^p}$ are equivalent for every $k \in {\mathbb N}$
and $1< p <\infty$.

\medskip

{\bf (c)}~Watanabe's pullback: 
For 
$F =(F^1, \ldots, F^m) \in {\mathbb D}_{\infty} ({\mathbb R}^m)$, 
we denote by $\sigma [F ](w) = \sigma_F (w) $
the Malliavin covariance matrix of $F$,
whose $(i,j)$-component is given by 
$\sigma_F^{ij} (w) =  \la DF^i (w),DF^j (w)\ra_{H^*}$.
Now we assume that $F$ is non-degenerate in the sense of Malliavin,
that is, 
$(\det \sigma [F])^{-1} \in L^p$ for every $1 <p<\infty$.

Then,  pullback 
$T \circ F =T(F)\in \tilde{\mathbb D}_{-\infty}$ 
of a tempered Schwartz distribution $T \in {\mathcal S}^{\prime}({\mathbb R}^m)$
on ${\mathbb R}^m$
by a non-degenerate Wiener functional $F \in {\mathbb D}_{\infty} ({\mathbb R}^m)$
is well-defined and has nice properties. 
(See \cite[Section V-9]{iwbk}.)
The key to justify this pullback is an integration by parts formula 
in the sense of Mallavin calculus.
(Its generalization is given in Item {\bf (d)} below.)

\medskip

{\bf (d)}~A generalized version of the integration by parts formula in the sense 
of Malliavin calculus
 for Watanabe distribution,
which is given as follows (see \cite[p. 377]{iwbk}):

For a non-degenerate Wiener functional
$F =(F^1, \ldots, F^m) \in {\mathbb D}_{\infty} ({\mathbb R}^m)$, 
we denote by $\gamma^{ij}_F (w)$ the $(i,j)$-component of the inverse matrix $\sigma^{-1}_F$.
Note that $\sigma^{ij}_F \in {\mathbb D}_{\infty} $ and
$D \gamma^{ij}_F =- \sum_{k,l} \gamma^{ik}_F ( D\sigma^{kl}_F ) \gamma^{lj}_F $.
Hence, derivatives of $\gamma^{ij}_F$ can be written in terms of
$\gamma^{ij}_F$'s and the derivatives of $\sigma^{ij}_F$'s.
Suppose 
$G \in {\mathbb D}_{\infty}$ and $T \in {\mathcal S}^{\prime} ({\mathbb R}^m)$.
Then, the following integration by parts holds:
\begin{align}
{\mathbb E} \bigl[
(\partial_i T \circ F )  G 
\bigr]
=
{\mathbb E} \bigl[
(T \circ F )  \Phi_i (\, \cdot\, ;G)
\bigr]
\qquad
(1 \le i \le m),
\label{ipb1.eq}
\end{align}
where ${\mathbb E}$ stands for the generalized expectation 
and $\Phi_i (w ;G) \in  {\mathbb D}_{\infty}$ is given by 
\begin{equation}\label{eq.ibpPHI}
\Phi_i (w ;G) =
\sum_{j=1}^m  D^* \Bigl(   \gamma^{ij }_F(w)  G (w) DF^j(w)
\Bigr).
\end{equation}

\medskip

{\bf (e)}~
Watanabe's asymptotic expansion theorem is a key theorem in his distributional Malliavin calculus, which can be found
 in \cite[Theorem 9.4, pp. 387-388]{iwbk}.
It can be summarized as follows:

Let $F(\ve , \,\cdot\,) \in {\mathbb D}_{\infty} ({\mathbb R}^m)$ for $0<\ve \le 1$.
We say $F(\ve , \,\cdot\,)$ is uniformly non-degenerate in the sense of Malliavin if 
\[
\sup_{0< \ve \le 1}  \|  
 \bigl( \det \sigma [DF (\ve, \,\cdot\,) ]
\bigr)^{-1}
 \|_{L^p}  <\infty
\qquad
\mbox{ for every $1<p<\infty$.}
\]

Let us assume that 
$F(\ve, \,\cdot\,) \in {\mathbb D}_{\infty} ({\mathbb R}^m) ~(0 < \ve \le 1)$ 
 is uniformly non-degenerate in the sense of Malliavin
 and 
admits the following asymptotic expansion:
\[
F(\ve, \,\cdot\,) \sim  f_0 + \ve^1  f_1 +\ve^2 f_2    + \cdots   
\qquad
\mbox{in ${\mathbb D}_{\infty} ({\mathbb R}^m)$ as $\ve \searrow 0$}
\]
with $f_j \in {\mathbb D}_{\infty} ({\mathbb R}^m)$ for all $j \in {\mathbb N}$.  
Then, for any $T \in {\mathcal S}' ({\mathbb R}^m)$,
$T \circ F(\ve , w)$ admits the following asymptotic expansion:
\begin{equation}\label{asymp.WAT}
T \circ F(\ve, \,\cdot\,) \sim  \psi_0 + \ve^1 \psi_1 +\ve^2 \psi_2 + \cdots   
\qquad
\mbox{in $\tilde{{\mathbb D}}_{-\infty}$ as $\ve \searrow 0$,}
\end{equation}
where $\psi_j \in \tilde{{\mathbb D}}_{- \infty}$ 
is given by the formal Taylor expansion.
(For example, $\psi_0 = T(f_0)$ and $\psi_{1} =\sum_{i=1}^m f_1^i 
\cdot (\partial T / \partial x^i )(f_0)$, etc.)


\section{Free nilpotent Lie group and lift of Brownian motion}
In this section we introduce 
the free nilpotent Lie groups 
and algebras, following Friz-Victoir \cite[Chapter 7]{fv}. 
The set of iterated integrals (i.e., multiple Wiener integrals)
of Brownian motion becomes
a left-invariant hypoelliptic diffusion process on this Lie group.
See \cite{ya} for example.
(According to \cite{tak2},
a similar fact also holds for the iterated integrals
of a continuous local semimartingale.)
The logarithm of this process will play a major role 
since it approximates the diffusion process under investigation
 in short times.

Let $N \ge 1$, which  is the step of nilpotency. 
We denote by $T^N ({\mathbb R}^n) =\oplus_{i=0}^N ( {\mathbb R}^n)^{\otimes i}$ 
the truncated tensor algebra of step $N$,
where $( {\mathbb R}^n )^{\otimes 0} ={\mathbb R}$ by convention.
The dilation by $c \in {\mathbb R}$ is denoted by $\Delta_c^N$,
that is, $\Delta_c^N (1, a_1, \ldots, a_N) = (1, c^1 a_1, \ldots, c^N a_N)$.
For $N \le M$, 
$\Pi_{N}^M$ denotes the canonical projection from
$T^M ({\mathbb R}^n) $ onto $T^N ({\mathbb R}^n)$.

We write ${\frak t}^N ({\mathbb R}^n)=\{ 0+ A 
\mid  0 \in {\mathbb R}, \,
A \in  \oplus_{i=1}^N ( {\mathbb R}^n)^{\otimes i} \}$.
This is a Lie algebra under the bracket $[A, B] := A \otimes B - B\otimes A$.
Then, 
$1+{\frak t}^N ({\mathbb R}^n)
=\{ 1+ A 
\mid 1 \in {\mathbb R}, \, A \in  \oplus_{i=1}^N ( {\mathbb R}^n)^{\otimes i} \} 
= \exp ({\frak t}^N ({\mathbb R}^n))$
is a Lie group. 
The unit element is denoted by ${\bf 1}$. 
Here, $\exp =\exp_N:{\frak t}^N ({\mathbb R}^n) \to 1+{\frak t}^N ({\mathbb R}^n)$ 
is the exponential map defined in the usual way.
Its inverse is the logarithm map $\log =\log_N$.
By the correspondence 
$1+A \mapsto A \in \oplus_{i=1}^N ( {\mathbb R}^n )^{\otimes i}$,
$1+{\frak t}^N ({\mathbb R}^n)$ is  diffeomorphic to the linear space 
$\oplus_{i=1}^N ( {\mathbb R}^n)^{\otimes i} \cong {\frak t}^N ({\mathbb R}^n)$.
This map gives a (global) chart on this group.

The free nilpotent Lie algebra of step $N$
is denoted by ${\frak g}^N ({\mathbb R}^n)$, which is a 
sub-Lie algebra of ${\frak t}^N ({\mathbb R}^n)$ 
generated by the elements of ${\mathbb R}^n$.
More precisely, 
\[
{\frak g}^N ({\mathbb R}^n) 
:= {\mathbb R}^n \oplus [{\mathbb R}^n, {\mathbb R}^n] \oplus 
\cdots 
 \oplus
 \underbrace{
[[ [{\mathbb R}^n, {\mathbb R}^n], \ldots], 
{\mathbb R}^n
]
 }_{(N-1) {\rm brackets}}.
 \]
The set
$G^N ({\mathbb R}^n) 
   = \exp ({\frak g}^N ({\mathbb R}^d))$ is called the free nilpotent 
   Lie group of step $N$.
   It is a sub-Lie group of  $1+{\frak t}^N ({\mathbb R}^n)$.
Note that  $\log: G^N ({\mathbb R}^n) \to {\frak g}^N ({\mathbb R}^n)$
is a diffeomorphism and its inverse is the exponential map 
$\exp: {\frak g}^N ({\mathbb R}^n) \to G^N ({\mathbb R}^n)$.
Using this diffoemorphism, we can define a new
group product on ${\frak g}^N ({\mathbb R}^n)$ by 
\[
A \times B := \log ( \exp(A) \exp (B) )
\qquad
(A, B \in {\frak g}^N ({\mathbb R}^n)).
\]
Thanks to the Baker-Campbell-Hausdorff formula,
the right-hand side has an explicit expression:
\begin{align*}
 \log ( \exp(A) \exp (B) ) =& A +B +\frac12 [A,B] 
 \\
& \qquad +
 \frac{1}{12} ( [A, [A,B]] +  [B, [B,A]]  ) 
-
 \frac{1}{24}  [B,  [A, [A,B]]]
 +\cdots 
  \end{align*}
Here, terms of degree greater than $N$ should be neglected.
This is in fact a finite sum due to nilpotency
and hence is a well-defined Lie polynomial in $A$ and $B$.

Now we fix some symbols for linear basis on free nilpotent
 Lie algebra and words.
The canonical basis of ${\mathbb R}^n$ is denoted by 
$\{{\bf e}_i \mid 1 \le i \le n\}$.
Set ${\cI} (N) =
\cup_{k=1}^N  \{ (i_1, \ldots, i_k) \mid  1 \le  i_1, \ldots, i_k \le n\} $
for 
$1 \le N \le \infty$.
This is the set of words of $n$ letters with length at most $N$, 
where the length of a word is defined by
$| (i_1, \ldots, i_k) | =:k$.
For $I =   (i_1, \ldots, i_k) \in {\cI} (\infty)$,
we set 
${\bf e}_I = {\bf e}_{i_1} \otimes \cdots \otimes   {\bf e}_{i_k}$.
When $k=1$, we will often write $i_1$ for $(i_1)$.
For $I =   (i_1, \ldots, i_k) \in {\cI} (\infty)$,
we define ${\bf e}_{[I]}$ as follows:
\[
{\bf e}_{ [i_1]} :={\bf e}_{ i_1}, \quad 
{\bf e}_{ [i_1, \ldots, i_k] }:= [ {\bf e}_{ [i_1, \ldots, i_{k-1}]} ,  {\bf e}_{ i_k}] \quad (k \ge 2).
\]
Here and in what follows, we write 
$[i_1, \ldots, i_k]$ for $[(i_1, \ldots, i_k)]$ for simplicity of notations.

Let $\cG (N) \subset {\cI} (N) ~(N=1,2,\ldots)$ be such that 
$\cG (N) \subset \cG (N+1)$ for all $N \ge 1$
and
$\{{\bf e}_{[I]}  \mid  I \in  \cG (N) \}$ forms a linear basis of 
${\frak g}^N ({\mathbb R}^n)$.  
The choice of such
$\cG (N) ~(N=1,2,\ldots)$ is not unique.
We write $\cG (\infty) =\cup_{k=1}^{\infty} \cG (k)$.
(For example, we can take
$\cG (1) =\{ (i) \mid 1 \le i \le n\}$
and 
$\cG (2) = \cG (1)  \cup  \{ (i,j) \mid 1 \le i <j \le n\}$.)

Now we introduce vector fields on the Lie group and the Lie algebra.
Note that 
${\bf e}_i \in {\mathbb R}^n \subset {\frak t}^N ({\mathbb R}^n) 
\cong T_{{\bf 0}} {\frak t}^N ({\mathbb R}^n)$.
Here, since ${\frak t}^N ({\mathbb R}^n)$ is a linear space,
it is identifed with its tangent space at the origin.
Since 
$1+{\frak t}^N ({\mathbb R}^n)$ and $G^N ({\mathbb R}^n)$
are submanifolds of a linear space 
$T^N ({\mathbb R}^n)$,
their tangent space can naturally be identified 
with a linear subspace of $T^N ({\mathbb R}^n)$.
By straightforward computation, 
\[
\exp_* {\bf e}_i :=
\left. \frac{d}{ds}\right|_{s=0} \exp (s {\bf e}_i) 
= {\bf e}_i \in  T_{{\bf 1}} ({\frak g}^N ({\mathbb R}^n) ) 
\cong {\frak g}^N ({\mathbb R}^n).
\]

Let $Q_i^N$ be the unique left-invariant vector field on 
$1+{\frak t}^N ({\mathbb R}^n)$ or on $G^N ({\mathbb R}^n)$ 
such that
$Q_i^N ( {\bf 1})= {\bf e}_i$  ($1 \le i \le n$).
More concretely, 
\[
Q_i^N (A) = \left. \frac{d}{ds}\right|_{s=0} A \otimes \exp (s {\bf e}_i)
\qquad
(A \in 1+{\frak t}^N ({\mathbb R}^n)).
\]
The above limit is taken in $T^N ({\mathbb R}^n)$.
If we choose $\{ {\bf e}_I \mid  I \in {\cI}(N) \}$
as a basis of ${\frak t}^N ({\mathbb R}^n)$,
an element of this linear space can be expressed as
$(y^I)_{I \in {\cI}(N)} \in {\mathbb R}^{{\cI}(N)}$.
In this coordinate we have 
\[
Q_i^N (A)= \frac{\partial }{\partial y^i} + 
\sum_{(j_1, \ldots, j_k) \in {\cI}(N-1)}
y^{(j_1, \ldots, j_k)} 
\frac{\partial }{\partial y^{(j_1, \ldots, j_k, i)} }
\qquad
(A = 1 + \sum_{I \in {\cI}(N)} y^I {\bf e}_I  \in 1+{\frak t}^N ({\mathbb R}^n))
\]
for $N \ge 2$. See \cite[p. 174]{tak}.
As vector fields on $G^N ({\mathbb R}^n)$,
$\{ Q_i^N\}_{1 \le i \le n}$ satisfy H\"ormander's bracket-generating condition 
at ${\bf 1}$ and hence at
every point in $G^N ({\mathbb R}^n)$ by the left-invariance.

Define $\hat{Q}_i^N = \log_* Q_i^N$. 
Then, 
$\{ \hat{Q}_i^N\}_{1 \le i \le n}$ are smooth vector fields on 
${\frak t}^N ({\mathbb R}^n)$
and 
satisfy the H\"ormander condition  as vector fields on ${\frak g}^N ({\mathbb R}^n)$ at
every point in ${\frak g}^N ({\mathbb R}^n)$. 
By way of construction, 
these are left-invariant with respect to the product $\times$.
The Baker-Campbell-Hausdorff formula implies that, if we write 
\[
 \hat{Q}_i^N (A) = \sum_{I \in {\cG}(N)}   
 (\hat{Q}_i^N)^I (A)  \frac{\p}{\p y^I}
 \qquad
(A =  \sum_{I \in {\cG}(N)} y^I {\bf e}_{[I]}  \in {\frak g}^N ({\mathbb R}^n)),
  \]
then the coefficient 
$(\hat{Q}_i^N)^I$ is actually a real-valued polynomial in $(y^I)_{I \in {\cG}(N)}$.

If $N=3$ for example,  we have for 
$A = \sum_{I \in {\cG}(3)}  y^I {\bf e}_{[I]}   \in {\frak g}^3 ({\mathbb R}^n)$ that
\[
\hat{Q}_i^3 (A) = \left. \frac{d}{ds}\right|_{s=0} A \times  (s {\bf e}_i)
=
 {\bf e}_i +\frac12 [A, {\bf e}_i]  + \frac{1}{12} [A, [A, {\bf e}_i] ],
 \]
which is a second order polynomial in  $(y^I)_{I \in {\cG}(3)}$.
Here, the linear space ${\frak g}^3 ({\mathbb R}^n)$
and its tangent space are identified in the usual way.

Consider the following  ODE on $G^N ({\mathbb R}^n)$
driven by an ${\mathbb R}^n$-valued Cameron-Martin path $h \in H$:
\begin{equation}\label{eq.odeG^N}
dy_t^N  = \sum_{i=1}^n Q_i^N ( y^N_t) dh_t^i
\qquad
\mbox{with $y_0^N ={\bf 1}$.}
\end{equation}
It is well-known that
a unique solution of (\ref{eq.odeG^N}) has the following explicit expression 
in the form of iterated path integrals (e.g., \cite[Chapter 7]{fv}): 
\[
y_t^N  = y_t^N (h) =\sum_{I \in {\cI}(N)}  h^I_t {\bf e}_I,
\]
where we set
\[
h_t^{(i_1)} := h_t^{i_1}, \qquad
h_t^{(i_1, \ldots, i_k)} := \int_0^t  h_s^{(i_1, \ldots, i_{k-1})}  dh_s^{i_k} 
\qquad
(k \ge 2).
\]
In rough path theory, 
$y^N$ is called the level $N$ rough path lift of $h$.
By (a trivial modification of) \cite[Theorem 7.30]{fv},
we have $G^N ({\mathbb R}^n) = \{ y^N_T (h) \mid h \in H\}$ for every $T>0$.

The corresponding Stratonovich-type SDE on $G^N ({\mathbb R}^n)$
driven by a $n$-dimensional Brownian motion $w$ is as follows:
\begin{equation}\label{eq.sdeG^N}
dY_t^N  = \sum_{i=1}^n Q_i^N ( Y_t^N) \circ dw_t^i
\qquad
\mbox{with $Y^N_0 ={\bf 1}$.}
\end{equation}
A unique solution of (\ref{eq.sdeG^N}) has the following explicit expression 
in the form of iterated Stratonovich integrals: 
\[
Y_t^N = Y_t^N (w) =\sum_{I \in {\cI}(N)}  w^I_t {\bf e}_I, \qquad 
\mbox{a.s.,}
\]
where we set
\[
w_t^{(i_1)} := w_t^{i_1}, \qquad
w_t^{(i_1, \ldots, i_k)} := \int_0^t  w_s^{(i_1, \ldots, i_{k-1})}  \circ dw_s^{i_k} 
\qquad
(k \ge 2).
\]
In rough path theory, $Y^N$ is called the level $N$ rough path lift of $w$
or Brownian rough path of level $N$.
(In most of the cases in rough path theory, $N=2$.)

Set $U_t^N = \log Y_t^N$.
We can easily see that
 $(U_t^N)$ is a diffusion process on ${\frak g}^N ({\mathbb R}^n)$
which satisfies the following Stratonovich-type SDE:
\begin{equation}\label{eq.sdeg^N}
dU_t^N  = \sum_{i=1}^n \hat{Q}_i^N ( U_t^N) \circ dw_t^i
\qquad
\mbox{with $U_0^N ={\bf 0}$.}
\end{equation}
Note that ({\rm i}) the processes $(\Delta_c^N  U_t^N)$
and 
$(U_{c^2 t}^N)$ have the same law for every $c \in {\mathbb R}$
(i.e., the scaling property)
and
 ({\rm ii}) $U_t^{N, I} (-w) = (-1)^{|I|} U_t^{N, I} (w)$ a.s.
for every $I \in \cG(N)$.
One can show these facts 
by first showing the counterparts for $(Y_t^N)$
and then taking the logarithm. 

Since $\{( U_t^N)_{t \ge 0}\mid N \ge 1\}$ are consistent with the projection system, that is,
$\Pi_{N}^{M} ( U_t^M)= U_t^N$ for $M \ge N$,
we have
$U_t^{N,I} =U_t^{M,I}$ if $|I| \le N \le M$.
Therefore, we may and will simply write $U_t^{I}$ for this object.

\begin{remark}\label{re.pol}
Before we apply Malliavin calculus to (\ref{eq.sdeg^N}),
we make a comment on the regularity of smooth coefficient  vector fields.
A standard assumption requires that 
all the derivatives of the coefficients of 
$(\hat{Q}_i^N)^I$ of order $\ge 1$ be bounded.
(However, this is not satisfied in our case). 

The main reason why this cannot be relaxed so easily 
is  
because a  solution of the SDE may explode in finite time 
without this kind of assumption.
However, if existence of a time-global solution is known for some reason,
then it is enough to assume that 
all the derivatives of the coefficients of  
$(\hat{Q}_i^N)^I$ are of at most polynomial growth.
Then, most of  standard results in Malliavin calculus for SDEs still hold.
(In our present case,  the coefficients of 
$(\hat{Q}_i^N)^I$ are literally  polynomials, as we have seen).

Precisely,  it suffices to check that 
\begin{equation}\label{ineq.ks1}
\sup_{0 \le t \le 1} 
\Bigl( \| U^N_t \|_{L^p} + \| \partial U^N_t \|_{L^p} 
+  \| (\partial U^N_t)^{-1} \|_{L^p}\Bigr) 
< \infty
\qquad
(1<p<\infty).
\end{equation}
Here, $\partial U^N$ is the Jacobian process (at ${\bf 0}$) associated 
with SDE (\ref{eq.sdeg^N}) and 
takes values in ${\rm GL} ({\frak g}^N ({\mathbb R}^n))$.
More explicitly,  if we denote by $U^N (t, A)$ 
the solution of SDE (\ref{eq.sdeg^N})  which starts 
at $A\in {\frak g}^N ({\mathbb R}^n)$,
then $\partial U^N_t := \nabla U^N (t, A) \vert_{A = {\bf 0} }$,
where $\nabla$ is the gradient operator with respect to $A$-variable
on ${\frak g}^N ({\mathbb R}^n)$.

The reason why this is sufficient is as follows:
The higher order $H$-derivatives $D^k U^N~(k=1,2,\ldots)$ 
can be written as a stochastic integral
which only involves $w, U^N, \partial U^N, (\partial U^N)^{-1}$
and 
$D U^N, D^2 U^N,  \ldots, D^{k-1} U^N$.
(See \cite[Section 6.1]{sh} for example.)
Due to this ``triangular structure" of  the integral expression,
verifying (\ref{ineq.ks1}) is enough.

Since $U^N_t = \log Y_t^N$, every component of $U^N_t$
is a polynomial in $w_t^{(i_1, \ldots, i_k)}~(1 \le k \le N)$,
$L^p$-norm of  $U^N_t$ clearly satisfies (\ref{ineq.ks1}).
By the left-invariance, 
we have $U^N (t, A)  =A \times U^N_t$.
Let $\{ {\bf e}_{[I]}  \mid  I \in \cG (N) \}$ be a basis of 
${\frak g}^N ({\mathbb R}^n)$
and arrange them in increasing order of the step number.
From the Baker-Campbell-Haudorff formula and straightforward computation,  
we can see that
$\partial U^N_t$ is represented 
with respect to this basis by an lower triangular matrix 
with all the diagonal entries being $1$.
Other non-zero entries of this matrix  are 
polynomials in $w_t^{(i_1, \ldots, i_k)}~(1 \le k \le N-1)$.
Therefore, $L^p$-norms of 
$\partial U^N_t$ and its inverse satisfy (\ref{ineq.ks1}).
\end{remark}

Let $\sigma [U_1^N] = (\la DU_1^{N,I},  DU_1^{N,J} \ra_H)_{I,J \in \cG (N)}$ 
be the Malliavin covariance matrix of $U_1$
and $\lambda  [U_1^N]$ its smallest eigenvalue.
(This means that ${\frak g}^N ({\mathbb R}^n)$ 
is implicitly equipped with an inner product with respect to which 
$\{ {\bf e}_{[I]} \mid I\in \cG (N) \}$ forms an orthonormal basis.)

\begin{proposition}\label{pr.MalU}
Let the notations be as above.
Then, $\lambda  [U_1^N] >0$  a.s. and 
$\lambda  [U_1^N]^{-1} \in \cap_{1<p <\infty} L^p$.
In particular, $U_1^N$ is non-degenerate in the sense of Malliavin.
\end{proposition}

\begin{proof}
By using a standard stopping time argument, 
we only need information of the coefficient vector fields 
near the starting point ${\bf 0}$.
Then, this problem reduces to the one under the standard 
regularity assumption in 
Malliavin calculus presented in Remark \ref{re.pol} above.
See Kusuoka-Stroock \cite{ks2} for example.
\end{proof}


We need the following estimate of
the exit probability: 
For every $\kappa >0$, there exists positive constants 
$C_{N, \kappa}, \hat{C}_{N, \kappa}$
such that
\begin{equation}\label{ineq.exitU}
{\mathbb P} \bigl(  \sup_{0 \le s \le t} | \Delta^N_\ve U^N_s | >\kappa  \bigr) 
 \le \hat{C}_{N, \kappa} \exp ( - C_{N, \kappa} / 4t\ve^2)
 \qquad
 \mbox{(if $0 < t\ve^2< C_{N, \kappa}$)}
\end{equation}
(\cite[p. 181]{tak}).
This follows form the scaling property for $U^N$ and 
a standard argument for the exit probablity for (local) semimartingales.

\begin{remark}\label{re.pr_consist}
Although $N$ is an arbitrarily fixed number in this section,
all the objects are in fact consistent 
with the system of projections $\{ \Pi_{N}^{M} \}_{M \ge N}$.
For example, 
$\Pi_{N}^{M} (G^M ({\mathbb R}^n)) =G^N ({\mathbb R}^n)$,
$\Pi_{N}^{M} ({\frak g}^M ({\mathbb R}^n)) 
={\frak g}^N ({\mathbb R}^n)$,
$\Pi_{N}^{M} \circ \exp_M = \exp_N \circ \Pi_{N}^{M}$,
$\Pi_{N}^{M} \circ \log_M = \log_N \circ \Pi_{N}^{M}$,
$(\Pi_{N}^{M})_* Q_i^M = Q_i^N$,
$\Pi_{N}^{M} ( U_t^M)= U_t^N$, etc.
The projections of course commute with 
the dilations, too.
This consistency indicates that 
these objects actually live in the projective limit spaces,
but we do not take this viewpoint in this paper.
 \end{remark}

\begin{remark}\label{re.hjk}
The Lie group product $A \times B$ on 
${\frak g}^N ({\mathbb R}^n)$ in \cite{tak}
equals
$B \times A$ in the present paper.
Concerning this, the vector field $\hat{Q}_i^N$ is 
{\it left-invariant} here, not right-invariant as in  \cite{tak}.
This modification 
is only for the aesthetic reason and
of no mathematical importance.
 \end{remark}

\begin{remark}\label{re.abc}
In many literatures $({\frak g}^N ({\mathbb R}^n), \times)$,
instead of $G^N ({\mathbb R}^n)$,
is called the free nilpotent {\it group} of step $N$.
See Cygan \cite{cy} for example.
 \end{remark}

%

\section{Vector fields on ${\mathbb R}^d$}
\label{sec:v.f.on.Rd}
In this section 
we discuss vector fields on  ${\mathbb R}^d$.
We fix notations and recall some basic facts for later use.
There are no new results in this section.
The set of all the vector fields on ${\mathbb R}^d$
is denoted by ${\frak X} ({\mathbb R}^d)$.
An element of ${\mathbb R}^d$
is denoted by $x= (x^1, \ldots, x^d)$ as usual.
The set of all linear mappings from a vector space ${\mathcal X}$ to 
another vector space ${\mathcal Y}$ is denoted by $\cL({\mathcal X}, 
{\mathcal Y})$.

For $V \in {\frak X} ({\mathbb R}^d)$, 
we write $V^i (x) = \la dx^i, V (x)\ra$ and hence
$V (x) =\sum_{i=1}^d V^i (x) \frac{\p}{\p x^i}$.
Note that 
a vector field is always regarded as a first order 
differential operator.
The ${\mathbb R}^d$-valued function 
$(V^1 (x), \ldots,V^d (x))^*$ is denoted by
$(V  {\rm Id})(x)$, where ${\rm Id}$ stands for the identity map of 
${\mathbb R}^d$.
Here and in what follows, the superscript $*$ stands 
for the transpose of a matrix.

For the rest of this section, let $n \ge 1$ and
$V_1, \ldots, V_n \in {\frak X}  ({\mathbb R}^d)$.
For $I =   (i_1, \ldots, i_k) \in {\cI} (\infty)$,
we set 
$V_I = V_{i_1}  V_{i_2}  \cdots   V_{i_k}$,
which is a differential operator of order $n$. 
We also set $V_{[I]} \in {\frak X}  ({\mathbb R}^d)$ as follows:
\[
V_{ [i_1]} =V_{ i_1}, \qquad 
V_{ [i_1, \ldots, i_k] }= 
[ V_{ [i_1, \ldots, i_{k-1}]} ,  V_{ i_k}] \qquad (k \ge 2).
\]
The correspondence ${\bf e}_{[I]} \mapsto V_{[I]}$
naturally extends to a Lie algebra homomorphism
from the free Lie algebra generated by ${\mathbb R}^d$
to $ {\frak X}  ({\mathbb R}^d)$.
In particular, every linear relation for
$\{ {\bf e}_{[I]} \mid  I \in  \cI (N) \}$
still holds for $\{ V_{[I]} \mid  I \in  \cI (N) \}$.

Now we give a simple lemma for later use.
This lemma is essentially implied by 
\cite[Corollary 2.3, Propositions 3.9 and 4.4]{tak}. 
Our proof below is, however, different from the one in  \cite{tak} 
and more straightforward and algebraic.
\begin{lemma}\label{lm.YtoU}
Let $N \ge 1$.
Then, for every $x \in {\mathbb R}^d$ and 
$u = \sum_{ J \in \cG (N)} u^J {\bf e}_{[J]} \in {\frak g}^N ({\mathbb R}^n)$,
\begin{eqnarray}
\label{eq.YtoU}
\sum_{I \in \cI (N)}  (V_{I} {\rm Id}) (x) 
 \pi_I \bigl(
 \exp (u)
\bigr)
=
\sum_{k=1}^N \frac{1}{k!}
\sum_{ |J_1| + \cdots + |J_k| \le N}  
( V_{[J_1]} \cdots  V_{[J_k]}  {\rm Id} )(x) u^{J_1} \cdots u^{J_k}
\end{eqnarray}
(the  summation  runs over 
 all $(J_1, \ldots, J_k) \in \cG (N)^k$ such that $|J_1| + \cdots + |J_k| \le N$).
Here, $\pi_{I}$ is the linear functional on $T^N ({\mathbb R}^n)$ 
that picks up the coefficient of ${\bf e}_I$ and 
$\exp$ is the exponential map 
from ${\frak g}^N ({\mathbb R}^n)$ to $G^N ({\mathbb R}^n)$.
\end{lemma}
\begin{proof}
Let $\alpha_J^K \in {\mathbb R}$ be such that
\[
{\bf e}_{[J]} = \sum_{K\in \cI (N)} \alpha_J^K {\bf e}_{K}
\qquad
(J \in \cG (N)).
\]
Note that $\alpha_J^K =0$ if $|J| \neq |K|$.
Then, it holds that
\begin{equation}\label{eq.V[]}
V_{[J]} = \sum_{K \in \cI (N)} \alpha_J^K   V_{K}
\qquad
(J \in \cG (N)).
\end{equation}
The left-hand side of (\ref{eq.YtoU}) is equal to 
\[
\sum_{I \in \cI (N)}  (V_{I} {\rm Id}) (x) 
 \pi_I \Bigl(
 \sum_{k=0}^N \frac{1}{k!}
\bigl( 
\sum_{ J \in \cG (N)} u^J   
\sum_{K \in \cI (N)} \alpha_J^K {\bf e}_{K}
\bigr)^{\otimes k}
\Bigr),
\]
which is a polynomial in $u^I$'s.
Let us compute its $k$th order term.
For $k=0$, it vanishes since $|I| \ge 1$.
For $k=1$, we see from \eqref{eq.V[]} that
\begin{align*}
\sum_{I \in \cI (N)}  (V_{I} {\rm Id}) (x) 
 \pi_I \Bigl(
\sum_{ J \in \cG (N)} u^J   
\sum_{K \in \cI (N)} \alpha_J^K {\bf e}_{K}
\Bigr)
&=
\sum_{ J \in \cG (N)} u^J  
\sum_{K \in \cI (N)}    \alpha_J^K  (V_{K} {\rm Id}) (x)  
\\
&=
\sum_{ J \in \cG (N)} u^J   (V_{[J]} {\rm Id}) (x) .
\end{align*}
For $k \ge 2$, the computation gets a little bit complicated.
Let us consider the case $k=2$.
The concatenation of two words, $K_1$ and $K_2$, 
is denoted by $(K_1, K_2)$.
By summing over $I$ first, we see that
\begin{eqnarray}
\lefteqn{
\sum_{I \in \cI (N)}  (V_{I} {\rm Id}) (x) 
 \pi_I \Bigl(
\sum_{ J_1, J_2 \in \cG (N)} u^{J_1}   u^{J_2}
\sum_{K_1, K_2 \in \cI (N)} \alpha_{J_1}^{K_1}\alpha_{J_2}^{K_2} {\bf e}_{(K_1, K_2)}
\Bigr)
}
\label{eq.souji}\\
&=&
\sum_{J_1,  J_2 \in \cG (N)}  u^{J_1} u^{J_2}
\sum_{K_1, K_2 \in \cI (N), |K_1|+ |K_2| \le N}   \alpha_{J_1}^{K_1}  \alpha_{J_2}^{K_2} 
  (V_{K_1}V_{K_2} {\rm Id}) (x).
  \nn
  \end{eqnarray}
For any $l, m \ge 1$ with $l+m \le N$, we have by \eqref{eq.V[]} that
\[
\sum_{|K_1| =l}   \sum_{ |K_2| =m}   
\alpha_{J_1}^{K_1}  \alpha_{J_2}^{K_2} 
  (V_{K_1}V_{K_2} {\rm Id}) (x)
  =  (V_{[J_1]}V_{[J_2]} {\rm Id}) (x) \delta_{|J_1|,l}\delta_{|J_2|,m}.
  \]
 Hence, the left-hand side of (\ref{eq.souji}) is equal to
 \[
 \sum_{J_1, J_2 \in \cG (N), |J_1|+ |J_2| \le N}  
 u^{J_1} u^{J_2}
  (V_{[J_1]}V_{[J_2]} {\rm Id}) (x).
    \] 
 This proves the case for $k=2$. 
We can prove the case $k \ge 3$ essentially in the same way.
Thus, we have shown (\ref{eq.YtoU}). 
\end{proof}

Next we give two types of bracket-generating condition for the vector fields.
For $x \in {\mathbb R}^d$ and $k \ge 1$,
define $\cA_{k} (x)$ to be the linear span of
$\{ V_{[I]} (x) \mid  I \in \cI (k) \}$
in $T_x {\mathbb R}^d \cong {\mathbb R}^d$.
Note that it equals the linear span of
$\{ V_{[I]} (x) \mid  I \in \cG (k) \}$.

\vspace{4mm}
\noindent 
${\bf (HC)}_x$:~
We say that $\{V_1, \ldots, V_n\}$ satisfies the H\"ormander
condition at $x$ if there exists $N \ge 1$ such that
$\cA_{N} (x) ={\mathbb R}^d$.

\vspace{4mm}

The smallest number $N$ with this property is called 
the step of the H\"ormander condition at $x$
and denoted by $N_0 (x)$.
We set $\nu (x)= \sum_{k=1}^{N_0 (x)} k (\dim {\cA}_k (x) 
-  \dim {\cA}_{k-1} (x))$ with
${\cA}_{0} (x) :=\{0\}$ by convention.

\vspace{4mm}

\noindent 
${\bf (ER)}_x$:~
We say that $\{V_1, \ldots, V_n\}$ satisfies the 
equiregular H\"ormander
condition on $O \subset {\mathbb R}^d$ 
if {\rm (i)} it satisfies ${\bf (HC)}_x$ at every $x \in O$
and
{\rm (ii)} for all $k$, 
$\dim \cA_{k} (x)$ is constant in $x \in O$.
If the equiregular H\"ormander condition holds 
on some neighborhood of $x$, we simply say 
$\{V_1, \ldots, V_n\}$ satisfies the 
equiregular H\"ormander condition {\it near} $x$ 
and denote it by ${\bf (ER)}_x$.

\vspace{4mm}

Assume ${\bf (HC)}_x$ at some 
$x \in {\mathbb R}^d$. 
Then, we can find $\cH (x) \subset \cG (N_0 (x)) $ 
such that $\# \cH (x) =d$ and 
$\cA_{k} (x)$ equals the linear span of
$\{ V_{[I]} (x) \mid  I \in \cG (k) \cap \cH (x) \}$ for all 
$k =1,\ldots, N_0 (x)$.
Take $J \in \cG (N_0 (x))$
and write $V_{[J]} (x)$ as a unique linear combination
of $\{ V_{[I]} (x) \}_{I \in \cH (x)}$: 
\[
V_{[J]} (x) = \sum_{I \in \cH (x)}  c_J^I (x)  V_{[I]} (x).
\]
Then, we can immediately see from the definition 
of $\cH (x)$ that $c_J^I (x) =0$ is $|I| > |J|$.

Now we assume ${\bf (ER)}_{x_0}$ for $x_0  \in {\mathbb R}^d$.
Then, on a certain neighborhood $O$ of $x_0$,
we can choose $\cH (x)$ independently from $x$
and in that case we simply write $\cH$.
The linear subspace of ${\frak g}^{N_0} ({\mathbb R}^n)$
generated by $\{ {\bf e}_{[I]} \mid I \in \cH\}$
is denoted by ${\mathbb R} \la \cH\ra$.
Likewise, $N_0 (x)$ and $\nu (x)$ are independent of $x \in O$ 
and denoted by $N_0$ and $\nu$, respectively.
We will fix such $O$ for a while.

We introduce some linear maps 
for  each $x \in O$.
First, set $B_{\cH} (x) \in 
\cL ({\mathbb R} \la \cH \ra , {\mathbb R}^d)$
by
\begin{equation}\label{def.BcH}
B_{\cH}  (x) = \bigl(   V_{[I]}^i (x) \bigr)_{1 \le i \le d, I \in \cH},
\end{equation}
which is clearly invertible.
Next, set 
$\Gamma (x) 
=( \gamma^I_J (x))_{I \in \cH, J \in \cG (\infty) }
\in \cL ( {\frak g}^{\infty} ({\mathbb R}^n),
{\mathbb R} \la \cH\ra)$
by
\[
\Gamma (x) =B_{\cH}  (x)^{-1} 
\cdot
\bigl[
\bigl(   V_{[I]}^i (x) \bigr)_{1 \le i \le d, I \in \cG (\infty)} \bigr].
\]
Here, ${\frak g}^{\infty} ({\mathbb R}^n) 
(\cong {\mathbb R}^{\cG (\infty)})$ is the free Lie algebra
generated by ${\mathbb R}^n$.
Then, from  Lemma \ref{lm.la} below
and the fact that
$c_J^I (x) =0$ is $|I| > |J|$ for $I \in \cH$ and $J \in \cG (N_0)$, 
it follows that
\begin{equation}\label{eq.gam}
 \gamma^I_J (x)
 =
 \begin{cases} 
 \delta^I_J  & \mbox{if $J \in \cH$,} \\
 0 & \mbox{if $J \in \cG (N_0)$ and $|I| > |J|$.}
  \end{cases}
 \end{equation}
For $N \ge N_0$, we set
$\Gamma_N (x) 
=( \gamma^I_J (x))_{I \in \cH, J \in \cG (N) }
\in \cL ( {\frak g}^{N} ({\mathbb R}^n) ,  {\mathbb R}\la \cH \ra)$.
It immediately follows from  (\ref{eq.gam}) that
$\Gamma_N (x)\Gamma_N (x)^* \ge {\rm Id}_{\cH}$.


Here, we give a simple lemma on linear algebra in a general setting.
\begin{lemma} \label{lm.la}
Suppose that $\{ {\bf b}_1, \ldots, {\bf b}_d\}$ is a linear basis
of ${\mathbb R}^d$. 
Let ${\bf a}_1, \ldots, {\bf a}_m~(m \ge 1)$ be given by 
linear combinations of ${\bf b}_j$'s as follows:
\[
{\bf a}_k = \sum_{j=1}^d   c_k^j {\bf b}_j
\qquad
(1 \le k \le m).
\] 
Set an invertible matrix
$B=[{\bf b}_1, \ldots, {\bf b}_d]$
and a $d \times m$ matrix 
$C= (c_k^j)_{1 \le j \le d, 1 \le k \le m}$.
Then, we have
$$
B^{-1} [ {\bf b}_1, \ldots, {\bf b}_d, {\bf a}_1, \ldots, {\bf a}_m]
=
[ {\rm Id}_d | C]
$$
as a $d \times (d+m)$ matrix.
Here, ${\rm Id}_d$ stands for the identity matrix of size $d$.
\end{lemma}

\begin{proof}
The proof is immediate if we note that 
$B^{-1} {\bf b}_i = {\bf e}_i$ for all $i$,
where $\{ {\bf e}_1, \ldots, {\bf e}_d\}$ is the canonical 
linear basis of ${\mathbb R}^d$. 
\end{proof}

We fix a few more notations for $N \ge 1$.
In this paragraph we do not assume {\bf (HC)}, {\bf (ER)}
nor $N \ge N_0 (x)$.
Set $B_{N}  \in 
C^{\infty}( {\mathbb R}^d,\cL ({\frak g}^{N} ({\mathbb R}^n) , {\mathbb R}^d))$
by
\begin{equation}\label{def.B}
B_{N}  (x) = \bigl(   V_{[I]}^i (x) \bigr)_{1 \le i \le d, I \in \cG (N)}.
\end{equation}
Next, define 
${\bf V}_{I_1, \ldots, I_N} 
=({\bf V}_{I_1, \ldots, I_N}^{ij})_{1 \le i,j \le d}
\in C^{\infty}( {\mathbb R}^d,
\cL ({\mathbb R}^{d},{\mathbb R}^d) )$
for 
$I_1, \ldots, I_N \in \cI (\infty)$ by
\begin{equation}\label{def.bfV}
{\bf V}_{I_1, \ldots, I_N}^{ij}
=
\frac{\p}{\p x^j} ( V_{[I_1]} \cdots V_{[I_N]}   x^i),
\end{equation}
where $x^i$ stands for the $i$th coordinate function 
$x \mapsto x^i$ on ${\mathbb R}^d$.
By convention we set ${\bf V}_{\emptyset} = {\rm Id}_d$.
It is obvious that
\begin{equation}\label{eq.VbfV}
V_{[I_1]} \cdots V_{[I_N]}   x^i =\sum_{j=1}^d 
V_{[I_1]}^j
{\bf V}_{I_2, \ldots, I_N}^{ij}
\end{equation}
for $N \ge 2$.
We also define 
$M_N
=(M_N^{ij})_{1 \le i,j \le d}
\in C^{\infty}( {\mathbb R}^d \times {\frak g}^N ({\mathbb R}^d),
\cL ({\mathbb R}^d, {\mathbb R}^d) )$
by
\begin{equation}\label{def.M}
M_N^{ij} (x,u) = \delta^i_j  +\sum_{k=1}^{N-1} \frac{1}{(k+1)!}
\sum_{I_1, \ldots, I_{k} \in \cG (N)}
{\bf V}_{I_1, \ldots, I_{k}}^{ij} (x)
u^{I_1} \cdots u^{I_{k}}
\qquad
(u = \sum_{I \in \cG (N)} u^I {\bf e}_{[I]}).
\end{equation}
Finally, set 
$F_N \in C^{\infty}( {\mathbb R}^d \times {\frak g}^N ({\mathbb R}^d),
{\mathbb R}^d )$ by
\begin{equation}\label{def.F}
F_N (x,u) =M_N (x,u) B_N (x) u.
\end{equation}

Let us assume ${\bf (ER)}_{x_0}$ again
and that $x$ is sufficiently close to $x_0$ and $N \ge N_0$.
It immediately follows that $M_N (x,0) = {\rm Id}_d$ and 
$(\p_I F_N^i (x, 0))_{1 \le i \le d, I \in \cG (N_0)}
 =B_{N_0}  (x)$.
Here, $\p_I$ is a shorthand for $\p /\p u^I$. 
Therefore, there exist a neighborhood 
$O_N$
of $x_0$
and positive constants 
$\kappa_N, r$ such that
if $|u| \le \kappa_N$ and $x \in O_N$,
then
\begin{equation}
\label{ineq.detM}
\det M_N (x,u) \ge \frac12,
\qquad
M_N (x,u)^*  M_N (x,u) \ge \frac12 {\rm Id}_d
\end{equation}
and
\begin{align}\label{ineq.qdrt}
\lefteqn{
(\p_I F_N^i (x, u))_{1 \le i \le d, I \in \cG (N_0)}
[ (\p_I F_N^i (x, u))_{1 \le i \le d, I \in \cG (N_0)}]^*
}
\\
\nonumber
&\qquad \ge
\frac12 
B_{N_0}  (x_0)  B_{N_0}  (x_0)^*
\ge
\frac12 
B_{\cH}  (x_0)  B_{\cH}  (x_0)^*
\ge
r {\rm Id}_d.
\end{align}

We continue to  assume ${\bf (ER)}_{x_0}$ and let $O_N$ be as above.
We define four linear maps 
for $N \ge N_0$, $0<\ve \le 1$ and $x \in O_N$ as follows:
\begin{align}
\label{def.gntn1}
\tilde{\Gamma}_N^{\ve} (x)
&= \bigl( 
\ve^{|J| -|I|}  \gamma^I_J (x)
\bigr)_{I \in \cH, J \in \cG (N)}
\quad \in \cL ({\frak g}^N ({\mathbb R}^n),  {\mathbb R} \la \cH\ra),
\\
\label{def.gntn2}
\tilde{\Gamma}_N^{0} (x)
&= \bigl( 
\delta_{|J|}^{|I|}  \gamma^I_J (x)
\bigr)_{I \in \cH, J \in \cG (N)}
\quad \in \cL ({\frak g}^N ({\mathbb R}^n),  {\mathbb R} \la \cH\ra),
\\
P_N &= 
\bigl( 
\delta_{J}^{I}  
\bigr)_{I \in  \cG (N) \setminus \cH, J \in \cG (N)}
\quad \in 
 \cL ( {\frak g}^N ({\mathbb R}^n),
  {\frak g}^N ({\mathbb R}^n)/{\mathbb R} \la \cH\ra),
\\
\Delta_{\ve}^{\cH} 
&= 
\bigl( 
\ve^{|J|}  \delta_{J}^{I}  
\bigr)_{I \in \cH, J \in \cH}
\quad \in 
\cL ( {\mathbb R} \la \cH\ra ,  {\mathbb R} \la \cH\ra).
\end{align}
Note that $\Delta_{\ve}^{\cH}$ is the dilation by $\ve$
restricted to ${\mathbb R} \la \cH\ra$
and that $P_N$ is just the canonical projection.
Via the inner product on $ {\frak g}^N ({\mathbb R}^n)$,
${\frak g}^N ({\mathbb R}^n)/{\mathbb R} \la \cH\ra$
is canonically identified with the orthogonal complement 
of ${\mathbb R} \la \cH\ra$ in ${\frak g}^N ({\mathbb R}^n)$.
In this way $P_N$ can be regarded as the orthogonal projection.
In fact, no negative power of $\ve$ is involved in 
the components of $\tilde{\Gamma}_N^{\ve} (x)$,
thanks to (\ref{eq.gam}).
By definition, we have 
$\tilde{\Gamma}_N^{0} (x) u
=\tilde{\Gamma}_{N_0}^{0} (x) \Pi_{N}^{N_0}  u$
for all $u \in {\frak g}^N ({\mathbb R}^n)$ and $N \ge N_0$.
The linear mapping 
\begin{equation}\label{def.clvec}
\begin{pmatrix}
\tilde{\Gamma}_N^{\ve} (x)  \\
P_N
\end{pmatrix}
\in 
 \cL ( {\frak g}^N ({\mathbb R}^d),
  {\frak g}^N ({\mathbb R}^d))
\end{equation}
will play an important role.

Now we give two simple lemmas for later use.
\begin{lemma}\label{lm.veLA}
Let the notations be as above and 
let $N \ge N_0$, $0<\ve \le 1$. 
Then, if we take $O_N$ small enough, we have the following:
\begin{enumerate}
\item
$\lim_{\ve\searrow 0} \tilde{\Gamma}_N^{\ve} (x)
= \tilde{\Gamma}_N^{0} (x)$
uniformly in $x \in O_N$.
\item
$\Gamma_N (x)\Delta_{\ve}^{N}
= \Delta_{\ve}^{\cH}  \tilde{\Gamma}_N^{\ve} (x)$ for all $x \in O_N$.
\item
$\det  \Delta_{\ve}^{\cH} = \ve^{\nu}$. In particular, 
$\Delta_{\ve}^{\cH}$ is invertible.
\item
$\tilde{\Gamma}_N^{0} (x) \tilde{\Gamma}_N^{0} (x)^*
\ge {\rm Id}_{{\mathbb R} \la \cH\ra}$ for all $x \in O_N$.
\item The linear mapping defined in (\ref{def.clvec}) 
is invertible and there exists a positive constant $r_N$ 
such that, for all $x \in O_N$,
\[
\begin{pmatrix}
\tilde{\Gamma}_N^{0} (x)  \\
P_N
\end{pmatrix}
\begin{pmatrix}
\tilde{\Gamma}_N^{0} (x)  \\
P_N
\end{pmatrix}^*
\ge 
r_N {\rm Id}_{ {\frak g}^N ({\mathbb R}^d) }.
\]
\item
There exists $\ve_0 =\ve_0 (N) \in (0,1]$ such that,
for all $\ve \in (0,\ve_0]$ and $x \in O_N$, 
\[
\tilde{\Gamma}_N^{\ve} (x) \tilde{\Gamma}_N^{\ve} (x)^*
\ge
\frac12 {\rm Id}_{{\mathbb R} \la \cH\ra},
\qquad
\begin{pmatrix}
\tilde{\Gamma}_N^{\ve} (x)  \\
P_N
\end{pmatrix}
\begin{pmatrix}
\tilde{\Gamma}_N^{\ve} (x)  \\
P_N
\end{pmatrix}^*
\ge 
\frac{r_N}{2}
 {\rm Id}_{ {\frak g}^N ({\mathbb R}^d) }.
\]
\end{enumerate}
\end{lemma}

\begin{proof}
(2) and (3) are obvious.
By (\ref{eq.gam}) no component has a negative power of $\ve$,
from which (1) immediately follows.
Noting that $\tilde{\Gamma}_N^{0} (x) 
= [{\rm Id}_{{\mathbb R} \la \cH\ra} |\, *\,]$,
where 
$*$ is a certain smooth function in $x$,
we have (4).
In a similar way, we show (5). Obviously, 
\[
\begin{pmatrix}
\tilde{\Gamma}_N^{0} (x)  \\
P_N
\end{pmatrix}
=
\begin{pmatrix}
   {\rm Id}_{{\mathbb R} \la \cH\ra}   & * \\
{\bf 0} & 
{\rm Id}_{  {\frak g}^N ({\mathbb R}^d)  /{\mathbb R} \la \cH\ra} 
\end{pmatrix}
\]
is invertible for all $x \in O_N$.
A standard compactness argument implies existence 
of such a positive constant $r_N$. 
(6) is immediate from (1), (4), (5).
\end{proof}

\begin{lemma}\label{lm.nondegG}
Assume ${\bf (ER)}_{x_0}$ and $N \ge N_0$.
For convenience 
we set $Z^{\ve}_N (x)$ to be either  
$\tilde{\Gamma}_N^{\ve} (x) U_1^N$ or 
\[
\begin{pmatrix}
\tilde{\Gamma}_N^{\ve} (x)  \\
P_N
\end{pmatrix}
U_1^N.
\]
Then, there exist
a neighborhood $O_N$ of $x_0$ and 
$\ve_0 =\ve_0 (N) \in (0,1]$ such that $Z^{\ve}_N (x)$ is
non-degenerate in the sense of Malliavin
uniformly in $x \in  O_N$ and $\ve \in (0,\ve_0]$,
that is,
\[
\sup_{ x \in  O_N} \sup_{ 0 < \ve \le \ve_0}  
\| \sigma [ Z^{\ve}_N (x) ]^{-1} \|_{L^p}
<\infty
\qquad
(\mbox{for every $1 <p < \infty$}).
\]
\end{lemma}

\begin{proof}
It is easy to see that
\[
\sigma [ \tilde{\Gamma}_N^{\ve} (x) U_1^N]
\ge 
 \tilde{\Gamma}_N^{\ve} (x) 
 \sigma [  U_1^N]
   \tilde{\Gamma}_N^{\ve} (x)^* 
   \ge 
    \lambda [  U_1^N] \cdot
    \tilde{\Gamma}_N^{\ve} (x)    \tilde{\Gamma}_N^{\ve} (x)^*.
      \]
     The other Wiener functional 
     also satisfies a similar estimate.
Then, this lemma easily follows from 
Lemma \ref{lm.veLA} (6), and Proposition \ref{pr.MalU}.
\end{proof}

\section{Stochastic differential equation on ${\mathbb R}^d$}

For $V_1, \ldots, V_n \in {\frak X}  ({\mathbb R}^d)$
and $0 < \ve \le 1$, we consider 
the following Stratonovich-type SDE on ${\mathbb R}^d$:
\begin{equation}\label{def.sde}
dX^{\ve} (t, x) = \ve \sum_{i=1}^n V_i (X^{\ve} (t, x) ) \circ dw^i_t
\qquad
\mbox{with $X^{\ve} (0, x)=x$.}
\end{equation}
When $\ve =1$, we simply write $X (t, x)$ for $X^{\ve} (t, x)$.
By the well-known scaling property, 
$(X^{\ve} (t, x))_{t \ge 0}$ and $(X(\ve^2 t, x))_{t \ge 0}$
have the same law.

For SDE (\ref{def.sde}),
we always assume that  $V_i^j := \la dx^j, V_i\ra$
has bounded derivatives of all order $\ge 1$ 
($1 \le i \le n, 1 \le j \le d$).
This is a standard assumption in Malliavin calculus.

The aim of this section is to prove that 
$\delta_x ( X^{\ve} (1, x))$ admits an asymptotic expansion 
as $\ve \searrow 0$ in the space of Watanabe distributions 
uniformly in $x$ under the equiregular H\"ormander condition 
on the coefficient vector fields (Theorem \ref{tm.R^d}). 
The expansion for each fixed $x$ 
under the usual H\"ormander condition was already proved in \cite{tak}.
We carefully follow the argument  in \cite{tak}
and show the uniformity of the expansion under the equiregular condition.
At the end of the section, 
we discuss the case of SDE with a nice drift term 
(Corollary \ref{cor.Gir}).

Now we recall the stochastic Taylor expansion in $\ve$.
Note that (\ref{eq.te})--(\ref{def.R}) is an asymptotic expansion
in ${\mathbb D}_{\infty}$-topology for each fixed $x$ and $t$.
The aim of the next proposition is to make sure the uniformity 
of the expansion in $x$ as $x$ varies in a compact subset.

\begin{proposition}\label{pr.ste}
Let the notations be as above and let $N \ge 1$. Then, 
we have 
\begin{eqnarray}\label{eq.te}
X^{\ve} (t, x) = x + E_N^{\ve} (t,x) + R_{N+1}^{\ve} (t,x),
\end{eqnarray}
where we  set
\begin{align}
E_N^{\ve} (t,x) 
&=
\sum_{I \in \cI (N)}  \ve^{|I|} ( V_I {\rm Id} )(x) w_t^I,
\label{def.E}
\\
\label{def.R}
 R_{N+1}^{\ve} (t,x) &=
\ve^{N+1}
\sum_{I \in \cI (N+1) \setminus  \cI (N)}
\\
& \qquad
\int_0^{t} 
\cdots
\int_0^{t_3} 
\int_0^{t_2}
 ( V_I {\rm Id} )( X^{\ve} (t_1, x)) \circ dw_{t_1}^{i_1}
 \circ dw_{t_2}^{i_2} 
 \cdots  \circ dw_{t_{N+1}}^{i_{N+1}}.
 \nonumber
 \end{align}
Moreover, for each compact set $K \subset {\mathbb R}^d$, 
the asymptotic expansion (\ref{eq.te}) is uniform in 
$(t, x) \in [0,1] \times K$, that is,
\begin{equation}\label{ineq.Runi}
\sup_{0 \le t \le 1} \sup_{x \in K} 
\|  R_{N+1}^{\ve} (t,x)\|_{p,k} \le C \ve^{N+1}
\qquad
(N \ge 1, 1 <p<\infty, k \in {\mathbb N})
\end{equation}
holds for all $\ve \in (0,1]$.
Here,  $C =C(N, p, k)$ is a certain 
positive constant independent of  $\ve$.
\end{proposition}

\begin{proof}
This is well-known at least when $x$ is fixed.  
So we only give a sketch of proof so that
one can see the uniformity in $x$.
In this proof the positive constant $C =C(N, p, k)$
 may change from line to line.

Firstly, it is well-known that
\[
\sup_{0 \le t \le 1} \sup_{x \in K} 
\|  D^k X^{\ve} (t,x)\|_{L^p} \le C \ve^k
\qquad
(1 <p<\infty, k \in {\mathbb N}).
\]
The proof is standard, although it is long and may not be so easy.
Secondly, we can see from (\ref{def.R}) that
\[
\sup_{0 \le t \le 1} \sup_{x \in K} 
\|  R_{N+1}^{\ve} (t,x)\|_{L^p} \le C \ve^{N+1}
\qquad
(N \ge 1, 1 <p<\infty, k \in {\mathbb N}).
\]
Thirdly, $D^{N+1} E_N^{\ve}(t,x) =0$.

Now we use the stronger form of Meyer's inequality.
If $k \ge N+1$, then 
\begin{eqnarray*}
\|  R_{N+1}^{\ve} (t,x)\|_{p,k} &\le&
C ( \|  R_{N+1}^{\ve} X^{\ve} (t,x)\|_{L^p} 
  + \| D^{k} R_{N+1}^{\ve}  (t,x)\|_{L^p})
  \\
  &\le&
C(  \|  R_{N+1}^{\ve} X^{\ve} (t,x)\|_{L^p} 
  + \| D^{k} X^{\ve}  (t,x)\|_{L^p})
  \le
C  \ve^{N+1}.
  \end{eqnarray*}
  Since the Sobolev norm is increasing in $k$, we are done.
\end{proof}

\vspace{10mm}

We modify the stochastic Taylor expansion 
(\ref{eq.te})--(\ref{def.R}) for later use.
The definition of $F_N$ was given in (\ref{def.F}).

\begin{proposition}\label{pr.ste2}
Let $N \ge 1$. Then, we have 
\begin{eqnarray}\label{eq.te2}
X^{\ve} (t, x) = x + F_N (x, \Delta^N_{\ve} U^N_t)
+ \hat{R}_{N+1}^{\ve} (t,x).
\end{eqnarray}
Here we  set
\begin{align}
\label{def.hatR}
\hat{R}_{N+1}^{\ve} (t,x) &=
R_{N+1}^{\ve} (t,x) 
\\
& \quad
-
\sum_{k=1}^N \frac{1}{k!}
\sum_{ |I_1| + \cdots + |I_k| >N}  
\ve^{|I_1| + \cdots + |I_k|}   
( V_{[I_1]} \cdots  V_{[I_k]}  {\rm Id} )(x)  U_t^{I_1} \cdots U_t^{I_k},
\nonumber
 \end{align}
 where the second summation  runs over 
 all $(I_1, \ldots, I_k) \in \cG (N)^k$ such that
 $|I_1| + \cdots + |I_k| >N$.
Moreover, for each compact set $K \subset {\mathbb R}^d$, 
$\hat{R}_{N+1}^{\ve} (t,x)$ satisfies the same estimate 
as in (\ref{ineq.Runi}) for a different constant $C >0$.
\end{proposition}

\begin{proof}
The second assertion is almost obvious. 
Since it is immediate from (\ref{def.bfV})--(\ref{def.F}) that
\[
F_N (x, \Delta^N_{\ve} U^N_t)
=
\sum_{k=1}^N \frac{1}{k!}
\sum_{I_1, \ldots, I_k \in \cG (N)}  
\ve^{|I_1| + \cdots + |I_k|}   
( V_{[I_1]} \cdots  V_{[I_k]}  {\rm Id} )(x) U_t^{I_1} \cdots U_t^{I_k},
\]
it is enough to see that
\begin{equation}\label{eq.Eu}
E_N^{\ve} (t,x) 
=
\sum_{k=1}^N \frac{1}{k!}
\sum_{ |I_1| + \cdots + |I_k| \le N}  
\ve^{|I_1| + \cdots + |I_k|}   
( V_{[I_1]} \cdots  V_{[I_k]}  {\rm Id} )(x) U_t^{I_1} \cdots U_t^{I_k}.
\end{equation}
Here, the second summation  runs over 
 all $(I_1, \ldots, I_k) \in \cG (N)^k$ such that
 $|I_1| + \cdots + |I_k| \le N$.
Equality (\ref{eq.Eu}) immediately follows from Lemma \ref{lm.YtoU}
 and the definition of $U_t^N$.
\end{proof}

Recall Kusuoka-Stroock's estimate 
for Malliavin covariance matrix of $X^{\ve}(1,x)$ under
the H\"ormander condition at $x_0$.
Our aim here is to make sure the estimate is uniform in $x$
as it varies in a small neighborhood of $x_0$.
Note that the equiregular condition is not needed here.
\begin{proposition}\label{pr.KS}
Assume ${\bf (HC)}_{x_0}$ at $x_0 \in {\mathbb R}^d$.
Then, there exist a neighborhood $O$ of $x_0$
and a positive constants  $M$ independent of $p$, $x$ and $\ve$ 
such that 
\[
\sup_{x \in O} \sup_{0 < \ve \le 1}
\ve^M
\| \det \sigma [X^{\ve}(1,x)]^{-1} \|_{L^p} <\infty
\qquad
\mbox{for every $p \in (1,\infty)$.}
\]
In particular, 
$X^{\ve}(1,x)$ is non-degenerate in the sense of Malliavin 
for every $\ve \in (0,1]$ and $x \in O$.
\end{proposition}

\begin{proof}
This is proved in \cite[Theorem (2.17)]{ks2}.
\end{proof}

Due to the non-degeneracy of $X^{\ve}(1,x)$ in the sense of Malliavin,
Watanabe's pullback of the delta function 
$\delta_x (X^{\ve}(1,x)) \in \tilde{\mathbb D}_{- \infty}$ 
is well-defined.
Takanobu \cite{tak} showed that
$\delta_x (X^{\ve}(1,x)) = \delta_0 (X^{\ve}(1,x) -x)$ admits an asymptotic expansion 
up to any order as $\ve\searrow 0$ in 
$\tilde{\mathbb D}_{- \infty}$-topology.

Now we  present our main result in this section.
This is a uniform version of Takanobu's main theorem in \cite{tak}.
To prove the uniformity in the starting point $x$,
we need to assume the equiregular H\"ormander condition
near $x_0$.
(Note that the SDE in \cite{tak} has a drift term. 
On the other hand, the dependency on $x$ is not studied in \cite{tak}.)

\begin{theorem}\label{tm.R^d}
Assume ${\bf (ER)}_{x_0}$. 
Then,  
there exists a decreasing sequence $\{ O_j \}_{j \ge 0}$ of
neighborhoods of $x_0$ 
such that the asymptotic expansion 
\begin{equation}\label{asy.dX}
\delta_x (X^{\ve}(1,x))
\sim
\ve^{-\nu} \bigl( \Theta_0 (x)+ \ve  \Theta_1 (x)
+ \ve^2  \Theta_2 (x)+\cdots
\bigr)
\qquad
\mbox{in $\tilde{\mathbb D}_{- \infty}$ as $\ve\searrow 0$.}
\end{equation}
holds for every $x \in O_0$ with the following properties:
 {\rm (i)}~ $\inf_{x \in O_0} {\mathbb E} [\Theta_0] >0$,
{\rm (ii)}~for every $j \ge 0$
there exists $k=k(j)>0$ such that
\[
\sup_{x \in O_j} \{  \|\Theta_j (x)\|_{p, -k} 
+\sup_{0 < \ve \le1} \|
\ve^{-(j+1-\nu)} 
r^{\ve}_{j+1} (x)\|_{p,-k}  \}<\infty
\]
for all $p \in (1,\infty)$.
Here, we set
\[
r^{\ve}_{j+1} (x)=\delta_x (X^{\ve}(1,x))
-
\ve^{-\nu} \bigl( \Theta_0 (x)
+\cdots
+ \ve^{j}  \Theta_{j} (x)
\bigr).
\]
Moreover, $\Theta_{2j-1}(x; \,\cdot\, )$ is odd as a Wiener functional
for every $j \ge 1$ and $x \in O_0$,
that is,  $\Theta_{2j-1} (x; -w) = - \Theta_{2j-1} (x; w)$.
\end{theorem}

\begin{remark}\label{re.cnt}
In fact,
$O_j \ni x \mapsto \Theta_j (x) 
\in \tilde{\mathbb D}_{- \infty}$ is continuous for every $j$.
This follows from the uniformity 
of the asymptotic expansion (\ref{asy.dX}) and
continuity of 
$x \mapsto  \delta_x (X^{\ve}(1,x)) 
=\delta_0 (X^{\ve}(1,x)-x) \in \tilde{\mathbb D}_{- \infty}$.
The latter, in turn, follows from Proposition \ref{pr.KS}
and continuity of 
$x \mapsto   X^{\ve}(1,x)-x \in {\mathbb D}_{\infty}$.
\end{remark}

The rest of this section is devoted to proving 
the above theorem.
The neighborhoods $O$ and  $O_j$, $j \ge 0$, may change from line to line.

We introduce a few functions for technical purposes.
Take $\psi \in C^{\infty} ({\mathbb R}, [0,1])$ such that
$\psi (s) = 0$ if $|s| \ge 1$ and $\psi (s) = 1$ if $|s| \le 1/2$.
Take any $\kappa >0$
and set $\psi_N (x)= \psi (x/ (\kappa /2)^2)$ for $N \ge N_0$.
Define a smooth Wiener function 
$\chi^{\ve}_N \in {\mathbb D}_{\infty}$ by 
$\chi^{\ve}_N=\psi_N ( |\Delta^N_{\ve} U^N_1 |^2)$.

\begin{lemma}\label{lm.Xloc}
Assume ${\bf (HC)}_{x_0}$ and $N \ge N_0$. 
Then, there exist a positive constant $k$ independent of $N$ and
a neighborhood $O_N$ of $x_0$ 
such that the following property holds:
For every  $p \in (1, \infty)$, 
there exist positive constants $c_1$ and $c_2$ independent 
of $\ve$ and $x \in O_N$ such that
\[
\sup_{x \in O_N}  
\| \delta_x (X^{\ve}(1,x)) 
- \chi^{\ve}_N\cdot \delta_x (X^{\ve}(1,x))\|_{p, -k} 
\le c_1 e^{-c_2/\ve^2}
\qquad
\mbox{as $\ve\searrow 0$.}
\]
\end{lemma}

\begin{proof}
We use \cite[p. 374, Formula (8.47)]{iwbk}:
For every $q, r \in (1, \infty)$ such that
$1/p:= 1/q+1/r <1$ and every $k \in {\mathbb N}$,
there exists a positive constant $C_{q,r,k}$ such that
\[
\| FG \|_{p, -k} \le C_{q,r,k}  \| F \|_{q, k} \| G \|_{r, -k}
\qquad
(F \in {\mathbb D}_{q,k}, G  \in {\mathbb D}_{r, -k}).
\]

We use this formula with $F= 1 - \chi^{\ve}_N$
and $G= \delta_x (X^{\ve}(1,x))$.
By Proposition \ref{pr.KS} and Watanabe's pullback theorem, 
we can find $k$ and $M' >0$ such that
\[
\sup_{x \in O} \sup_{0 < \ve \le 1}
\ve^{M'}
 \|  \delta_x (X^{\ve}(1,x)) \|_{r, -k}  <\infty
 \]
for any $r \in (1,\infty)$.
On the other hand, we can easily see from (\ref{ineq.exitU}) that 
\[
  \| 1 - \chi^{\ve}_N \|_{q, k}  = O (e^{-  C_{N,\kappa}/4 \ve^2 }) 
  \qquad
  \mbox{as $\ve \searrow 0$}
  \]
 for every $q \in (1,\infty)$ and $k \in {\mathbb N}$. 
This completes the proof.
\end{proof}

The following is a slight extension of \cite[Lemma 5.8]{tak}.
Under the equiregular H\"ormander condition near $x_0$, 
we prove a uniform version of the lemma.
Recall that, for a Wiener functional $G$,
 $\lambda [G]$ stands for the lowest eigenvalue of 
 Malliavin covariance matrix $\sigma [G]$.
\begin{proposition}\label{pr.tak58}
Assume ${\bf (ER)}_{x_0}$ and let 
$r >0$ be the constant in (\ref{ineq.qdrt})
 and $N \ge N_0$.
Then, there exist a neighborhood $O_N$ of $x_0$ 
and $\kappa_N >0$ such that, 
for all $\ve \in (0,1]$ and $x \in O_N$,  it holds that
\begin{equation}\label{ineq.key00}
\lambda [F_N (x, \Delta^N_{\ve} U^N_1)] 
\ge
r \ve^{2N_0}  \lambda [U^N_1] 
\qquad
\mbox{on $\{ |\Delta^N_{\ve} U^N_1|<\kappa_N \}$.}
\end{equation}
In particular, for $\ve \in (0,1]$ and $x \in O_N$, we have
$\lambda [F_N (x, \Delta^N_{\ve} U^N_1)] >0$ almost surely
on $\{ |\Delta^N_{\ve} U^N_1|<\kappa_N \}$ 
and
\begin{equation}\label{ineq.keyyos}
{\mathbb E} [\lambda [F_N (x, \Delta^N_{\ve} U^N_1)]^{-p} 
~;~
|\Delta^N_{\ve} U^N_1|<\kappa_N]^{1/p}
\le
r^{-1}\ve^{-2N_0}  \| \lambda [U^N_1]^{-1}\|_{L^p} <\infty
\end{equation}
for every $1<p<\infty$. 
\end{proposition}

\begin{proof}
Inequality (\ref{ineq.keyyos}) is  immediate from 
(\ref{ineq.key00}) and (\ref{ineq.qdrt}).
We give a quick proof of (\ref{ineq.key00}).
In the same way as in the proof of \cite[Lemma 5.8]{tak},
it holds that, for every $z \in {\mathbb R}^d$,
\begin{align*}
\lefteqn{
\la  \sigma [F_N (x, \Delta^N_{\ve} U^N_1)]  z, z\ra
}
\\
&=
\sum_{I, J \in  \cG (N) }  \sigma [U^N_1]^{IJ}
\Bigl( \sum_{i=1}^d
z^i \ve^{|I|} \p_I F_N^i (x,  \Delta^N_{\ve} U^N_1 )
\Bigr)
\Bigl( \sum_{j=1}^d
z^j \ve^{|J|} \p_J F_N^j (x,  \Delta^N_{\ve} U^N_1 )
\Bigr)
\\
&\ge
 \lambda [U^N_1] \sum_{I \in  \cG (N) }\ve^{2 |I|} 
  \Bigl( \sum_{i=1}^d
z^i \p_I F_N^i (x,  \Delta^N_{\ve} U^N_1 )
\Bigr)^2
\\
&\ge
 \lambda [U^N_1]
  \sum_{I \in  \cG (N_0) } \ve^{2 |I|} 
    \Bigl( \sum_{i=1}^d
z^i\p_I F_N^i (x,  \Delta^N_{\ve} U^N_1 )
\Bigr)^2.
\end{align*}
The first equality is immediate from the chain rule for the $H$-derivative $D$. 
By (\ref{ineq.qdrt}), the right-hand side
is bounded from below by 
\[
\ve^{2N_0}  \lambda [U^N_1] 
\cdot
\bigl|
[ (\p_I F_N^i (x,  \Delta^N_{\ve} U^N_1 ))_{1 \le i \le d, I \in \cG (N_0)}]^*  z
\bigr|^2
\ge
r \ve^{2N_0}  \lambda [U^N_1]  |z|^2
\]
uniformly in $x \in O_N$.
\end{proof}

In what follows, 
we choose $\kappa_N >0$ as  in Proposition \ref{pr.tak58} 
and set $\psi_N (x)= \psi (x/ (\kappa_N /2)^2)$ 
and $\chi^{\ve}_N=\psi_N ( |\Delta^N_{\ve} U^N_1 |^2)$ for $N \ge N_0$.

Since non-degeneracy of 
$F_N (x, \Delta^N_{\ve} U^N_1)$ is not known, 
we cannot use the standard version of 
Watanabe's pullback (see Item {\bf (c)}, Section 2) to justify
$\delta_0 ( F_N (x, \Delta^N_{\ve} U^N_1))$.
However, 
thanks to Proposition \ref{pr.tak58} above,
a modified version of  Watanabe's pullback is available.
\begin{proposition}\label{pr.mdWat}
Assume ${\bf (ER)}_{x_0}$,
$N \ge N_0$ and let $O_N$ as in Proposition \ref{pr.tak58}.
Fix any $\ve$ and $x \in O_N$.
Then, the mapping $\cS ({\mathbb R}^d) \ni \phi
\mapsto  \chi^{\ve}_N\cdot \phi
 ( F_N (x, \Delta^N_{\ve} U^N_1) )   \in {\mathbb D}_{\infty}$
uniquely extends to a continuous linear mapping
\[
\cS^{\prime} ({\mathbb R}^d) \ni \Phi
\mapsto  \chi^{\ve}_N\cdot \Phi
 ( F_N (x, \Delta^N_{\ve} U^N_1) )   \in \tilde{\mathbb D}_{-\infty}.
 \]
\end{proposition}

\begin{proof}
This fact is actually well-known to experts of Malliavin calculus.
The key point is the integrability (\ref{ineq.keyyos}) in Proposition \ref{pr.tak58}.
For a detailed proof, see Yoshida \cite{yo}.
\end{proof}

The next lemma states that $\delta_x (X^{\ve}(1,x))$
can be approximated by 
$\delta_0 ( F_N (x, \Delta^N_{\ve} U^N_1))$
uniformly in $x$ if $N$ is large enough. 
Therefore, the problem reduces 
to the expansion of
 the latter Watanabe distribution.

\begin{lemma}\label{lm.Xaprx}
Assume ${\bf (ER)}_{x_0}$. 
Then, there exist $k>0$,
a sequence of 
$\{ O_{N} \}_{N \ge N_0}$ neighborhoods of $x_0$
and a  sequence $\{ l_{N} \}_{N \ge N_0}$ of real numbers  diverging to $+\infty$ such that, for every $p \in (1,\infty)$ and $N \ge N_0$, 
 \[
\sup_{x \in O_N}  
\| \delta_x (X^{\ve}(1,x)) 
- \chi^{\ve}_N\cdot \delta_0
 ( F_N (x, \Delta^N_{\ve} U^N_1) )\|_{p, -k} 
=
O (\ve^{l_N})
\qquad
\mbox{as $\ve\searrow 0$.}
\]
\end{lemma}

\begin{proof}
Due to Lemma \ref{lm.Xloc},
it is sufficient to show that
\begin{equation}\label{ineq.suf39}
\sup_{x \in O_N}  
\| \chi^{\ve}_N\cdot \delta_0 (X^{\ve}(1,x) -x) 
- \chi^{\ve}_N\cdot \delta_0
 ( F_N (x, \Delta^N_{\ve} U^N_1) )\|_{p, -k} 
=
O (\ve^{l_N})
\qquad
\mbox{as $\ve\searrow 0$.}
\end{equation}

As always, the key tool is the integration by parts formula 
for Watanabe distributions.
We also use the estimates in
Proposition \ref{pr.ste2} {\rm (ii)},  Proposition \ref{pr.KS},  
Proposition \ref{pr.tak58},   
Proposition \ref{pr.mdWat}.
In this proof we write $A^{\ve} = X^{\ve}(1,x) -x$ and
$B^{\ve} = F_N (x, \Delta^N_{\ve} U^N_1) )$
for simplicity.

First, we  prove the case $d=1$ to observe what is happening.
Set $g (x) = x \vee 0$ for $x \in {\mathbb R}$.
Then, $g^{\prime\prime} (x) = \delta_0 (x)$ in the distributional sense.
Choose smooth functions $\psi_i: {\mathbb R} \to {\mathbb R}$ ($i=1, 2, 3$) 
so that $\psi_1 =\psi$, 
$\psi_i \equiv 1$ on the support of $\psi_{i-1}$ ($i=2, 3$),
and 
 the support of $\psi_3$ is contained in $(-2, 2)$.
Set $\chi^{\ve}_{N,i} = \psi_i (  |\Delta^N_{\ve} U^N_1 |^2 / (\kappa_N/2)^2)$.
Note that Proposition \ref{pr.mdWat} still holds 
even if $\chi^{\ve}_{N} =\chi^{\ve}_{N,1}$ is replaced by 
$\chi^{\ve}_{N,2}$ or $\chi^{\ve}_{N,3}$.
Note also that $\|\chi^{\ve}_{N,i} \|_{p, k}$ is bounded in $\ve$
for any $1<p< \infty, k \ge 0, 1 \le i \le 3$

Take any $G \in {\mathbb D}_{\infty}$.
By integration by parts formula and the way $\psi_i$ ($i=1, 2, 3$) are defined, 
we have
\begin{align*}
\la \chi^{\ve}_N\cdot \delta_0 ( B^{\ve}) , G \ra
&=
\la \chi^{\ve}_{N, 2} \chi^{\ve}_{N, 3}  
g^{\prime\prime} ( B^{\ve}), \chi^{\ve}_{N, 1}G \ra
\\
&=
\Bigl\la D [
\chi^{\ve}_{N, 2} \chi^{\ve}_{N, 3}  g^{\prime} ( B^{\ve}) ], 
\frac{  D B^{\ve}  }{ \| D B^{\ve} \|_{H}^2} \chi^{\ve}_{N, 1} G 
\Bigr\ra
\\
&=
\Bigl\la 
\chi^{\ve}_{N, 2} \chi^{\ve}_{N, 3} g^{\prime} ( B^{\ve}) , 
D^* \bigl[
\frac{  D B^{\ve}  }{ \| D B^{\ve} \|_{H}^2} \chi^{\ve}_{N, 1} G 
\bigr]
\Bigr\ra,
\end{align*}
where $D$ is the $H$-derivative (the gradient operator) and
$D^*$ is its adjoint. 
(Thanks to Proposition \ref{pr.tak58}, the right-hand side is well-defined.)
Note that $\| D B^{\ve} \|_{H}^2 = \det \sigma [B^{\ve}]$ since $d=1$. 
Therefore, the second component of the pairing on the right-hand side 
coincides at least formally with $\Phi$ in (\ref{eq.ibpPHI}) 
with $m=1$
and  $F$ and $G$ being replaced by
$B^{\ve}$ and  $\chi^{\ve}_{N, 1} G$, respectively.

Using the formula again, we have
\begin{align*}
\la \chi^{\ve}_N\cdot \delta_0 ( B^{\ve}) , G \ra
&=
\Bigl\la 
  g ( B^{\ve}) , 
 \chi^{\ve}_{N, 3}
D^* \Bigl[ 
\frac{  D B^{\ve}  }{ \| D B^{\ve} \|_{H}^2}
\chi^{\ve}_{N, 2}
D^* \bigl[
\frac{  D B^{\ve}  }{ \| D B^{\ve} \|_{H}^2} \chi^{\ve}_{N, 1} G 
\bigr]
\Bigr]
\Bigr\ra
\nn\\
&=
\Bigl\la 
  g ( B^{\ve}) , 
D^* \Bigl[ 
\frac{  D B^{\ve}  }{ \| D B^{\ve} \|_{H}^2}
\chi^{\ve}_{N, 2}
D^* \bigl[
\frac{  D B^{\ve}  }{ \| D B^{\ve} \|_{H}^2} \chi^{\ve}_{N, 1} G 
\bigr]
\Bigr]
\Bigr\ra.
\end{align*}
This equation still holds for $A^{\ve}$ instead of $B^{\ve}$ for the same reason.
Observe that on the right-hand side 
$B^{\ve}$ is plugged into a (Lipschitz) continuous function $g$, 
not a Schwartz distribution. 
Hence, the difference $\|g ( A^{\ve}) - g ( B^{\ve})  \|_{L^p}$ 
is dominated by 
$\| A^{\ve} - B^{\ve}  \|_{L^p} = O (\ve^{N+1})$,
where Proposition \ref{pr.ste2} {\rm (ii)} is used.

By straight forward calculations, we can show the following estimate:
There exist constants $a\in {\mathbb N}$ 
(independent of $\ve, N, p, x$) 
and $C_p >0$  (independent of $\ve, N, x$) such that 
\begin{equation}\label{ineq.AB}
| \la \chi^{\ve}_N\cdot \delta_0 ( A^{\ve}) 
 -  \chi^{\ve}_N\cdot \delta_0 ( B^{\ve}) , G \ra|
  \le 
  C_p \| G\|_{q, 2} \,   \ve^{ N+1 - a (M+N_0) }
\end{equation}
for every $p \in (1, \infty)$, where $1/p +1/q =1$
and every $\ve \in (0,1]$ and $x \in O$.
Here, we used Propositions \ref{pr.ste2} {\rm (ii)},  \ref{pr.KS},  
\ref{pr.tak58},  and \ref{pr.mdWat}
($M$ is the positive constant in Propositions \ref{pr.KS}).
This impies (\ref{ineq.suf39}) when $d=1$ with
$k=2$ and $l_N = N+1 - a (M+N_0)$.

The proof for $d \ge 2$ is essentially the same in spirit,
but 
the notations get quite complicated
and we have to use the  integration by parts formula many times
($2d$-times is enough).
Note that the differentiability index 
$-k$ in (\ref{ineq.suf39}) is determined by this number 
and hence depends only on $d$.

Set $g (x) = \prod_{i=1}^d (x_i \vee 0)$ for $x=(x_1, \ldots, x_d) \in {\mathbb R}^d$.
Then, $(\partial_1^2  \cdots \partial_d^2 g) (x) = \delta_0 (x)$ in the distributional sense.
Choose smooth functions $\psi_i: {\mathbb R} \to {\mathbb R}$ ($i \ge 0$) 
so that $\psi_0 =\psi$, 
$\psi_i \equiv 1$ on the support of $\psi_{i-1}$ ($i \ge 1$),
and 
 the support of $\psi_{i}$ is contained in $(-2, 2)$.
Set $\chi^{\ve}_{N,i} = \psi_i (  |\Delta^N_{\ve} U^N_1 |^2 / (\kappa_N/2)^2)$.
For every $i \ge 0$, 
Proposition \ref{pr.mdWat} still holds for $\chi^{\ve}_{N,i}$  
 and $\|\chi^{\ve}_{N,i} \|_{p, k}$ is bounded in $\ve$
for any $1<p< \infty, k \ge 0$.

For a multi-index $\alpha=(\alpha_1,\dots,\alpha_d) \in {\mathbb N}^d$, 
set 
$i_\alpha=\max\{i;\alpha_i\ne0\}$ and define
$\alpha^\prime=(\alpha_1-\delta_{1i_\alpha},\dots,
 \alpha_d-\delta_{di_\alpha})$, $\delta_{ij}$ being Kronecker's
 delta.
 We define $\Phi_{(\alpha)}$ with respect to $B^{\ve}$  as follows.
If $|\alpha|:=\sum_{k=1}^d \alpha_k=1$,  we set
\[
    \Phi_{(\alpha)}(\,\cdot\,  ;G)
    =\Phi_{i_\alpha}(\,\cdot\,  ;\chi^{\ve}_{N,1} G).
\]
Recall that $\Phi_i$ is given in (\ref{eq.ibpPHI}) with $F$ 
being replaced by $B^{\ve}$.
Thanks to the ``cutoff" functional $\chi^{\ve}_{N,1}$, 
$\Phi_{i_\alpha}(\,\cdot\,  ;\chi^{\ve}_{N,1} G)$ is well-defined, 
though $B^{\ve}$ is not non-degenerate in the standard sense of Malliavin calculus.
If $|\alpha| \ge 2$, we set
\[
    \Phi_{(\alpha)}(\,\cdot\,  ;G)
    =\Phi_{ i_\alpha  }  
    \bigl( \,\cdot\,;
       \chi^{\ve}_{N, |\alpha|}  \Phi_{ (\alpha^\prime) }(\cdot;G) \bigr).
\]

Using the integration by parts formula (\ref{ipb1.eq}) repeatedly
in the same way as above,
we can show that
\begin{equation}\label{tn.eq3}
\la \chi^{\ve}_N\cdot \delta_0 ( B^{\ve}) , G \ra
 =
 \la
 g ( B^{\ve}),
   \Phi_{(\alpha)}(\,\cdot\,  ;G) \ra
\end{equation}
for every $G\in{\mathbb D}_{\infty}$, 
where $\alpha =(2,2, \ldots, 2) \in {\mathbb N}^d$.
Note  that 
(\ref{tn.eq3}) can be viewed as the definition of the Watanabe distribution
$\chi^{\ve}_N\cdot \delta_0 ( B^{\ve})$.

One can also define $  \Phi_{(\alpha)}(\,\cdot\,  ;G)$ for $A^{\ve}$ instead of $B^{\ve}$
in the same way.
Then,  (\ref{tn.eq3}) holds for $A^{\ve}$, too.
Once we have (\ref{tn.eq3}) for both $B^{\ve}$ and $A^{\ve}$,
it is straightforward to check that (\ref{ineq.AB}) also holds 
in the multi-dimensional case 
for with the differentiability index $2d$ instead of $2$
 (and possibly different $a$).
                 \end{proof}

In what follows we expand 
$\chi^{\ve}_N\cdot \delta_0 ( F_N (x, \Delta^N_{\ve} U^N_1) )$
for each fixed $N \ge N_0$.
In the next lemma the same
$\ve_0 =\ve_0 (N)$ as in 
Lemmas \ref{lm.veLA} and \ref{lm.nondegG} will do.
Note that (\ref{ineq.detM}) is implicitly used.

\begin{lemma}\label{lm.hnki}
Assume ${\bf (ER)}_{x_0}$ and $N \ge N_0$.
Then, there exist
a negihborhood $O_N$ of $x_0$ and 
$\ve_0 =\ve_0 (N) \in (0,1]$ such that 
\begin{equation}\label{eq.key}
\chi^{\ve}_N\cdot \delta_0 ( F_N (x, \Delta^N_{\ve} U^N_1) )
=
\ve^{-\nu} |\det B_{\cH}(x) |^{-1}
\frac{ \chi^{\ve}_N}
{  \det M_N (x, \Delta_{\ve}^N U_1^N)}
\cdot
\delta_0 ( \tilde{\Gamma}_N^{\ve} (x) U_1^N)
\end{equation}
holds for all $x \in O_N$ and $\ve \in (0,\ve_0]$.
Here, 
$\delta_0$ on the right-hand side is the delta function 
defined on ${\mathbb R} \la \cH\ra$.
\end{lemma}

\begin{proof}
We follow \cite[pp. 189--191]{tak}.
Since (\ref{eq.key}) is an equality and have nothing to do 
with the uniformity in $x$ of the asymptotic expansion, 
some parts of the proof here is not so detailed as the corresponding part in \cite{tak}. 

It is easy to see that
\[
U_1^N 
=
\begin{bmatrix}
\tilde{\Gamma}_N^{\ve} (x)  \\
P_N
\end{bmatrix}^{-1}
\begin{bmatrix}
\tilde{\Gamma}_N^{\ve} (x)  \\
P_N
\end{bmatrix}
U_1^N
=
\begin{bmatrix}
\tilde{\Gamma}_N^{\ve} (x)  \\
P_N
\end{bmatrix}^{-1}
\begin{bmatrix}
\tilde{\Gamma}_N^{\ve} (x)  U_1^N\\
P_N U_1^N
\end{bmatrix}.
\]
From   
Lemma \ref{lm.veLA} (2) and  an obvious fact 
that $\Gamma_N = B_{\cH}^{-1} B_N$  we see easily that
\begin{eqnarray*}
 F_N (x, \Delta^N_{\ve} U^N_1) 
=
M_N (x, \Delta_{\ve}^N U_1^N) B_{\cH} (x) 
\Delta^{\cH}_{\ve} \tilde{\Gamma}_N^{\ve} (x) U^N_1.
\end{eqnarray*}

Next, take a non-negative test function $g$ on ${\mathbb R}^d$
such that $\int g =1$ and set 
$g_{\kappa} (x) = {\kappa}^{-d} g (x/ \kappa)$ for $\kappa >0$.
Then, $g_{\kappa} \to \delta_0$ in $\cS^{\prime} ({\mathbb R}^d)$
as $\kappa \to 0$.
By the modified version of Watanabe's theory (Proposition \ref{pr.mdWat}), 
\[
\chi^{\ve}_N\cdot g_{\kappa} ( F_N (x, \Delta^N_{\ve} U^N_1) )
\to 
\chi^{\ve}_N\cdot \delta_0 ( F_N (x, \Delta^N_{\ve} U^N_1) )
\]
in $\tilde{\mathbb D}_{- \infty}$ as $\kappa \to 0$.

Before we start computing this quantity, 
we set some notations for simplicity.
Set  $T= \Delta^N_{\ve}$,
\[
\begin{bmatrix}
V\\
W
\end{bmatrix}
=
\begin{bmatrix}
\tilde{\Gamma}_N^{\ve} (x)  U_1^N\\
P_N U_1^N
\end{bmatrix}
\quad
\mbox{and}
\quad
C=
\begin{bmatrix}
\tilde{\Gamma}_N^{\ve} (x)  \\
P_N
\end{bmatrix},
\quad
\mbox{then}
\quad
U_1^N
=
C^{-1 }
\begin{bmatrix}
V\\
W
\end{bmatrix}.
\]
Then, we have for every $G \in {\mathbb D}_{\infty}$ that
\begin{align}
\lefteqn{
{\mathbb E} [G \chi^{\ve}_N\cdot g_{\kappa} ( F_N (x, \Delta^N_{\ve} U^N_1) ) ]
}
\nn\\
&=
{\mathbb E} \Bigl[  G
\psi_N  \Bigl( 
\Bigl|
TC^{-1}
\begin{bmatrix}
V\\
W
\end{bmatrix}  
\Bigr|^2
\Bigr)
g_{\kappa} \Bigl( 
M_N \Bigl(x,   TC^{-1}
\begin{bmatrix}
V\\
W
\end{bmatrix}    \Bigr) 
B_{\cH} (x) 
\Delta^{\cH}_{\ve}
V 
\Bigr)
\Bigr]
\nn\\
&=
\int_{ {\mathbb R} \langle \cH \rangle  } dv
 \int_{{\mathbb R} \langle \cH \rangle^{\perp} } dw
\Bigl\langle
G, \delta_{
(v,w)
} 
\Bigl( 
\begin{bmatrix}
V\\
W
\end{bmatrix}
 \Bigr)  
\Bigr\rangle
\psi_N  \Bigl( 
\Bigl|
TC^{-1}
\begin{bmatrix}
v\\
w
\end{bmatrix}  
\Bigr|^2
\Bigr)
\nn\\
&
\qquad \qquad \times
g_{\kappa} \Bigl( 
M_N \Bigl(x,   TC^{-1}
\begin{bmatrix}
v\\
w
\end{bmatrix}    \Bigr) 
B_{\cH} (x) 
\Delta^{\cH}_{\ve}
v
\Bigr).
\nn
\end{align}
(Since it is difficult to put a column vector as a subscript of $\delta$,
we wrote $\delta_{(v,w)}$. )
We change variables by $v \mapsto (\Delta^{\cH}_{\ve} )^{-1} \kappa v$
and use Lemma \ref{lm.veLA} (3).
Then,
\begin{align}
\lefteqn{
{\mathbb E} [G\chi^{\ve}_N\cdot g_{\kappa} ( F_N (x, \Delta^N_{\ve} U^N_1) ) ]
}
\nn\\
&=
\ve^{-\nu}
\int_{ {\mathbb R} \langle \cH \rangle  } dv
 \int_{{\mathbb R} \langle \cH \rangle^{\perp} } dw
\Bigl\langle
G, \delta_{
(\Delta^{\cH}_{\ve} \kappa v,w)
} 
\Bigl( 
\begin{bmatrix}
V\\
W
\end{bmatrix}
 \Bigr)  
\Bigr\rangle
\nn
\\
& \qquad \times
\psi_N  \Bigl( 
\Bigl|
TC^{-1}
\begin{bmatrix}
\Delta^{\cH}_{\ve} \kappa v\\
w
\end{bmatrix}  
\Bigr|^2
\Bigr)
g \Bigl( 
M_N \Bigl(x,   TC^{-1}
\begin{bmatrix}
\Delta^{\cH}_{\ve} \kappa v\\
w
\end{bmatrix}    \Bigr) 
B_{\cH} (x) 
v
\Bigr).
\nn
\end{align}

Now, we use the dominated convergence theorem 
for $dv dw$-integration
as $\kappa \searrow 0$.
Due to (\ref{ineq.detM}), we can find a large constant $R >0$ 
independent of $\kappa$
such that the integrand above is dominated by 
$R \cdot {\bf 1}_{ \{ |v| <R,  |w| <R\} }$.
($R$ may depend on other parameters.)
Letting $\kappa \searrow 0$, we have
\begin{align}
{\mathbb E} [G\chi^{\ve}_N\cdot \delta_{0} ( F_N (x, \Delta^N_{\ve} U^N_1) ) ]
&=
\ve^{-\nu}
\int_{ {\mathbb R} \langle \cH \rangle  } dv
 \int_{{\mathbb R} \langle \cH \rangle^{\perp} } dw
\Bigl\langle
G, \delta_{
(0,w)
} 
\Bigl( 
\begin{bmatrix}
V\\
W
\end{bmatrix}
 \Bigr)  
\Bigr\rangle
\nn
\\
& \qquad \cdot
\psi_N  \Bigl( 
\Bigl|
TC^{-1}
\begin{bmatrix}
0\\
w
\end{bmatrix}  
\Bigr|^2
\Bigr)
g \Bigl( 
M_N \Bigl(x,   TC^{-1}
\begin{bmatrix}
0\\
w
\end{bmatrix}    \Bigr) 
B_{\cH} (x) 
v
\Bigr).
\nn
\end{align}

Changing variables again by 
$
v \mapsto 
\{ M_N \Bigl(x,   TC^{-1}
\begin{bmatrix}
0\\
w
\end{bmatrix}    \Bigr) 
B_{\cH} (x) 
\}^{-1}
v
$,
we have 
\begin{align}
{\mathbb E} [G\chi^{\ve}_N\cdot \delta_{0} ( F_N (x, \Delta^N_{\ve} U^N_1) ) ]
&=
\ve^{-\nu} |\det B_{\cH}(x) |^{-1}
\nn\\
&
\quad \times
\Bigl\langle
G,
 \frac{ \chi^{\ve}_N}
{  \det M_N (x, TU_1^N)}
 \int_{{\mathbb R} \langle \cH \rangle^{\perp} } dw
  \delta_{
(0,w)
} 
\Bigl( 
\begin{bmatrix}
V\\
W
\end{bmatrix}
\Bigr)
\Bigr\rangle.
 \nn
\end{align}
It is easy to see from Lemma \ref{lm.nondegG} that
\[
 \int_{{\mathbb R} \langle \cH \rangle^{\perp} } dw
  \delta_{
(0,w)
} 
\Bigl( 
\begin{bmatrix}
V\\
W
\end{bmatrix}
\Bigr)
=
\delta_0 (V).
\]
This completes the proof.
\end{proof}

\vspace{3mm}

Now we are in a position to
prove our main result in this section.
\begin{proof} [Proof of Theorem~\ref{tm.R^d}]
We expand the (generalized)
Wiener functionals on the right-hand side of (\ref{eq.key}).
First, note that 
\[
\Delta^N_{\ve} U^N_1 = \sum_{I \in \cG (N)} \ve^{|I|} U_1^I.
\]
This is just a polynomial in $\ve$ 
whose coefficients belong to an inhomogeneous Wiener chaos.

By the choice of $\psi$ and a routine argument, 
we have 
\begin{equation}\label{eq.asym1}
\chi^{\ve}_N  =
\psi ( |\Delta^N_{\ve} U^N_1 |^2 / (\kappa_N/2)^2 )
=
1 + O(\ve^{\infty})
\qquad
\mbox{in ${\mathbb D}_{\infty}$}
\end{equation}
as $\ve \searrow 0$.
Therefore, this is actually a dummy factor 
introduced for technical purposes
and makes no contribution to the asymptotic expansion.
Obviously, $x$ is not involved in this functional.

By the definition in (\ref{def.gntn1}) and (a comment after that),
$\tilde{\Gamma}_N^{\ve} (x) U_1^N$
is also a polynomial in $\ve$ that takes values in 
$\cL ({\frak g}^N ({\mathbb R}^d),  {\mathbb R} \la \cH\ra)$
whose coefficients belong to an inhomogeneous Wiener chaos.
(Moreover, it depends smoothly in $x$).
Since this is uniformly non-degenerate 
(see Lemma \ref{lm.nondegG}), 
we can use the standard version 
of Watanabe's theory (\ref{asymp.WAT}) to obtain the following 
asymptotic expansion:
\begin{equation}\label{eq.asym2}
\delta_0 ( \tilde{\Gamma}_N^{\ve} (x) U_1^N)
=
\cY^N_0 (x)+ \ve \cY^N_1(x)+  \ve^2 \cY^N_2 (x)+ \cdots
\qquad
\mbox{in $\tilde{\mathbb D}_{-\infty}$}
\end{equation}
as $\ve \searrow 0$.
Since Lemma \ref{lm.nondegG} claims 
uniform dependency in $x$, this expansion is uniform in $x \in O_N$.
By Lemma \ref{lm.veLA} (1) and a comment after
(\ref{def.gntn2}),
$\cY^N_0 (x)= \delta_0 (\tilde{\Gamma}_N^{0} (x) U_1^N ) 
=\delta_0 (\tilde{\Gamma}_{N_0}^{0} (x) U_1^{N_0} )$.

By the explicit definition of $M$ in (\ref{def.M}) and 
the uniform lower bound of $\det M$ in (\ref{ineq.detM}),
we also obtain the following 
asymptotic expansion uniformly in $x \in O_N$:
\begin{equation}\label{eq.asym3}
\det M_N (x, \Delta_{\ve}^N U_1^N)^{-1}
=
1 + \ve \cZ^N_1 (x)+  \ve^2 \cZ^N_2 (x)+ \cdots
\qquad
\mbox{in ${\mathbb D}_{\infty}$}
\end{equation}
as $\ve \searrow 0$.

Take $L>0$ arbitrarily large. 
We will show that $\delta_x (X^{\ve}(1,x))$ admits
an asymptotic expansion up to order $L$ as $\ve\searrow 0$.
For this $L$, we choose $N \ge N_0$ 
so that $l_N \ge L+\nu+1$.
Here, $\{l_N\}$ is the diverging sequence given in Lemma \ref{lm.Xaprx}.
We also take $O_N$ small enough so that all the previous 
results are available.

From
Lemma \ref{lm.Xaprx}, Lemma \ref{lm.hnki} and 
(\ref{eq.asym1})--(\ref{eq.asym3}),
we obtain the following 
asymptotics in $\tilde{\mathbb D}_{- \infty}$ 
as $\ve\searrow 0$ uniformly in $x \in O_N$ :
\begin{align}
\label{eq.asymp4}
\delta_x (X^{\ve}(1,x))
&=
 |\det B_{\cH}(x) |^{-1}
 \ve^{-\nu}
\\
&
\times  \bigl\{
 \delta_0 (\tilde{\Gamma}_{N_0}^{0} (x) U_1^{N_0} )
+ \ve  \Theta_1^N (x)
+\cdots + \ve^{ L+\nu}  \Theta_{  L+\nu}^N (x)
\bigr\}
+O (\ve^{L+1})
\nonumber
\end{align}
for some $\Theta_j^N (x) \in \tilde{\mathbb D}_{- \infty}
~(1 \le j \le L+\nu)$.
Since the coefficients of an asymptotic expansion 
are uniquely determined,
$\Theta_j^N (x)$ is actually independent of the choice of $N$.
This proves (\ref{asy.dX}).

By a routine argument, 
$\Theta_j^N (x; -w) = -\Theta_j^N (x; w)$ 
as a generalized Wiener functional if $j$ is odd.
This implies ${\mathbb E} [\Theta_j^N (x)]=0$ if $j$ is odd.

Finally, we show ${\mathbb E} [\delta_0 (\tilde{\Gamma}_{N_0}^{0} (x) U_1^{N_0} )] >0$.
Recall that $\tilde{\Gamma}_{N_0}^{0} (x) 
= [{\rm Id}_{{\mathbb R} \la \cH\ra} |\, *\,]$ 
is  a (possibly non-orthogonal) projection from 
${\frak g}^{N_0} ({\mathbb R}^d)$ to ${\mathbb R} \la \cH\ra$.
If we denote by $q^{N_0}$
 the smooth density of the law of $U_1^{N_0}$ on ${\frak g}^{N_0} ({\mathbb R}^d)$,
 then 
 \[
 {\mathbb E} [\delta_0 (\tilde{\Gamma}_{N_0}^{0} (x) U_1^{N_0} )]
 =
 K_x \int_{ \ker \tilde{\Gamma}_{N_0}^{0} (x)} q^{N_0} (u)du,
 \]
where $du$ is the Lebesgue  measure on the subspace 
$\ker \tilde{\Gamma}_{N_0}^{0} (x)$
and $K_x >0$ is a constant
which depends on  ``the angle" of the kernel and the image 
of the projection $ \tilde{\Gamma}_{N_0}^{0} (x)$. 
(If the projection is orthogonal, then $K_x =1$.)

Since the everywhere positivity of $q^{N_0}$ is shown 
in \cite[p. 202]{tak}
or originally in Kunita \cite{ku},
we have ${\mathbb E} [\delta_0 (\tilde{\Gamma}_{N_0}^{0} (x) U_1^{N_0} )] >0$
and  the proof of Theorem \ref{tm.R^d} is done.
\qedhere
 \end{proof}

\begin{remark}
There is another way to prove that $q^{N}$
is everywhere positive.
Let $u^N (h)$ be the solution of the skeleton ODE 
which corresponds to SDE (\ref{eq.sdeg^N}) driven by $h \in H$.
In other words, $u^N (h) =\log y^N (h)$.
It is sufficient to show that, for every $u \in {\frak g}^{N} ({\mathbb R}^d)$,
there exists $h \in H$ such that 
$u^N_1 (h)  =u$ and the tangent map
$D u^N_1 (h) \colon H \to  {\frak g}^{N} ({\mathbb R}^d)$
is surjective.
(See Aida-Kusuoka-Stroock \cite{aks} for example. See also Remark \ref{re.pol}.)

Such an $h$ can be found as follows. 
Take any Cameron-Martin path $k \colon [0,1/2] \to {\mathbb R}^n$
such that
$D u^N_{1/2} (k)$ is surjective (it does exist).
Since $G^N ({\mathbb R}^n) = \{ y^N_{1/2} (h) \mid h \in H\}$,
there exists 
a Cameron-Martin path $\hat{k} \colon [0,1/2] \to {\mathbb R}^n$
such that
$u^N_{1/2} (\hat{k}) = u^N_{1/2} (k)^{-1} \times u$.
Then, the concatenated path $k * \hat{k} \in H$ is the desired path.
Here, $k * \hat{k}$ is defined to be $k$ on $[0,1/2]$
and $\hat{k} ( \,\cdot\, -1/2) +k(1/2)$ on $[1/2,1]$.
Here, we used the left-invariance with respect to the product $\times$.
\end{remark}

 As a corollary  of Theorem \ref{tm.R^d}, 
we consider an SDE with drift instead of the driftless SDE (\ref{def.sde})
and 
prove an asymptotic expansion of the associated heat kernel
under an assumption that the drift vector field $V_0$ 
can be written as a linear combination of $V_1, \ldots, V_n$. 
It is important that the leading positive constant in the expansion
is independent of such $V_0$.

For $V_0, V_1, \ldots, V_n \in {\frak X}  ({\mathbb R}^d)$
and $0 < \ve \le 1$, we consider 
the following SDE on ${\mathbb R}^d$:
\begin{equation}\label{def.sde+dr}
d\hat{X}^{\ve} (t, x) 
= \ve \sum_{i=1}^n V_i (\hat{X}^{\ve} (t, x) ) \circ dw^i_t
+
\ve^2  V_0 (\hat{X}^{\ve} (t, x) ) dt
\qquad
\mbox{with $\hat{X}^{\ve} (0, x)=x$.}
\end{equation}
We continue to assume that  $V_i^j := \la dx^j, V_i\ra$ 
has bounded derivatives of all order $\ge 1$  ($0 \le i \le n, 1 \le j \le d$).
By the scaling property, 
$(\hat{X}^{\ve} (t, x))_{t \ge 0}$ and $(\hat{X}(\ve^2 t, x))_{t \ge 0}$
have the same law.
Here, we simply write $\hat{X} (t, x)$ for $\hat{X}^{\ve} (t, x)$ when $\ve =1$. 
When it exists, we denote by 
$p_t (x, x')$ 
the heat kernel associated with $\hat{X} (t, x)$,
which is
the density of the law of $\hat{X} (t, x)$ with respect to the Lebesgue measure.

\begin{corollary}\label{cor.Gir}
Let the notations be as above. 
Suppose  ${\bf (ER)}_{x_0}$ for $\{V_1, \ldots, V_n\}$
at $x_0 \in  {\mathbb R}^d$.
Suppose also that there exist 
smooth, bounded functions $a_1, \ldots, a_n \colon  {\mathbb R}^d \to  {\mathbb R}$ 
with bounded derivatives of all order 
which satisfy
 that $V_0 (x) = \sum_{i=1}^n a_i (x) V_i (x)$ for every $x \in {\mathbb R}^d$.

Then,  
there exists a decreasing sequence $\{ O_j \}_{j \ge 0}$ of
neighborhoods of $x_0$ 
such that the asymptotic expansion 
\[
p_t (x, x) 
\sim
t^{-\nu/ 2} \bigl( c_0 (x)+   c_1 (x)t
+  c_2 (x) t^2+\cdots
\bigr)
\qquad
\mbox{as $t \searrow 0$}
\]
holds for every $x \in  O_0$ with the following properties:
 {\rm (i)}~ $\inf_{x \in  O_0} c_0 (x)  >0$,
 {\rm (ii)}~
 for every $j \ge 0$, 
\[
\sup_{x \in  O_j} \Bigl\{  |c_j (x) | 
+\sup_{0 < t \le1} 
t^{( \nu/2) -j-1}
 \Bigl|
 p_t (x, x) 
-
t^{-\nu/ 2} \bigl( c_0 (x)
+\cdots + c_j (x)t^j
\bigr)
 \Bigr|
   \Bigr\}<\infty.
\]
Moreover, $c_0 (x) ={\mathbb E} [\Theta_{0} (x)]$ 
and hence is independent of $\{a_1, \ldots, a_n\}$.
\end{corollary}

\begin{proof}
We prove  the corollary by combining  the driftless case
(Theorem \ref{tm.R^d}) and Girsanov's theorem.
Set 
\[
M_t^{\ve, x} = \exp \Bigl(  
\ve \sum_{i=1}^n  \int_0^t  a_i ( X^{\ve}(s,x))  dw_s^i
-
\frac{\ve^2}{2} 
\sum_{i=1}^n  \int_0^t  |a_i ( X^{\ve}(s,x)) |^2  ds
\Bigr).
\]
Since $a_i~(1 \le i \le n)$ are bounded, $t \mapsto M_t^{\ve, x}$
is a true martingale.
By the scaling property of Brownian motion and
Girsanov's theorem, we have
\begin{eqnarray*}
 p_{\ve^2} (x, x) 
 &=&
 {\mathbb E} [ \delta_x  (\hat{X} (\ve^2 ,x) ) ] 
 =
 {\mathbb E} [ \delta_x  (  \hat{X}^{\ve}(1 ,x)  ) ] 
 =
 {\mathbb E} [ M_1^{\ve, x} \delta_x  (  X^{\ve}(1 ,x)  ) ].
   \end{eqnarray*}
Since we have already seen in Theorem \ref{tm.R^d} that 
$ \delta_x  (  X^{\ve}(1 ,x)  )$ admits an asymptotic expansion
 in $\tilde{\mathbb D}_{- \infty}$,
 it is sufficient to show that $M_1^{\ve, x}$ 
 admits an asymptotic expansion in ${\mathbb D}_{\infty}$ uniformly 
 in $x \in  O_0$.

By Proposition \ref{pr.ste}, $X^{\ve}(s ,x)$ admits an 
asymptotic expansion in ${\mathbb D}_{\infty}$
uniformly in $s \in [0,1]$ and $x \in  O_0$.
Moreover, each term in the expansion is measurable 
with respect to $\sigma ( w_u\mid 0 \le u \le s)$.
Therefore, 
$\sum_{i=1}^n  \int_0^1  a_i ( X^{\ve}(s,x))  dw_s^i$
admits an 
asymptotic expansion in ${\mathbb D}_{\infty}$
uniformly in $x \in  O_0$
and so does $\sum_{i=1}^n  \int_0^t  |a_i ( X^{\ve}(s,x)) |^2  ds$.
Since $a_i~(1 \le i \le n)$ and their derivatives are all bounded, 
we can easily see that
\begin{equation}\label{asy.Gir}
M_1^{\ve, x}
\sim 
1 + \ve \Xi_1 (x)  + \ve^2 \Xi_2 (x)  + \cdots
\qquad
\mbox{in ${\mathbb D}_{\infty}$ as $\ve\searrow 0$ 
 uniformly in $x \in  O_0$.}
\end{equation}
Moreover, $\Xi_{2j-1}(x; \,\cdot\, )$ is odd as a Wiener functional
for every $j \ge 1$ and $x \in  O_0$,
that is,  $\Xi_{2j-1} (x; -w) = - \Xi_{2j-1} (x; w)$.

By multiplying 
(\ref{asy.dX}) and (\ref{asy.Gir}) and taking  the generalized expectation,
 we have the desired expansion of $p_{\ve^2} (x, x)$.
Note that 
the odd-numbered terms in the expansion of 
$M_1^{\ve, x} \delta_x  (  X^{\ve}(1 ,x)  )$ are also odd
 as generalized 
 Wiener functionals and hence their generalized expectations vanish. 
 Note also that
 since the leading term on the right-hand side of (\ref{asy.Gir}) is $1$,
 $c_0 (x)$ does not depend on $a_i ~(1 \le i \le n)$.

It is a routine to check that $x \mapsto \Xi_j (x)$ is continuous.
By Remark \ref{re.cnt},
we can easily check the continuity of $c_j (x)$ in $x$.
Positivity of $c_0 (x)$ is immediate from Theorem \ref{tm.R^d}.
  \end{proof}

 \begin{remark}\label{re.hikaku}
Sections 3--5 basically follows its counterpart in \cite{tak}.
However, we believe that 
our argument here is simpler and more readable 
for the following reasons.
(i) Fortunately, it suffices to consider a driftless SDE \eqref{def.sde}
for our purpose.
Hence, we need not use the ``anisotropic dilation" 
on the tensor algebra.
This simplifies our notations much.
(ii) In \cite{tak} (originally in Yamato \cite{ya}) 
proofs of important properties of the free nilpotent groups/algebras
are done via
computations in the coordinates 
with respect to a linear basis $\cG (N)$.
This could be compared to doing all the differential 
geometric computation on a manifold via  local coordinates
and therefore
 does not provide a very clear view of what is going on.
In recent developments of rough path theory 
and numerical study of SDEs, 
study of the free nilpotent groups/algebras advanced much
(cf. e.g. \cite[Chapter 7]{fv} for a summary).
It provides us a clear view of these objects
and helps us simplify our argument.
In particular, 
proofs via the flow of ODEs on 
the nilpotent Lie groups/algebra in \cite{tak}
are replaced by (linear or Lie) algebraic proofs.
(iii) Some non-trivial facts on Malliavin calculus 
are presented without proofs in \cite{tak}.
We added proofs and explanations for non-experts.
 \end{remark}

\section{On sub-Riemannian manifolds}

Let $(M,\cD,g)$ be a sub-Riemannian manifold as in
Section~1; hence $d,n,\nu$ and $N_0$ are all as described there.
In this section we prove  the uniform asymptotic
expansion of the heat kernel on $M$ via localization method.
We emphasize that our arugument is  almost purely probabilistic.
Two key tools to achieve this goal are 
the stochastic parallel transport for the $\Delta /2$-diffusion process 
and 
Malliavin calculus for manifold-valued SDEs.
The stochastic parallel transport,
or the Eells-Elworthy-Malliavin method
of constructing   diffusion processes on a general sub-Riemannian 
manifolds was done in \cite{gt, thal}.
Methods in  these papers 
are slightly different and the latter is used in this paper.
Malliavin calculus for manifold-valued SDEs
was done in \cite{tan}.
It was shown there that a solution to an SDE at a fixed time
is non-degenerate in the sense of Malliavin
 under the {\it partial} H\"ormander condition on the coefficient vector fields
 of the SDE.

We shall define a ``div-grad type'' sub-Laplacian.
The horizontal gradient of $f\in C^\infty(M)$ is defined to
be the unique section  
$\nabla^{\cD} f\in C^\infty(M;\cD)$ such that  
\[
    g(\nabla^{\cD} f,A)=Af,\quad A\in C^\infty(M;\cD).
\]
Let $A_1,\dots,A_k$ be a local orthonormal frame for $\cD$, 
i.e., a family of local sections 
$A_1,\dots,A_n\in C^\infty(U;\cD)$ over an open set
$U\subset M$ with $g_x((A_i)_x,(A_j)_x)=\delta_{ij}$ for 
$x\in U$ and $1\le i,j\le n$. 
Then
\begin{equation}
    \nabla^{\cD} f=\sum_{i=1}^n (A_i f)A_i.
\end{equation}

Take a smooth measure $\mu$ on $M$.
For $A\in C^\infty(M;TM)$, define its $\mu$-divergence  
$\text{\rm div}_\mu A$ by 
\begin{equation}\label{eq.divergence}
    \int_M f (\text{\rm div}_\mu A) d\mu
    =-\int_M Af\,d\mu, \quad f\in C_0^\infty(M).
\end{equation}
The sub-Laplacian associated with a positive volume form $\mu$
is the second order differential operator given by
\[
    \Delta f
    =\text{\rm div}_\mu(\nabla^{\cD} f),
    \quad f\in C^\infty(M).
\]
In terms of a local orthonormal frame $A_1,\dots,A_n$
for $\cD$,
\begin{equation}\label{eq.delta.local}
    \Delta=\sum_{i=1}^n \Bigl(A_i^2
        +\text{\rm div}_\mu(A_i)A_i\Bigr).
\end{equation}
In what follows, for the sake of simplicity, we assume that
$M$ is compact.

The goal of this section is to show

\begin{theorem}\label{tm.subrie}
Let $(M,\cD,g)$ and $\mu$ be as above.
Then, the following hold.
\\
(i) There exists a diffusion process generated by
$\Delta/2$ and it possesses a transition
density function $p_t(x,y)$, which is smooth in 
$(t,x,y)\in(0,\infty)\times M\times M$, with respect to
$\mu$.
\\
(ii) Suppose $M$ is equiregular.
Then, the asymptotic expansion
\[
p_t (x, x) 
\sim
t^{-\nu/ 2} \bigl( c_0 (x)+   c_1 (x)t
+  c_2 (x) t^2+\cdots
\bigr)
\qquad
\mbox{as $t \searrow 0$}
\]
holds for every $x \in M$
with the following properties:
{\rm (a)}~ 
$\inf_{x \in M} c_0 (x)  >0$,
{\rm (b)}~for every $j \ge 0$, 
\[
\sup_{x \in M} \Bigl\{  |c_j (x) | 
+\sup_{0 < t \le1} 
t^{( \nu/2) -j-1}
 \Bigl|
 p_t (x, x) 
-
t^{-\nu/ 2} \bigl( c_0 (x)
+\cdots + t^j  c_j (x)
\bigr)
 \Bigr|
   \Bigr\}<\infty.
\]
\end{theorem}

We shall show the theorem by constructing the diffusion
process via the Eells-Elworthy-Malliavin method modified for
sub-Riemannian manifolds, and then applying Corollary
\ref{cor.Gir}.
It should be noted that the method
gives us strong solutions to stochastic differential
equations, which enable us to treat systematically the
assertions in the theorem together.
In fact, to construct diffusion process, a weak solution is 
enough; since $\Delta/2$ is smooth, the associated
martingale problem is well-posed, and hence the diffusion
process exists.
In this case, by \eqref{eq.delta.local} and H\"ormander's
theorem, one can prove the assertion
(i) in the theorem, but proving the 
short time asymptotics is another matter.  

Let
\[
    O(\cD)_x=\bigl\{
      u:\R^n\to\cD_x \big|
      \text{$u$ is a bijective linear isometry}\bigr\}
    \quad\text{and}\quad
    O(\cD)
    =\bigsqcup_{x\in M} O(\cD)_x
\]
and define $\pi \colon O(\cD)\to M$ by $\pi(u)=x$ for 
$u\in O(\cD)_x$, $x\in M$.
Then, 
$\pi \colon O(\cD)\to M$ is an $O(n)$-principal bundle, where
$O(n)$ is the space of $n\times n$ orthogonal matrices. 
To apply the Eells-Elworthy-Malliavin method to 
a sub-Riemannian manifold, we first recall 
the horizontal
vector fields on $O(\cD)$ (cf. \cite{thal}).
To do this,  let $\nabla$ be a partial metric connection on
$(M,\cD,g)$; that is, $\nabla$ is a bilinear mapping
\[
    \nabla \colon C^\infty(M;\cD)\times C^\infty(M;\cD)
    \ni (A,B)\mapsto \nabla_AB\in
    C^\infty(M;\cD)
\]
such that $\nabla_A(fB)=f\nabla_AB+(Af)B$ for 
$f\in C^\infty(M)$ and $\nabla g=0$, where 
$(\nabla_A g)(B,C):=\nabla_A g(B,C)-g(\nabla_A B,C)
     -g(B,\nabla_A C)$ for $A,B,C\in C^\infty(M;\cD)$. 
A typical example of partial metric connections is a
restriction of Levi-Civita connection.
In fact, 
let $\tilde{g}$ be a Riemannian metric tensor on $M$
and $\tilde{\nabla}$ be its Levi-Civita connection.
If $\tilde{g}$ tames $g$,
i.e., $\tilde{g}|_{\cD\times\cD}=g$, then
$\nabla_A B=\text{\rm pr}_{\cD}\tilde{\nabla}_AB$,
$\text{\rm pr}_{\cD}$ being the projection onto $\cD$,
is a partial metric connection.

In terms of a local orthonormal frame $A_1,\dots,A_n$ for
$\cD$, define $\omega_i^j\in C^\infty(M;\cD^*)$,
where $\cD^*$ is the dual subbundle of $\cD$, by 
\[
    \nabla A_i=\sum_{j=1}^n \omega_i^j A_j,
    \quad\text{i.e., }
    \nabla_B A_i=\sum_{j=1}^n \omega_i^j(B) A_j
    \quad\text{for $i=1,\dots,n$ 
               and $B\in C^\infty(M;\cD)$}.
\]
Since $\nabla g=0$, $\omega_j^i=-\omega_i^j$, 
$1\le i,j\le n$.
We now extend the partial connection form
$\omega=(\omega_i^j)$ to a smooth partial $1$-form on
$O(\cD)$ with values in the Lie algebra $\mathfrak{o}(n)$ of
$O(n)$, say $\omega$ again, given by 
\[
    s^{-1}\omega s+s^{-1}ds,
\]
where we have used the local trivialization $M\times O(n)$
of $O(\cD)$ and then $s$ is the coordinate of $O(n)$;
more precisely,
$s^{-1}ds$ is the Maurer-Cartan form $\theta$ given by
$\theta(X)=s^{-1}X$ for $X\in T_sO(n)$ and $s\in O(n)$.

Define the holizontal subspace
$K_u\subset T_uO(\cD)$, $u\in O(\cD)$, by
\[
    K_u=\{A\in T_uO(\cD)\,|\,
         (\pi_*)_uA\in \cD_{\pi(u)},\omega_u(A_u)=0\}.
\]
In terms of a local orthonormal frame $A_1,\dots,A_n$ for
$\cD$, it holds
\[
    K_u=\Bigl\{\sum_{\alpha=1}^n a^\alpha A_\alpha
    -\sum_{p,q,r,s=1}^n \omega^p_{qr} e^q_s a^r
    \frac{\partial}{\partial e^p_s}\,\Big|\,
    (a^1,\dots,a^n)\in \mathbb{R}^n \Bigr\},
\]
where
\[
    \omega^p_{qr}=\omega_q^p(A_r)
    =g(\nabla_{A_r}A_q,A_p), \quad 1\le p,q,r\le n,
\]
and $(e^p_q)_{1\le p,q\le n}$ stands for the matrix
coordinate of $O(n)$.
Then the horizontal lift
$\ell_u:\cD_{\pi(u)}\to K_u$ defined by
\[
    \ell_u\biggl(\sum_{i=1}^n a^i A_i\biggr)
    =\sum_{i=1}^n a^i A_i
    -\sum_{p,q,r,s=1}^n \omega^p_{qr} e^q_s a^r
    \frac{\partial}{\partial e^p_s}
\]
is bijective.

Let $\{{\mathbf e}_i\mid 1\le i\le n\}$
be the canonical basis of $\R^n$.
Define the canonical horizontal vector
fields $V_1,\dots,V_n$ on $O(\cD)$ by 
\[
    (L_i)_u=\ell_u(u{\mathbf e}_i),\quad 1\le i\le n.
\]
In terms of an orthonormal frame $A_1,\dots,A_n$ for $\cD$,
it holds
\begin{equation}\label{eq.hor.v.f.local}
    L_i=\sum_{j=1}^n e_i^j A_j
    -\sum_{p,q,r,s=1}^n \omega^p_{qr} e^q_s e^r_i
    \frac{\partial}{\partial e^p_s},\quad 1\le i\le n.
\end{equation}

The following lemma asserts that the operator
$(1/2)\sum_{i=1}^n L_i^2$ corresponds to another
sub-Laplacian $\Delta^\prime$ given by
\begin{equation}\label{def.lap'}
    \Delta^\prime f=\text{\rm tr}\,\nabla df,
    \quad f\in C^\infty(M),
\end{equation}
where 
the Hessian $\nabla df$ is given by 
\begin{equation}\label{eq.nabla.df}
    (\nabla df)(A,B)=ABf-(\nabla_AB)f,
    \quad A,B\in C^\infty(M;\cD),
\end{equation}
and $\text{\rm tr} \nabla df$ is the trace at each point in
$M$ of the bilinear form  $(A,B)\mapsto (\nabla df)(A,B)$.

\begin{lemma}\label{lem.another.sublap}
For $f\in C^\infty(M)$, set $\widetilde{f}=f\circ\pi$.
Then
\[
    L_i L_j\widetilde{f}(u)
    =(\nabla df)_{\pi (u)} (u{\mathbf e}_i,u{\mathbf e}_j),
    \quad u\in O(\cD), 1\le i,j\le n.
\]
In particular,
\[
    \frac12\sum_{i=1}^n L_i^2 \widetilde{f}
    =\frac12 (\Delta^\prime f)\circ\pi.
\]
\end{lemma}

\begin{proof}
By \eqref{eq.hor.v.f.local}, we have
\[
    L_j \widetilde{f}
    =\sum_{p=1}^n e^p_j \widetilde{A_p f}.
\]
Hence
\[
    L_i L_j \widetilde{f}
    =\sum_{p,q=1}^n e^q_i e^p_j
         \widetilde{A_q A_p f}
    -\sum_{p,q,r=1}^n \omega^p_{qr}
        e^q_j e^r_i \widetilde{A_p f}.
\]
Since
\[
    \nabla_{u{\mathbf e}_i}(u{\mathbf e}_j)
    =\nabla_{\sum_{p=1}^n e^p_i A_p}
     \Bigl(\sum_{q=1}^n e^q_j A_q\Bigr)
    =\sum_{p,q,r=1}^n e^{r}_i e^q_j
    \omega^p_{qr} A_p, 
\]
we obtain the first identity by \eqref{eq.nabla.df}.
The second identity immediately follows from the first one.
\end{proof}

We are now ready to prove Theorem~\ref{tm.subrie}.

\begin{proof}[Proof of Theorem~\ref{tm.subrie}]
(i) 
Extend the metric $g$ to a Riemannian metric $\tilde{g}$ on
$M$.
Let $\tilde\nabla$ be the Levi-Civita connection associated
with $\tilde{g}$, and define the partial metric connection
$\nabla$ by 
$\nabla_A B=\text{\rm pr}_{\cD}\tilde{\nabla}_AB$.
Denote by $\Delta^\prime$ the sub-Laplacian on $M$ given
by \eqref{def.lap'}.

In a local orthonormal frame $A_1,\dots,A_n$ for $\cD$, it
holds
\[
    \Delta^\prime
    =\sum_{i=1}^n A_i^2 
           -\sum_{i=1}^n \Bigl(\sum_{j=1}^n 
             \omega_j^i (A_j)\Bigr)A_i.
\]
Hence $N=(\Delta-\Delta^\prime)/2$ satisfies
\[
     N=\frac12\sum_{i=1}^n \Bigl\{
     \text{\rm div}_\mu A_i+\sum_{j=1}^n
      \omega_j^i(A_j) \Bigr\}A_i.
\]
In particular, $N\in C^\infty(M;\cD)$.

Set
\[
    L_0(u)=\ell_u(N_{\pi(u)}),\quad u\in O(\cD).
\]
Let $(r(t))_{t\ge0}$ be the unique solution to the
stochastic differential equation on $O(\cD)$
\[
    dr(t)=\sum_{i=1}^n L_i(r(t))\circ dw^i(t)
       +L_0(r(t))dt,
    \quad r(0)=u\in O(\cD),
\]
and put $X(t)=\pi(r(t))$, $t\ge0$.
Since $\pi_*L_0=N$, by Lemma \ref{lem.another.sublap},
the projected process $(X(t))_{t\ge0}$ solves the 
$\Delta/2$-martingale problem, i.e.,
\[
     \Bigl(
     f(X(t))-\int_0^t \frac12\Delta f(X(s))ds
     \Bigr)_{t\ge0}
\]
is a martingale for any $f\in C^\infty(M)$.
Thus the $\Delta/2$-diffusion process is realized
as the projected process
$(X(t))_{t\ge0}$.  
In particular, the law of $(X(t))_{t\ge0}$ is independent of
the choice of $u\in\pi^{-1}(x)$.

By the H\"ormander condition at every $x\in M$, 
$\{L_1, \ldots, L_n\}$ satisfies the partial H\"ormander 
condition at every $u\in O(\cD)$, 
that is, 
 the linear span of
$\{(\pi_*)_u L_{[I]} (u) \mid I\in \cI(\infty)\}$
is equal to  $T_{\pi(u)}M$.
Then, we know from \cite{tan} 
the non-degeneracy of the Malliavin covariance of $X(t)$ 
for all $t>0$ and $r(0)=u\in O(\cD)$.
By the integration by parts formula for manifold-valued
Wiener functionals, 
we obtain the existence of $p_t(x,y)$ 
and  the smoothness in $y\in M$ (cf. \cite{in-tan}).
The smoothness in $(t,x)\in(0,\infty)\times M$ is obtained
as an application of It\^o's formula and the stochastic flow
property of $(r(t) )_{t\ge0}$.
By the way,
the heat kernel admits the following 
explicit expression as in the Euclidean case:
\[
p_t(x,y) = {\mathbb E} [ \delta_y ( X(t))],
\]
where $\delta_y$ is the delta function 
at $y$ with respect to $\mu$.

\noindent
(ii)
Since $M$ is compact, it suffices to show that
for each $x_0\in M$, 
there exists a decreasing sequence $\{ O_j \}_{j \ge 0}$ of
neighborhoods of $x_0$ 
such that the asymptotic expansion 
\[
p_t (x, x) 
\sim
t^{-\nu/ 2} \bigl( c_0 (x)+   c_1 (x)t
+  c_2 (x) t^2+\cdots
\bigr)
\qquad
\mbox{as $t \searrow 0$}
\]
holds for every $x \in O_0$ with the following properties:
(a)~ $\inf_{x \in O_0} c_0 (x)  >0$,
(b)~for every $j \ge 0$, 
\[
\sup_{x \in O_j} \Bigl\{  |c_j (x) | 
+\sup_{0 < t \le1} 
t^{( \nu/2) -j-1}
 \Bigl|
 p_t (x, x) 
-
t^{-\nu/ 2} \bigl( c_0 (x)
+\cdots + t^j  c_j (x)
\bigr)
 \Bigr|
   \Bigr\}<\infty.
\]

To do this, let $U_1$ and $U_2$ be open sets in $M$ such
that $x_0\in U_1$, $\overline{U_1}\subset U_2$ and there
exists a local orthonormal frame $A_1,\dots,A_n$ for $\cD$
over $U_2$.
Viewing $U_2$ as a part of $\R^d$, we extends $A_1,\dots,A_n$
on $U_2$ to $C_b^\infty$-vector fields $V_1,\dots,V_n$ on 
$\mathbb{R}^d$, respectively, and extend each
$(\text{\rm div}_\mu A_i)/2$ on $U_2$ 
to $a_i\in C_b^\infty(\R^d)$.
Let $\tilde{p}_t(x,y)$ be the heat kernel with respect to
the Lebesgue measure on $\R^d$ associated with
\[
    \frac12\sum_{i=1}^n V_i^2
    +\sum_{i=1}^n a_i V_i. 
\]
Denote by $(\tilde{X}(t,x))_{t\ge0}$ the solution to the
SDE
\[
    d\tilde{X}(t)=\sum_{i=1}^n V_i(\tilde{X}(t))\circ dw^i
    +V_0(\tilde{X}(t))dt,
     \quad \tilde{X}(0)=x,
\]
where $V_0=\sum_{i=1}^n a_iV_i$.
Then $\tilde{p}_t(x,y)$ is the transition density function
of $(\tilde{X}(t,x) )_{t\ge0}$
with respect to the Lebesgue measure.

In repetition of the argument employed to show the
estimation (10.57) in \cite[p.421]{iwbk}, we obtain 
positive constants $c_1$ and $c_2$ such that
\begin{equation}\label{ineq.lcl}
    \sup_{x,y\in U_1}|\rho(y)p_t(x,y)-\tilde{p}_t(x,y)|
    \le c_1 e^{- c_2/t}
    \quad\text{as }t\searrow 0,
\end{equation}
where $d\mu(y)=\rho(y)dy^1\dots dy^d$, $(y^1,\dots,y^d)$
is the local coordinates on $U_1$ identified with 
that on $\R^d$.

Indeed, we can show \eqref{ineq.lcl} by combining the
following two observations:
(i) 
$u_f(t,x) :=\int_{U_1}
 \{\rho(y)p_t(x,y)-\tilde{p}_t(x,y)\}f(y)dy$, 
where $f\in C^\infty(\mathbb{R}^d)$ whose 
support is contained in $U_1$, 
satisfies the estimation  
$|u_f(t,x)|\le \sup_{s\in [0,t], z\in\partial U_2}|u_f(s,z)|$,
$\partial U_2$ being the boundary of $U_2$.
(ii) 
There exist positive constants $c_3$ and $c_4$ such that 
\begin{equation}\label{eq.hc}
p_s(z,y) \vee \tilde{p}_s(z,y)
\le c_3 e^{-c_4/s},
\qquad 
y\in U_1, z\in\partial U_2, s\in (0,1].
\end{equation}
A rough sketch of proof of \eqref{eq.hc} is as follows. 
The non-degeneracy of the Malliavin covariances of
$X(s)$ and  $\tilde{X}(s)$ under the (partial) 
H\"ormander condition
enables us to use  the integration
by parts formula. 
So we can replace the delta functions 
in the Feynman-Kac type representation formulae
for the heat kernels
by
continuous functions as in the proof of Lemma~\ref{lm.hnki}.
Then, the exponential decay of exit times
of semimartingales like \eqref{ineq.exitU} and
Kusuoka-Stroock's estimate like Proposition \ref{pr.KS}
for both $X$ and $\tilde{X}$ imply \eqref{eq.hc}.

It immediately follows from the two observations that
$|u_f(t,x)| \le c_3 (1 +\| \rho\|_{\infty}) \|f\|_{L^1}  e^{-c_4/t}$
for every $t \in (0,1]$ and $x \in U_1$.
Letting $f$ tend to $\pm \delta_y~(y \in U_1)$,
we prove \eqref{ineq.lcl}.

One should note that the equiregular condition has not been used so far.
Once \eqref{ineq.lcl} is obtained,
the desired asymptotic expansion of $p_t(x,x)$ follows from
that of $\tilde{p}_t(x,x)$.
Thus,
the assertion (ii) follows by applying
Corollary~\ref{cor.Gir} to $\tilde{p}_t(x,x)$.
\end{proof}

\section{Leading constant: Examples}

From the viewpoint of spectral geometry, 
it is very important to obtain an explicit expression 
of the leading constant of the asymptotic expansion of the heat trace.
However, it seems quite difficult in general.
Therefore, in this section 
we provide some examples for which the leading constant 
is explicitly computable by our method
and we check that these leading constants coincide 
with known results.

Since we have already shown in Theorem \ref{tm.subrie} that 
the asymptotic expansion of the heat kernel is uniform 
in the space parameter $x$,  
we may compute  the leading term of the asymptotics of $p_t (x,x)$ 
in the most convenient way for  each fixed $x \in M$.

We recall symbols and notations which will be used 
in subsequent examples.
The dimension of the manifold $M$ is $d$
and the number of independent linear Brownian motion is $n$.
For a given set of vector fields $\{ V_i \mid 1 \le i \le n\}$,
which are actually the coefficients of the corresponding SDE, 
$N_0$ stands for the step of the equiregular H\"ormander condition.
Matrices $B_{\cH}  (x)$ and  $B_{N_0} (x)$ 
are defined in (\ref{def.BcH}) and (\ref{def.B}), respectively.
 Recall also that
$\Gamma_{N_0} (x) 
=( \gamma^I_J (x))_{I \in \cH, J \in \cG (N_0) }$
and 
$\tilde{\Gamma}_{N_0}^{0} (x)
= \bigl( 
\delta_{|J|}^{|I|}  \gamma^I_J (x)
\bigr)_{I \in \cH, J \in \cG (N_0)}$ which is defined in (\ref{def.gntn2}).
The leading constant of $p_t (x,x)$ in the Euclidian case 
was shown in (\ref{eq.asymp4}) to be
\[
|\det B_{\cH}  (x)|^{-1}
{\mathbb E} [\delta_0 (\tilde{\Gamma}_{N_0}^{0} (x) U_1^{N_0} ) ].
\]
Here, $( U_t^{N_0})_{t \ge 0}$ is the ${\frak g}^{N_0} ({\mathbb R}^n)$-valued  
hypoelliptic diffusion process introduced in (\ref{eq.sdeg^N}).
In the manifold case, this constant should be adjusted 
by being divided by the density function as we discussed in (\ref{ineq.lcl}).

\medskip

\begin{example}~
(The case of Riemannian manifold)~
Let $M$ be a compact Riemannian manifold of dimension $d$
with the Riemannian measure $\mu$.
The div-grad type operator is the usual Laplace-Beltrami operator $\Delta_M$.
In this case $N_0 =1$, $d=n$,
$\cG (1) = \cH =\{ (i) \mid 1 \le i \le d \}$.

Take a coordinate chart $(x^1, \ldots, x^d)$. 
We denote the metric tensor by $G (x): = (g_{ij} (x) )_{1 \le i,j \le d}$.
Then,  $\mu (dx) =\sqrt{ \det G (x)}  dx^1 \cdots dx^d$ on this chart.
We write $G (x)^{-1/2} =  (\sigma^{ij} (x) )_{1 \le i,j \le d}$
and set 
$V_i (x) = \sum_j \sigma^{ij} (x) (\partial / \partial x^j)$
so that 
$\{ V_i  \mid 1 \le i \le d \}$ is a local orthonormal frame.
The Laplace-Beltrami operator is expressed as 
$
\Delta_M = \sum_{i=1}^d V_i^2 + \mbox{(a vector field).} 
$

It is easy to see that $B_{\cH}  (x) = B_{N_0} (x) = G (x)^{-1/2}$,
 $\Gamma_{N_0} (x) = \tilde{\Gamma}_{N_0}^{0} (x) = {\rm Id}$.
Hence, by adjusting the density function of $\mu$ as in (\ref{ineq.lcl}),
we see that the leading constant 
in the asymptotics of $p_t (x,x)$
equals 
\[
\frac{1}{\sqrt{ \det G (x)} }
|\det B_{\cH}  (x)|^{-1}
{\mathbb E} [\delta_0 (\tilde{\Gamma}_{N_0}^{0} (x) U_1^{N_0} ) ]
=
{\mathbb E} [\delta_0 (w^1_1, \ldots, w^d_1) ]
=
(2\pi)^{-d /2}.
\]
In particular,
we see that ${\rm Trace} (e^{ -t \Delta_M/2}) \sim (2\pi t)^{-d /2} \mu (M)$
as $t \searrow 0$.
Thus, we have recovered the well-known result in Riemannian geometry.
\end{example}

\medskip

\begin{example}\label{example.3D}~
(The case of $3D$ contact sub-Riemmanian manifold)~
In this example, we calculate the leading constant for
a three-dimensional contact sub-Riemmanian manifold and 
check that it coincides with Barilari's result in \cite{ba}.

Let $M$ be a 
compact sub-Riemmanian 
 manifold with $\dim M =3$
with a distribution $\cH$ of rank $2$.
We assume that $\cK$ is contact, namely, 
there exists a one-form $\omega$ such that $\omega \wedge d\omega$
vanishes nowhere.
As a volume on $M$, we choose the following measure.
Let $\{V_1, V_2\}$  be a local orthonormal frame of $\cK$
on a coordinate chart
and regard
$\lambda_1\wedge \lambda_2 \wedge \lambda_3$
as a measure on the chart,
where $\{\lambda_1, \lambda_2, \lambda_3\}$ is the dual basis 
of $\{V_1, V_2, [V_1, V_2]\}$.
This defines a measure $\mu$ on $M$. 
Note that $\mu$ is a constant multiple of Popp's measure 
if we use the definition (or results) in \cite{br}.
In this case $N_0 =2$, $d=3$, $n=2$,
$\cG (2) = \cH =\{ (1),  (2),  (2,1) \}$
and the Hausdorff dimension is $\nu=4$.

We use the normal coordinates for three-dimensional contact manifolds 
in the same way as in \cite{ba}.
For every $x \in M$, we can find a local coordinate chart 
$(u^1, u^2, u^3)$ 
and a local orthonormal frame $\{V_1, V_2\}$ of $\cK$ on this chart
such that 
$x$ corresponds to $0 \in {\mathbb R}^3$ and 
\begin{eqnarray*}
V_1 (u^1, u^2, u^3) &=&  
\Bigl(  \frac{\partial }{\partial u^1} + \frac{u^2}{2} \frac{\partial }{\partial u^3}    \Bigr)
+
\beta u^2 \Bigl(
u^2 \frac{\partial }{\partial u^1} - u^1 \frac{\partial }{\partial u^2}
\Bigr)
+
\gamma u^2 
\frac{\partial }{\partial u^3},
\\
V_2 (u^1, u^2, u^3) &=&
\Bigl(  \frac{\partial }{\partial u^2} - \frac{u^1}{2} \frac{\partial }{\partial u^3}    \Bigr)
-
\beta u^1 \Bigl(
u^2 \frac{\partial }{\partial u^1} - u^1 \frac{\partial }{\partial u^2}
\Bigr)
+
\gamma u^1 
\frac{\partial }{\partial u^3},
\end{eqnarray*}
where $\beta = \beta (u^1, u^2, u^3)$
and $\gamma = \gamma (u^1, u^2, u^3)$ are certain smooth functions
which vanish at $0$.
The sub-Laplacian can be written locally as $\Delta = V_1^2 +V_2^2 +
\mbox{(a section of $\cK$)}$.

From these explicit expressions, we can easily see that the density 
$\rho := d \mu/ du^1du^2du^3$ satisfies $\rho (0) = 1$.
Moreover, $B_{\cH}  (0) = B_{N_0} (0) = {\rm Id}$
and
$\Gamma_{N_0} (0) = \tilde{\Gamma}_{N_0}^{0} (0) = {\rm Id}$.
Hence,  the leading constant 
in the asymptotics of $p_t (x_0, x_0)$ associated with $\Delta /2$
equals 
\begin{equation}\label{top3D}
\frac{1}{\rho (0)}
|\det B_{\cH}  (0)|^{-1}
{\mathbb E} [\delta_0 (\tilde{\Gamma}_{N_0}^{0} (0) U_1^{N_0} ) ]
=
{\mathbb E} [\delta_{(0,0,0)} (w^1_1, w^2_1, S_1 (w^1,w^2)) ],
\end{equation}
where 
\begin{equation}\label{def.levy}
S_t (w^1,w^2) 
=
\frac12
\int_0^t  ( w^1_s  dw^2_s - w^2_s  dw^1_s)
=
\frac12
\int_0^t  ( w^1_s \circ dw^2_s - w^2_s \circ dw^1_s)
\end{equation}
is Levy's stochastic area of the two-dimensional Brownian motion. 
A well-known formula for Levy's stochastic area
(e.g. \cite[Theorem 5.8.5, p. 272]{mt}) 
states that 
\begin{equation}\label{fml.levy}
{\mathbb E} [ \exp (\sqrt{-1} \lambda  S_1 (w^1,w^2))  
\delta_{(0,0)} (w^1_1, w^2_1) ]
=
\frac{1}{2\pi} \frac{\lambda/2}{ \sinh ( \lambda/2)}
\qquad
(\lambda \in {\mathbb R}).
\end{equation}
Then, we see that the right-hand side on (\ref{top3D}) equals  
\begin{align*}
{\mathbb E} [\delta_{(0,0)} (w^1_1, w^2_1) \delta_0 (S_1 (w^1,w^2)) ]
&=
{\mathbb E} \Bigl[ 
\delta_{(0,0)} (w^1_1, w^2_1)
\frac{1}{2\pi}
\int_{{\mathbb R} }   \exp (\sqrt{-1} \lambda  S_1 (w^1,w^2))   d\lambda
\Bigr]
\\
&=
\frac{1}{(2\pi)^2} \int_{{\mathbb R} }
\frac{\lambda/2}{ \sinh ( \lambda/2)}
d\lambda
=\frac14.
\end{align*}
For a proof of the last equality, see \cite[Lemma A.2, p. 260]{bfi1}.
This constant $1/4$ coincides with one in Theorem 1, \cite{ba}.
(We need to replace $t$ in \cite{ba} by $t/2$ since
the heat kernel in \cite{ba} is associated with $\Delta$, 
not $\Delta /2$.)
In particular,
we see that ${\rm Trace} (e^{ -t \Delta /2}) \sim  \mu (M)/4 t^2$
as $t \searrow 0$.
\end{example}

\medskip

\begin{example}~
To state the next example, we review strictly
pseudoconvex CR manifolds.
For details, see \cite{dt}.

Let $M$ be a CR-manifold, i.e., $M$ is a real smooth
manifold with a complex subbundle $T_{1,0}$ of the
complexified tangent bundle $\mathbb{C} TM$ such that 
\[
  T_{1,0}\cap T_{0,1}=\{0\}\quad\text{and}\quad
  [T_{1,0},T_{1,0}]\subset T_{1,0},
\]
where $T_{0,1}=\overline{T_{1,0}}$.
Suppose that the real dimension of $M$ is $2k+1$ and the
complex dimension of $T_{1,0}$ is $k$ 
($k\ge1$).

There exists a real nowhere vanishing $1$-form $\theta$ which
annihilates 
$\cD:=\text{\rm Re}(T_{1,0}\oplus T_{0,1})$.
The associated Levi form $L_\theta$ is defined by 
\[
    L_\theta(Z,W)=-\kyosu \,d\theta(Z,W),\quad
    Z,W\in C^\infty(M;T_{1,0}\oplus T_{0,1}).
\]
We assume that $M$ is strictly pseudoconvex, i.e., $L_\theta$ 
is positive definite.

There exists a unique real vector field $T$, called the
characteristic direction, such that
\[
    \theta(T)=1,\quad T\rfloor d\theta=0,
\]
where $T\rfloor$ stands for the interior product by $T$.
The Webster metric $g_\theta$ on $TM=\cD\oplus\mathbb{R} T$ 
is defined by
\[
    g_\theta(X,Y)=d\theta(X,JY),\quad
    g_\theta(X,T)=0,\quad
    g_\theta(T,T)=1\quad\text{for }X,Y\in 
    C^\infty(M;\cD),
\]
where $J$ is a linear mapping on $T_{1,0}\oplus T_{0,1}$
such that $J|_{T_{1,0}}=\kyosu $ and
$J|_{T_{0,1}}=-\kyosu $.

In this example, let $\mu$ be the Riemannian volume measure
associated with $g_\theta$ and consider $\Delta$
associated with this $\mu$.
It should be noted that $\mu$ is a constant multiple of 
Popp's measure (cf. \cite{br}). 
Moreover, 
$\Delta$ is the standard sub-Laplacian on a CR-manifold, and
coincides with $\Delta^\prime$, which is constructed by
using the Tanaka-Webster connection on $M$ 
(\cite[Section 2.1]{dt}).

To compute locally around fixed $x\in M$, 
we introduce the Folland-Stein normal
coordinates, following \cite[Section 3.2]{dt}.
Let $T_1,\dots,T_n$ be a local orthonormal frame on an open
set $U\subset M$ for $T_{1,0}$ with respect to $L_\theta$, i.e.,
(i) $T_\alpha\in C^\infty(U;T_{1,0})$ and 
(ii) $L_\theta(T_\alpha,T_{\overline{\beta}})
  =\delta_{\alpha\beta}$ 
for  $1\le \alpha,\beta\le k$, where
$T_{\overline{\beta}}=\overline{T_\beta}$.
Set $X_\alpha=T_\alpha+T_{\overline{\alpha}}$ and
$Y_\alpha=\kyosu(T_{\overline{\alpha}}-T_\alpha)$.
Then 
\[
    g_\theta(X_\alpha,X_\beta)=g_\theta(Y_\alpha,Y_\beta)
    =2\delta_{\alpha \beta },\quad
    g_\theta(X_\alpha,Y_\beta)=0\quad
    \text{for }1\le \alpha,\beta \le k.
\]
There exists a coordinate chart
$u=(u^1,\dots,u^{2k+1})$, called {\it the
Folland-Stein normal coordinates}, such that $x$ corresponds
to $0\in\R^{2k+1}$, and 
\begin{equation}\label{eq.f.s.coord}
\left\{
\begin{aligned}
    X_\alpha & =\frac{\partial}{\partial u^{2\alpha-1}}
     +2u^{2\alpha}\frac{\partial}{\partial u^{2k+1}}
     +\sum_{i=1}^{2k} O^1 \frac{\partial}{\partial u^i}
       +O^2 \frac{\partial}{\partial u^{2k+1}},
    \\
    Y_\alpha & =\frac{\partial}{\partial u^{2\alpha}}
     -2u^{2\alpha-1}\frac{\partial}{\partial u^{2k+1}}
     +\sum_{i=1}^{2k} O^1 \frac{\partial}{\partial u^i}
       +O^2 \frac{\partial}{\partial u^{2k+1}},
    \\
    T & =\frac{\partial}{\partial u^{2k+1}}
     +\sum_{i=1}^{2k} O^1 \frac{\partial}{\partial u^i}
      +O^2 \frac{\partial}{\partial u^{2k+1}},
\end{aligned}
\right.
\end{equation}
where $O^j$, 
$j=1,2$, 
stand for functions with the property that
\[
    O^j=O\biggl(\Bigl(\sum_{i=1}^{2k} |u^i|\Bigr)^j
      +|u^{2k+1}|^{j/2}\biggr).
\]

By \eqref{eq.f.s.coord},
\begin{align*}
    & g_\theta\biggl(
     \Bigl(\frac{\partial}{\partial u^i}\Bigr)_u,
     \Bigl(\frac{\partial}{\partial u^j}\Bigr)_u
    \biggr)=2\delta_{ij}+O^1,\quad 1\le i,j\le 2k,
    \\
    & g_\theta\biggl(
     \Bigl(\frac{\partial}{\partial u^p}\Bigr)_u,
     \Bigl(\frac{\partial}{\partial u^{2k+1}}\Bigr)_u
    \biggr)=\delta_{p,2k+1}+O^1,\quad 1\le p\le 2k+1.    
\end{align*}
Thus $\mu(du)=(2^k+O^1) du^1\cdots du^{2k+1}$.
In particular, the density $\rho=d\mu/du^1\cdots du^{2k+1}$
satisfies $\rho(0)=2^k$.

Let
\[
    V_{2 \alpha-1}=\frac1{\sqrt{2}}X_\alpha,\quad
    V_{2 \alpha}=\frac1{\sqrt{2}}Y_\alpha,
    \quad 1\le \alpha \le k,
\]
where, as in the proof of Theorem~\ref{tm.subrie}, we have
extended $X_\alpha$ and $Y_\alpha$, $1\le \alpha\le k$, to
$\R^{2k+1}$, and used the same letters to indicate the
extensions.
Then what we need to investigate is the transition density
function of the 
diffusion process generated by
\[
    \frac12\sum_{i=1}^{2k}V_i^2+\sum_{i=1}^{2k} a_i V_i,
\]
where $a_i=(\text{\rm div}_\mu V_i)/2$, $1\le i\le 2k$.
Moreover, by \eqref{eq.f.s.coord}, it holds
\begin{align}
    & V_i(x)=\frac1{\sqrt{2}}
       \Bigl(\frac{\partial}{\partial u^i}\Bigr)_0,
       \quad 1\le i\le 2k,
    \label{eq.v_i.local}
    \\
    & [V_i,V_j](x)
       =-2\Bigl(\sum_{p=1}^k\delta_{i,2p-1}\delta_{j,2p}
          \Bigr)\Bigl(\frac{\partial}{\partial u^{2k+1}}
          \Bigr)_0
        +\sum_{p=1}^{2k} C_{ij}^p \Bigl(
            \frac{\partial}{\partial u^p}\Bigr)_0,
      \quad 1\le i<j\le 2k
    \label{eq.[v_i,v_j].local}
\end{align}
for some $C_{ij}^p\in\R$.
Thus we are in the situation that $d=2k+1$, $n=2k$, 
$N_0(x)=2$, and $\nu(x)=2k+2$.

We now proceed to the computation of
$\tilde{\Gamma}_2^0(0)U_t^2$.
Let 
\[
    \cG(1)=\{(i)\,|\,1\le i\le 2k\},\quad
    \cG(2)=\cG(1)\cup\{(i,j)\,|\,1\le i<j\le 2k\}.
\]
Set $\cH=\{(1),\dots,(2k),(1,2)\}$.
Then, by \eqref{eq.v_i.local} and
\eqref{eq.[v_i,v_j].local}, 
\renewcommand{\arraystretch}{1.5}
\[
    B_{\cH}(0)
    =\left(
     \begin{array}{ccc:c}
        &&& C_{12}^1 \\
        & 2^{-1/2}\mbox{\smash{\large\textrm{Id}}}_{2k} &&
        \vdots \\
        &&& C_{12}^{2k} \\ \hdashline
        0 & \cdots & 0 & -2
     \end{array}
     \right),
\]
where $\mbox{\textrm{Id}}_{2k}$ denotes the $2k$-dimensional
identity matrix. 
Hence
\[
    B_{\cH}(0)^{-1}
    =\left(
     \begin{array}{ccc:c}
        &&& 2^{-1/2} C_{12}^1 \\
        & 2^{1/2}\mbox{\smash{\large\textrm{Id}}}_{2k} &&
        \vdots \\
        &&& 2^{-1/2} C_{12}^{2k} \\ \hdashline
        0 & \cdots & 0 & -1/2
     \end{array}
     \right) \quad\text{and}\quad
   |\det B_{\cH}(0)|=2^{-k+1}.
\]
This and \eqref{eq.[v_i,v_j].local} yield
\[
    B_{\cH}(0)^{-1} V_{[i,j]}
    =\begin{pmatrix}
       2^{1/2} C_{ij}^1
          -2^{1/2} C_{12}^1\mathbf{1}_{\cG_0(2)}((i,j))
       \\
       \vdots
       \\
       2^{1/2} C_{ij}^{2k}
          -2^{1/2} C_{12}^{2k}\mathbf{1}_{\cG_0(2)}((i,j))
       \\
       \mathbf{1}_{\cG_0(2)}((i,j))
     \end{pmatrix},
     \quad 1\le i<j\le 2k,
\]
where $\cG_0(2)=\{(2i-1,2i)\,|\,1\le i\le k\}$ and
$\mathbf{1}_{\cG_0(2)}$ is the indicator function of
$\cG_0(2)$.
Hence we have
\[
    \tilde{\Gamma}_2^0(0)
    =\left(
     \begin{array}{ccc:c}
        &&&  \\
        & \mbox{\smash{\large\textrm{Id}}}_{2k} &&
        \mbox{\smash{\Large\textrm{0}}}_{2k\times k(2k-1)} \\
        &&& \\ \hdashline
        0 & \cdots & 0 & 
        \bigl(\mathbf{1}_{\cG_0(2)}((i,j))
            \bigr)_{(i,j)\in\cG(2)\setminus\cG(1)}
     \end{array}
     \right),
\]
where $0_{2k\times k(2k-1)}$ is the
$2k\times k(2k-1)$-zero matrix.
This implies that
\[
    \tilde{\Gamma}_2^0(0)U_t^2
    =\begin{pmatrix}
      w_t^1 \\ \vdots \\ w_t^{2k} \\
      \sum_{i=1}^k S_t(w^{2i-1},w^{2i})
     \end{pmatrix},
\]
where $S_t(w^{2i-1},w^{2i})$ is defined by \eqref{def.levy}.

As in Example~\ref{example.3D}, using the independence of 
$(w^{2i-1},w^{2i})$, $1\le i\le k$, we have
\begin{align*}
    {\mathbb E}[\delta_0(\tilde{\Gamma}_2^0(0)U_1^2)]
    & = \frac{1}{2\pi} \int_{\R} {\mathbb E}\Bigl[
             \delta_0(w_1^1,\dots,w_1^{2k})
             \exp\Bigl(\kyosu\lambda\sum_{i=1}^k
                 S_1(w^{2i-1},w^{2i})\Bigr)\Bigr]d\lambda
    \\
    & = \frac{1}{2\pi} \int_{\R} \prod_{i=1}^k
          {\mathbb E}\Bigl[
             \delta_0(w_1^{2i-1},w_1^{2i})
             \exp\Bigl(\kyosu\lambda
                 S_1(w^{2i-1},w^{2i})\Bigr)\Bigr]d\lambda
    \\
    & = \frac{1}{2\pi} \int_{\R} \Bigl(
           \frac1{2\pi}\,
           \frac{\lambda/2}{\sinh(\lambda/2)}\Bigr)^k
           d\lambda.
\end{align*}
To see the last identity, we have used
\eqref{fml.levy}.
Hence the leading constant in the asymptotics of $p_t(x,x)$
associated with $\Delta/2$ equals
\[
   \frac1{\rho(0)} |\det B_{\cH}(0)|^{-1}
   {\mathbb E}[\delta_0(\tilde{\Gamma}_2^0(0)U_1^2)]
   =\frac12\,\frac{1}{2\pi} \int_{\R} \Bigl(
           \frac1{2\pi}\,
           \frac{\lambda/2}{\sinh(\lambda/2)}\Bigr)^k
           d\lambda.
\]
The right-hand side is the heat kernel $p_1(0,0)$  
associated with the sub-Laplacian on the
Heisenberg group of dimension $2k+1$
(cf. \cite[Th\'eor\`eme 1]{ga}).
\end{example}

\medskip

Before providing our final example, 
we fix some notations.
Let $({\frak k}_1 \oplus {\frak k}_2, g)$ be such that
{\rm (i)}~${\frak k}_1 \oplus {\frak k}_2$ is a finite-dimensional 
graded Lie algebra (with ${\frak k}_j =\{0\}$ for $j \in {\mathbb Z} \setminus \{1,2\}$)
and $g$ is an inner product on ${\frak k}_1$.
Two such 
$({\frak k}_1 \oplus {\frak k}_2, g)$
and 
$(\hat{\frak k}_1 \oplus \hat{\frak k}_2, \hat{g})$
are said to be {\it isometrically isomorphic}
and denoted by 
$({\frak k}_1 \oplus {\frak k}_2, g)  \cong 
(\hat{\frak k}_1 \oplus \hat{\frak k}_2, \hat{g})$
if there exists an isomorphism 
$\phi\colon {\frak k}_1 \oplus {\frak k}_2 \to \hat{\frak k}_1 \oplus \hat{\frak k}_2$ 
of graded Lie algebras
whose restriction to ${\frak k}_1$ preserves the inner product.
An 
isometrical isomorphism class in this sense 
is denoted by $[({\frak k}_1 \oplus {\frak k}_2, g)]$.

Let $({\frak k}_1 \oplus {\frak k}_2, g)$ be as above
and write $n = \dim {\frak k}_1$ and $p = \dim {\frak k}_2$.
An adapted basis of this Lie algebra is defined to be a linear basis
$\{ v_1,\ldots, v_n ; z_1, \ldots, z_p\}$
such that $\{ v_1,\ldots, v_n\}$ is an orthonormal basis of $( {\frak k}_1, g)$
and 
$\{z_1, \ldots, z_p\}$ be a linear basis of ${\frak k}_2$.
Write 
\[
[v_i, v_j] = \sum_{k=1}^p C_{ij}^k  z_k
\qquad
(1\le i,j \le n).
\]
We call $\{ C_{ij}^k \}$ 
the structure constants  with respect to this adapted basis.
(Note that there are no other non-trivial Lie brackets.)
It is easy to see that
$({\frak k}_1 \oplus {\frak k}_2, g)  \cong 
(\hat{\frak k}_1 \oplus \hat{\frak k}_2, \hat{g})$
if and only if 
we can find an adapted basis of $({\frak k}_1 \oplus {\frak k}_2, g)$
and 
an adapted basis of $(\hat{\frak k}_1 \oplus \hat{\frak k}_2, \hat{g})$
whose structure constants exactly coincide.

\medskip

\begin{example}~
Let $(M, \cD, g)$ be a step-two equiregular compact sub-Riemannian manifold 
with $\dim M =n+p$ and ${\rm rank}~ \cD =n$ ($n \ge 1, ~p\ge 1$)
and let $\mu$ be Popp's measure.
In this case, the Hausdorff dimension is  $\nu =n +2p$. 
By the equiregularity,  $\cD_1 (x) \oplus ( \cD_2 (x)/ \cD_1 (x))$
has a natural structure of graded Lie algebra.
Clearly, $N_0 =2$ and we set
\[
    \cG(1)=\{(i)\,|\,1\le i\le n\},\quad
    \cG(2)=\cG(1)\cup\{(i,j)\,|\,1\le i<j\le n\}.
\]

The aim of this example is to calculate the leading constant 
$c_0 (x)$ explicitly in a probabilistic way
and
show that it depends only on 
$[(\cD_1 (x) \oplus ( \cD_2 (x)/ \cD_1 (x)), g_x) ]$.
(More precisely,  if  $(\hat{M}, \hat{\cD}, \hat{g})$ is another such 
sub-Riemannian manifold  and 
\begin{equation}\label{eq.isom}
(\cD_1 (x) \oplus ( \cD_2 (x)/ \cD_1 (x)), g_x)
\cong
(\hat{\cD}_1 (\hat{x}) \oplus ( \hat{\cD}_2 (\hat{x})/ \hat{\cD}_1 (\hat{x})), \hat{g}_{\hat{x}})
\end{equation}
holds for $x \in M$ and $\hat{x}\in \hat{M}$,
then $c_0 (x)= c_0 (\hat{x})$ holds.)

To this end  we use (a very special case of)
Bianchini-Stefani's adapted chart.
As was already demonstrated in \cite{pa, ha}, this chart looks
quite useful for short time asymptotic problems on sub-Riemannian manifolds.
Take $x \in M$ arbitrarily.
Then, 
by \cite[Corollary 3.1]{bs},
there exists a local coordinate chart $(u^1, \ldots, u^{n+p})$
around $x$
such that 
$x$ corresponds to $0 \in {\mathbb R}^{n+p}$ and 
$\cD_1 (x)$ equals the linear span of 
$\{( \frac{\partial }{\partial u^1})_0, \ldots,  (\frac{\partial }{\partial u^n})_0 \}$.
Note that this equality holds only at $x$
and such a chart is obviously not unique.

Take a local frame $\{V_1, \ldots, V_n, Z_1, \ldots, Z_p\}$ 
of $TM$ 
around $x$ such that $\{V_1, \ldots, V_n\}$ forms a local orthonormal frame
of $\cD =\cD_1$.
Such a local frame is called an adapted frame.
As usual the structure constants $\{ C_{ij}^k \}$ is defined by  
\[
[V_i, V_j] (x)= \sum_{k=1}^p C_{ij}^k  Z_k (x)
\quad
\mod
\cD_1(x)
\qquad
(1\le i,j \le n).
\]
The rank of a $p \times n(n-1)/2$ matrix 
$( C^k_{ij} )_{ 1 \le k \le p, (i,j) \in \cG(2) \setminus \cG(1) }$
is  $p$ due to the H\"ormander condition at $x$.
We will denote this matrix by $C$ for simplicity.

Changing the coordinates of 
$(u^1, \ldots, u^{n})$ and  $(u^{n+1}, \ldots, u^{n+p})$
by linear maps, 
we may additionally assume that 
$V_i (x) = (\frac{\partial }{\partial u^i})_0$ for $1 \le i \le n$ and 
$Z_j (x) = (\frac{\partial }{\partial u^{j+n}} )_0$
modulo $\cD_1(x)$ for $1 \le j \le p$. 
Then, it is obvious that
\[
B_{2} (0) 
=
\left(
     \begin{array}{c:c}
      \textrm{Id}_{n} & *
     \\ \hdashline
      \textrm{0}_{p \times n} & C
                  \end{array}
     \right).
     \]
Choose $(i_1, j_1), \ldots, (i_p, j_p) \in \cG(2) \setminus \cG(1)$
so that 
$( C^k_{i_a j_a} )_{ 1 \le k, a \le p}$ is an 
invertible $p \times p$ matrix
and we set $\cH =\cG (1) \cup \{ (i_a, j_a) \mid 1 \le  a \le p\}$.
Then, it is easy to see that
\[
B_{\cH} (0) \tilde{\Gamma}_2^0 (0)
=
\left(
     \begin{array}{c:c}
      \textrm{Id}_{n} &  \textrm{0}_{p \times n(n-1)/2}
           \\ \hdashline
      \textrm{0}_{p \times n} & C
                  \end{array}
     \right).
     \]

According to \cite{br}, Popp's measure can be computed from 
the structure constants for 
the local adapted frame as follows.
Set $\cC^{kl} = \langle C^{k}_{\bullet\star}, 
C^{l}_{\bullet\star},\rangle_{HS}$, $1 \le k,l \le p$,
where the right-hand side stands for the Hilbert-Schmidt
inner product for $n \times n$-matrices.
Then, we have $\mu (d \theta) = \rho (u) d u^1 \cdots d u^{n+p}$
with
$\rho (0) = \{\det (\cC^{kl})_{1 \le k,l \le p} \}^{-1/2}$.

Combining these all, we see that 
\begin{align*}
 c_0 (x)
 &=  \frac1{\rho(0)}  |\det B_{\cH}(0)|^{-1}
   {\mathbb E}[\delta_0(\tilde{\Gamma}_2^0(0)U_1^2)]
  \\
   &=
   \sqrt{\det (\cC^{kl})} 
     {\mathbb E}[\delta_0(  B_{\cH} (0) \tilde{\Gamma}_2^0 (0)  U_1^2)]
     \\
   &=
   \sqrt{\det (\cC^{kl})} 
   {\mathbb E} \Bigl[\delta_0( w^1_1, \ldots, w^n_1 ) \cdot 
   \delta_0 \Bigl(  \sum_{i<j} C^1_{ij} S_1^{ij}, 
   \ldots,  \sum_{i<j} C^p_{ij} S_1^{ij} \Bigr) \Bigr],
              \end{align*}
where 
we wrote $S_t^{ij}$ for L\'evy's stochastic area
$S_t(w^{i},w^{j})$ defined by (\ref{def.levy}) for simplicity.
The generalized expectation on the right-hand side 
is computed in Appendix. 
Thus, we obtain 
\[
 c_0 (x)
=
  \sqrt{\det (\cC^{kl})} 
  \frac1{(2\pi )^{(n/2)+p}} \int_{\R^p}
       \biggl[\det\biggl(
         \frac{\sinh(\kyosu (\zeta\cdot C)/2)}{
          \kyosu(\zeta\cdot C)/2}\biggr)\biggr]^{-1/2} d\zeta,
             \]
where for $\zeta=(\zeta^1,\dots,\zeta^p)\in\R^p$,
 $(\zeta\cdot C)$ is the $n\times n$ skew symmetric
matrix defined by
\[
    (\zeta\cdot C)
    = \Bigl(\sum_{k=1}^p  \zeta^k 
     C_{ij}^k\Bigr)_{1\le i,j\le n}.
\]
Note that $c_0 (x)$ depends only on the structure constants.

Finally, let us assume that $\hat{x} \in \hat{M}$ satisfies (\ref{eq.isom}).
Then, we can find a local  adapted frame 
$\{\hat{V}_1, \ldots, \hat{V}_n, \hat{Z}_1, \ldots, \hat{Z}_p\}$ 
around $\hat{x}$ 
which 
yields the same structure constants $(C^k_{ij})$.
By doing the same computation again, we see $c_0 (x)= c_0 (\hat{x})$. 
\end{example}


\appendix

\section{On step-two nilpotent Lie groups}

In \cite{ga}, Gaveau obtained exlicit expressions for the 
heat kernel for the Heisenberg groups and the free
nilpotent Lie groups of step-two.
The heat kernels for all nilpotent Lie groups of step-two were
obtained by Cygan (\cite{cy}).
They used an analytic method.
In this section, we recover such expressions by using a
probabilistic method.
Indeed, we shall obtain the expressions 
by using
an explicit expression of a stochastic oscillatory integral
with a quadratic Wiener functional as its phase function
(cf. \cite{tan2}). 
L\'evy's stochastic area defined in \eqref{def.levy} is a
typical example of such a quadratic Wiener functional.

We start this section with a preliminary observation on
linear combinations of 
L\'evy's stochastic areas.
For $t\ge0$, $x\in\R^n$, $1\le i,j\le n$, and an 
$n\times n$ skew symmetric matrix 
$\Xi=(\Xi^{ij})_{1\le i,j\le n}$, set
\[
    S_t^{ij}(x)=\int_0^t \{(x+w_s^i)\circ dw_s^j
       -(x+w_s^u)\circ dw_s^i\}
    \quad\text{and}\quad
    S_t(\Xi;x)
    =\frac12 \sum_{1\le i<j\le n} \Xi^{ij} 
     S_t^{ij}(x).
\]

Our first goal of this section is revisiting the following
expression (\cite{cy,ga}) by using the computation of
oscillatory integrals associated with quadratic Wiener
functionals in \cite{tan2}.
This is a generalization of the famous formula \eqref{fml.levy}
 for L\'evy's stochastic area.

\begin{theorem}\label{thm.h.kernel}
It holds that
\begin{multline}\label{t.h.kernel.1}
   {\mathbb E}[e^{\kyosu S_t(\Xi;x))} \delta_y(x+w_t)]
   =\frac1{(2\pi t)^{n/2}} 
     \biggl[\det \Bigl(
           \frac{\sinh(\kyosu t \Xi/2)}{
                \kyosu t \Xi/2}\Bigr)\biggr]^{-1/2}
   \\
    \times \exp\biggl(-\frac{\kyosu}2\langle\Xi x,y
          \rangle_{\R^n}
     -\frac1{2t}\langle T(t;\Xi)^{-1}(y-x),(y-x)
          \rangle_{\R^n}\biggr), \quad y\in\R^n,
\end{multline}
where, for $n\times n$ matrix $B$, 
\[
    \frac{\sinh(B)}{B}=\sum_{k=1}^\infty 
          \frac1{(2k-1)!}B^{2k-2}, \quad
     \cosh(B)=\sum_{k=0}^\infty 
          \frac1{(2k)!}B^{2k},
\]
and
\[
     T(t;\Xi)=
       \frac{\sinh(\kyosu t\Xi/2)}{
                            \kyosu t\Xi/2}
          \bigl(\cosh(\kyosu t\Xi/2)\bigr)^{-1}.
\]
\end{theorem}

\begin{remark}\label{rem.nonzero}
For an $n\times n$ skew symmetric matrix $B$, take
$\lambda_1,\dots,\lambda_k\in \R\setminus\{0\}$ such
that 
$\pm\kyosu\lambda_1,\dots,\pm\kyosu \lambda_k,
 \underbrace{0,\dots,0}_{n-2k}$ are its eigenvalues.
Then 
\begin{align*}
   \det\Bigl(\frac{\sinh(\kyosu B)}{\kyosu B}\Bigr)
   & =\prod_{i=1}^k \Bigl(\frac{\sinh\lambda_i}{\lambda_i}
     \Bigr)^2\ne0,
   \\
   \det(\cosh(\kyosu B))
   & =\prod_{i=1}^k (\cosh\lambda_i)^2
   \ne0.
\end{align*}
Thus, the reciprocal number and the inverse matrix appearing
in \eqref{t.h.kernel.1} are both well-defined.
\end{remark}

\begin{proof}
If $i\ne j$, then 
$w_s^i\circ dw_s^j=w_s^i dw_s^j$.
Since $\Xi$ is skew symmetric,
\[
    S_t(\Xi;x)
    =\frac12\int_0^t \langle  (-\Xi) (x+w_s),
          dw_s\rangle_{\R^n}
    =\frac12 \int_0^t \langle (-\Xi) w_s,
       dw_s\rangle_{\R^n}
           -\frac12 \langle\Xi x,w_t\rangle_{\R^n}.
\]
By the skew symmetry of $\Xi$ again, we have
\begin{equation}\label{eq.pr.19}
    {\mathbb E}[e^{\kyosu S_t(\Xi;x)}\delta_y(x+w_t)]
    =e^{-\kyosu\langle\Xi x,y\rangle/2}
       \mathbb{E}\biggl[\exp\biggl(
        \frac{\kyosu}2 \int_0^t 
           \langle  (-\Xi) w_s,
              dw_s\rangle \biggr)
      \delta_{y-x}(w_t)\biggr].
\end{equation}
Thus it suffices to compute 
${\mathbb E}[e^{\kyosu S_t(\Xi;0)}\delta_y(w_t)]$.

Applying \cite[Corollary 1.1 and Example 4.2]{tan2},
we obtain
\begin{equation}\label{eq.pr.20}
    {\mathbb E}[e^{\kyosu S_t(\Xi;0)}\delta_y(w_t)]
    =\frac1{\sqrt{\det A(0,t;\Xi)}}
     \frac1{(2\pi)^{n/2}\sqrt{C(t;\Xi)}}
     \,\exp\biggl(-\frac12\langle C(t;\Xi)^{-1}y,y
     \rangle_{\R^n}\biggr)
\end{equation}
where
\begin{align*}
    & A(s,t;\Xi)=\frac12\{I+\exp( -
             \kyosu(s-t)\Xi)\},
    \\
    & C(t;\Xi)=\int_0^t (A(0,s;\Xi)^{-1})^{*}
        A(0,s;\Xi)^{-1} ds,
\end{align*}
and, for $n\times n$-matrix $B$, 
$\exp(B)=\sum_{k=0}^\infty \frac1{k!}B^k$ and 
$B^*$ is the transposed matrix of $B$.
It should be emphasized that the superscript ${}^*$
indicates just being transposed and no complex conjugate are
taken even if $B$ is a complex matrix.
We shall compute $A(s,t;\Xi)$ and $C(t;\Xi)$.

First rewrite as
\begin{equation}\label{eq.A(t;xi)}
    A(s,t;\Xi)
    =\cosh\Bigl(\frac{\kyosu}2 (s-t)\Xi\Bigr)
     \exp\Bigl( -
       \frac{\kyosu}2 (s-t)\Xi
       \Bigr).
\end{equation}
Since $\Xi$ is skew symmetric, 
\[
    \det\biggl(\exp\Bigl(
       -  \frac{\kyosu}2 (s-t)\Xi\Bigr)\biggr)=1.
\]
Thus we have
\begin{equation}\label{eq.pr.21}
    \det A(0,t;\Xi)
    =\det\Bigl(\cosh\Bigl(\frac{\kyosu t}2 \Xi
      \Bigr)\Bigr).
\end{equation}

Next, due to the skew symmetry of $\Xi$ again, by
\eqref{eq.A(t;xi)}, we have 
\[
    A(s,t;\Xi)^{*}
    =\exp\Bigl(\frac{\kyosu}2
         (s-t)\Xi\Bigr)
     \cosh\Bigl(\frac{\kyosu}2 (s-t)\Xi\Bigr).
\]
In conjunction with \eqref{eq.A(t;xi)} again, this implies
\[
    A(s,t;\Xi)A(s,t;\Xi)^{*}
    =\Bigl(\cosh\Bigr(\frac{\kyosu}2 (s-t)\Xi
      \Bigr)\Bigr)^2.
\]
Hence
\[
    (A(s,t;\Xi)^{-1})^{*} A(s,t;\Xi)^{-1}
    =\Bigl(\cosh\Bigr(\frac{\kyosu}2 (s-t)\Xi
      \Bigr)\Bigr)^{-2}.
\]
Recall that 
$\sinh(B)=\frac12(\exp(B)-\exp(-B))$ and 
\[
    \frac{d}{ds}  s\frac{\sinh(sB)}{sB}
     \bigl(\cosh(sB)\bigr)^{-1}
    =\bigl(\cosh(sB)\bigr)^{-2}.
\]
Plugging this into the definition of $C(t;\Xi)$, we obtain
\begin{equation}\label{eq.pr.22}
    C(t;\Xi)
    =t \frac{\sinh(\kyosu t\Xi/2)}{\kyosu t\Xi/2}
    \Bigl(\cosh\Bigr(\frac{\kyosu}2 t\Xi\Bigr)^{-1}.
\end{equation}

Plugging \eqref{eq.pr.21} and \eqref{eq.pr.22} into
\eqref{eq.pr.20}, we obtain
\[
    {\mathbb E}[e^{\kyosu S_t(\Xi;0)}\delta_y(w_t)]
    = \frac1{(2\pi t)^{n/2}} 
     \biggl[\det \Bigl(
           \frac{\sinh(\kyosu t \Xi/2)}{
                \kyosu t \Xi/2}\Bigr)\biggr]^{-1/2}
    \exp\biggl(
     -\frac12\langle C(t;\Xi)^{-1}y,y
          \rangle_{\R^n}\biggr).
\]
Combined with \eqref{eq.pr.19}, this implies the desired
expression \eqref{t.h.kernel.1}.
\end{proof}

\begin{remark}
Given 
$\Theta=(\Theta_1,\dots,\Theta_d)\in
 C^\infty(\mathbb{R}^d;\mathbb{R}^d)$
whose 
derivatives of all order are at most polynomial
growth, the Schr\"odinger operator $L$ with 
the vector potential $\Theta$ is given by
\[
    L=-\frac12\sum_{\alpha=1}^d
    \biggl(\frac{\partial}{\partial x^\alpha}+\kyosu
    \Theta_\alpha \biggr)^2.
\]
The heat kernel $q_t(x,y)$ associated with $L$ 
possesses a probabilistic expression as follows (for
example, see \cite[Theorem 5.5.7]{mt}).
\[
    q_t(x,y)=\mathbb{E}[e(t,x)\delta_y(x+w_t)],
\]
where $e(t,x)$ is given by
\[
    e(t,x)=\exp\biggl(\kyosu \sum_{i=1}^n \int_0^t
     \Theta_\alpha(x+w_s)\circ dw_s^i
     \biggr).
\]
If $\Theta(x)= - \frac12\Xi x$ for 
$x\in \mathbb{R}^n$, then  
$e(t,x)=\exp(\kyosu S_t(\Xi;x))$ and hence
the right-hand side of \eqref{t.h.kernel.1} gives an
explicit expression of $q_t(x,y)$. 
\end{remark}

We now proceed to step-two nilpotent Lie groups.
For this purpose, let $G$ be a $(n+p)$-dimensional connected
and simply connected step-two nilpotent Lie group with the
Lie algebra ${\mathfrak g}$, where $p$ is the dimension of
$[{\mathfrak g},{\mathfrak g}]$.
Using the diffeomorphism $\exp:{\mathfrak g}\to G$ and
suitable bases of $\{Z_k\}_{k=1}^p$ and $\{X_i\}_{i=1}^n$ of 
$[{\mathfrak g},{\mathfrak g}]$ and its complement,
respectively,
we introduce a global coordinate $(x,z)\in\R^{n+p}$ on $G$;
for $g\in G$, $g=\exp\Bigl(\sum_{i=1}^n x^i X_i+
   \sum_{k=1}^p z^k Z_k\Bigr)$, 
where $x=(x^1,\dots,x^n)\in\R^n$ and
$z=(z^1,\dots,z^p)\in\R^p$.
In terms of this coordinate, the product
$\times$ on $G$ is given
by 
\begin{equation}\label{eq.*.on.G}
    (x,z)\times(u,v)
    =\Bigl(x+u,z+v+\frac12\sum_{i,j=1}^n
               x_iu_j C_{ij}\Bigr),
\end{equation}
where 
\[
    [X_i,X_j]=\sum_{k=1}^p C_{ij}^k Z_k
    \quad\text{and}\quad
    C_{ij}=\begin{pmatrix} C_{ij}^1 \\ \vdots \\
      C_{ij}^p\end{pmatrix}\in\R^p.
\]

Let $\widetilde{X}_i$ and $\widetilde{Z}_k$ 
be the left-invariant vector fields
associated with $X_i$ and $Z_k$, 
$1\le i\le n$ and $1\le k\le p$, respectively.
Set ${\mathbf e}_i=(\delta_{ij})_{1\le j\le n}\in\R^n$ 
and
$\hat{\mathbf e}_k=(\delta_{ik})_{1\le i\le p}\in\R^p$.
Since $X_i=\frac{d}{dt}\bigr|_{t=0}(t{\mathbf e}_i,0)$ 
and 
$Z_k=\frac{d}{dt}\bigr|_{t=0}(0,t\hat{\mathbf e}_k)$,
we have
\begin{align*}
    \widetilde{X}_i(x,z)
    & =\frac{d}{dt}\Bigr|_{t=0} (x,z)
      \times (t{\mathbf e}_i,0)
    =\Bigl(\frac{\partial}{\partial x^i}\Bigr)_x
     +\frac12 \sum_{j=1}^n \sum_{k=1}^p x_j C_{ji}^k
      \Bigl(\frac{\partial}{\partial z^k}\Bigr)_z,
    \\
    \widetilde{Z}_k(x,z)
    & =\frac{d}{dt}\Bigr|_{t=0} (x,z)
      \times(0,t\hat{\mathbf e}_k)
    =\Bigl(\frac{\partial}{\partial z^k}\Bigr)_z.
\end{align*}
This implies 
$[\widetilde{X}_i,\widetilde{X}_j]=\sum_{k=1}^p C_{ij}^k
 \widetilde{Z}_k$, and hence 
$\widetilde{X}_1(x),\dots,\widetilde{X}_n(x),
 \widetilde{Z}_1(x),\dots,\widetilde{Z}_p(x)$ spans
$T_xG$ for every $x\in G$.
In particular, $\widetilde{X}_1,\dots,\widetilde{X}_n$
satisfies the equiregular H\"ormander condition at every
$x\in G$.
Then the heat equation associated with the second order
differential operator
\[
   \cL=\frac12\sum_{i=1}^n \widetilde{X}_i^2
\]
possesses the heat kernel $p_t((x_0,y_0),(x,z))$ with
respect to the Lebesgue measure. 
Note that the Lebesgue measure is a Haar measure on $G$,
because, by \eqref{eq.*.on.G}, the Jacobian determinant of
the left translation is equal to $1$.
Moreover, by \cite{br}, it coincides with 
Popp's measure multiplied by 
$\Bigl(\det\Bigl(\Bigl(\sum_{i,j=1}^n C_{ij}^k
    C_{ij}^\ell\Bigr)_{1\le k,\ell\le p}\Bigr)\Bigr)^{1/2}$.

The diffusion process generated by $\cL$ is 
\[
    \Bigl(\Bigl(
     x_0+w_t,z_0+\sum_{i<j} C_{ij}S_t^{ij}(x_0)
     \Bigr)\Bigr)_{t\ge0}.
\]
Due to the H\"ormander condition, with the help of
generalized Wiener functional, the heat kernel is
represented as  
\[
    p_t((x_0,z_0),(x,z))
    ={\mathbb E}\Bigl[\delta_{(x,z)}\Bigl(
      x_0+w_t,z_0+\sum_{i<j} C_{ij}S_t^{ij}(x_0)\Bigr)\Bigr].
\]
By the left invariance of $\widetilde{X}_i$, 
$1\le i\le n$, it holds 
\[
    \Bigl(x_0+w_t,z_0+\sum_{i<j} C_{ij}S_t^{ij}(x_0)\Bigr)
    =(x_0,z_0) \times
     \Bigl(w_t,\sum_{i<j} C_{ij}S_t^{ij}(x_0)\Bigr).
\]
Hence
\[
    p_t((x_0,z_0),(x,z))
    =p_t((0,0),(x_0,z_0)^{-1} \times
       (x,z))
\]
Thus, in what follows, we assume $(x_0,z_0)=(0,0)$.

Using the Fourier transform of the Dirac measure, we have
\begin{equation}\label{eq.append.21}
    p_t((0,0),(x,z))
    =\frac1{(2\pi)^p} \int_{\R^p} 
       e^{-\kyosu\langle\zeta,z\rangle}
       {\mathbb E}\bigl[e^{\kyosu S_t((\zeta\cdot C);0)}
          \delta_x(w_t)\bigr] d\zeta,
\end{equation}
where 
for $\zeta=(\zeta^1,\dots,\zeta^p)\in\R^p$,
 $(\zeta\cdot C)$ is the $n\times n$ skew symmetric
matrix
\[
    (\zeta\cdot C)
    = \Bigl(\sum_{k=1}^p 
      \zeta^k C_{ij}^k\Bigr)_{1\le i,j\le n}    
\]
If we define
\begin{align*}
    \widehat{p}_t((x,z))
    & =\frac1{(2\pi t)^{(n/2)+p}} \int_{\R^p}
       \biggl[\det\biggl(
         \frac{\sinh(\kyosu (\eta\cdot C)/2)}{
          \kyosu(\eta\cdot C)/2}\biggr)\biggr]^{-1/2}
    \\
    & \qquad
      \times  \exp\biggl(-\frac1{t}\Bigl\{
            \langle\eta,z\rangle_{\R^p}
        +\frac12 \langle T(1;(\eta\cdot C))^{-1}x,x
                       \rangle_{\R^n}\Bigr\}\biggr)d\eta,
\end{align*}
then plugging Theorem~\ref{thm.h.kernel} into
\eqref{eq.append.21} and using the change variable
$\eta=t\zeta$, we obtain 
\[
    p_t((0,0),(x,z))=\widehat{p}_t((x,z)).
\]

Summing up, we arrive at the following expression of the
heat kernel, which was also shown in an analytical way in
\cite{cy,bgg}. 

\begin{theorem}\label{thm.h.k.nilpotent}
The heat kernel associated with $\cL$ has the form
\begin{equation}\label{t.h.k.nilpotent.1}
    p_t((x_0,z_0),(x,z))=\widehat{p}_t((x_0,z_0)^{-1}
        \times (x,z)).
\end{equation}
\end{theorem}

\bigskip

\noindent
{\bf Acknowledgement}~The first-named author is partially
supported by JSPS KAKENHI Grant Number 15K04922, and the
second-named author is partially supported by JSPS KAKENHI
Grant Number 15K04931.

%

\vspace{10mm}

\begin{flushleft}
\begin{tabular}{ll}
Yuzuru \textsc{Inahama}
\\
Graduate School of Mathematics,   Kyushu University,
\\
Motooka 744, Nishi-ku, Fukuoka 819-0395, JAPAN.
\\
Email: {\tt inahama@math.kyushu-u.ac.jp}
\end{tabular}
\end{flushleft}

\bigskip

\begin{flushleft}
\begin{tabular}{ll}
Setsuo \textsc{Taniguchi}
\\
Faculty of Arts and Science,   Kyushu University,
\\
Motooka 744, Nishi-ku, Fukuoka 819-0395, JAPAN.
\\
Email: {\tt se2otngc@artsci.kyushu-u.ac.jp}
\end{tabular}
\end{flushleft}

\end{document}